\documentclass[10pt]{article}

\usepackage[utf8]{inputenc}
\usepackage{amsmath,amsfonts,amssymb}
\usepackage{amsthm}
\usepackage[a4paper]{geometry}
\usepackage{color}
\usepackage{transparent}
\usepackage[dvipsnames]{xcolor}
\usepackage{graphicx}
\usepackage{subcaption}
\usepackage{tabularx,float}
\usepackage[hidelinks]{hyperref}
\usepackage{comment} 
\usepackage{tabularx} 
\usepackage{enumerate,mathrsfs}
\usepackage{mathtools}
\usepackage{enumitem}
\usepackage{orcidlink}
\usepackage{titling}
\usepackage[misc]{ifsym}
\usepackage{xargs}    
\usepackage[colorinlistoftodos,prependcaption,textsize=small]{todonotes}
\usepackage[titletoc,title]{appendix}
\usepackage{nomencl}
\usepackage{soul}
\usepackage{caption}
\usepackage{bbm}
\usepackage[subrefformat=parens,labelformat=parens]{subcaption}

\makenomenclature

\newcommand{\overbar}[1]{\mkern 1.5mu\overline{\mkern-1.5mu#1\mkern-1.5mu}\mkern 1.5mu}

\numberwithin{equation}{section}

\newcommand*{\dd}{\mathrm{d}}
\newcommand*{\ve}{\textup{vec}}
\newcommand*{\vem}{\textup{vecm}}
\newcommand*{\stack}{\textup{stk}}

\newtheorem{theorem}{Theorem}[section]

\newtheorem{proposition}[theorem]{Proposition}
\newtheorem{corollary}[theorem]{Corollary}
\newtheorem{lemma}[theorem]{Lemma}

\theoremstyle{definition}
\newtheorem{definition}[theorem]{Definition}
\newtheorem{remark}[theorem]{Remark}
\newtheorem{example}[theorem]{Example}

\counterwithin{figure}{section}

\newcommand{\eps}{\varepsilon}

\newcommand\norm[1]{\left\lVert#1\right\rVert}
\newcommand\normk[1]{\lVert#1\lVert}
\newcommand\abs[1]{\left\lvert#1\right\rvert}
\newcommand{\R}{\mathbb{R}}

\newcommand\blfootnote[1]{%
	\begingroup
	\renewcommand\thefootnote{}\footnote{#1}%
	\addtocounter{footnote}{-1}%
	\endgroup
}

\usepackage{fancyhdr}
\title{\textbf{Analysis of the Geometric Structure of Neural Networks and Neural ODEs via Morse Functions \\[3mm]}}

\author{Christian Kuehn \orcidlink{0000-0002-7063-6173}$^{1,2,3}$ \& Sara-Viola Kuntz \orcidlink{0009-0000-4611-9742}$^{1,2,3}$ \\[3mm]
}

\date{
	\small{$^1$\textit{Technical University of Munich, School of Computation, Information and Technology, \\ Department of Mathematics, Boltzmannstraße 3, 85748 Garching, Germany} \\
		$^2$\textit{Munich Data Science Institute (MDSI), Garching, Germany } \\
		$^3$\textit{Munich Center for Machine Learning (MCML), München, Germany }}\\[10mm]
	\large{February 10, 2026}
}

\begin{document}
	
	\maketitle

	\begin{abstract} 
		Besides classical feed-forward neural networks such as multilayer perceptrons, also neural ordinary differential equations (neural ODEs) have gained particular interest in recent years. Neural ODEs can be interpreted as an infinite depth limit of feed-forward or residual neural networks. We study the input-output dynamics of finite and infinite depth neural networks with scalar output. In the finite depth case, the input is a state associated with a finite number of nodes, which maps under multiple non-linear transformations to the state of one output node. In analogy, a neural ODE maps an affine linear transformation of the input to an affine linear transformation of its time-$T$ map. We show that, depending on the specific structure of the network, the input-output map has different properties regarding the existence and regularity of critical points. These properties can be characterized via Morse functions, which are scalar functions where every critical point is non-degenerate. We prove that critical points cannot exist if the dimension of the hidden layer is monotonically decreasing or the dimension of the phase space is smaller than or equal to the input dimension. In the case that critical points exist, we classify their regularity depending on the specific architecture of the network. We show that except for a Lebesgue measure zero set in the weight space, each critical point is non-degenerate if for finite depth neural networks the underlying graph has no bottleneck, and if for neural ODEs, the affine linear transformations used have full rank. For each type of architecture, the proven properties are comparable in the finite and infinite depth cases. The established theorems allow us to formulate results on universal embedding and universal approximation, i.e., on the exact and approximate representation of maps by neural networks and neural ODEs. Our dynamical systems viewpoint on the geometric structure of the input-output map provides a fundamental understanding of why certain architectures perform better than others.  
	\end{abstract}

	\vspace{3mm}
	
	\noindent {\small \textbf{Keywords: neural networks, neural ODEs, Morse functions, universal embedding, universal approximation}
		
		\vspace{4mm}
		
		
		\blfootnote{\textcolor{white}{.}\\[-2mm]
			\hspace{-5.4mm}\Letter \; ckuehn@ma.tum.de (Christian Kuehn)  \\[0.7mm]
			\Letter\; saraviola.kuntz@ma.tum.de (Sara-Viola Kuntz, corresponding author) \\[2mm]
			This version of the article has been accepted for publication, after peer review but is not the Version of Record and does not reflect post-acceptance improvements, or any corrections. The Version of Record is published in the Journal \textit{Advances in Computational Mathematics}, and is available online at  \url{https://doi.org/10.1007/s10444-025-10273-5}. \\}
		
		\newpage
		
		\tableofcontents
		
		\vspace{3mm}
		
		\section{Introduction} 
		\label{sec:introduction}
		
		Neural Networks are powerful computational models inspired by the functionality of the human brain. A classical neural network consists of interconnected neurons, represented by nodes and weighted edges of a graph. Through the learning process, the weighted connections between the nodes are adapted, such that the output of the neural network better predicts the data \cite{Aggarwal2018}. Often, the nodes of a neural network are organized in layers, and the information is fed forward from layer to layer. In the easiest case, a feed-forward neural network is a perceptron, studied already by Rosenblatt in 1957~\cite{Rosenblatt1957}. \medskip
		
		A classical feed-forward neural network is structured in layers $h_l \in \mathbb{R}^{n_l}$, $l \in \{0,1,\ldots,L\}$ with width $n_l$, where $h_0$ is called the input layer and $h_L$ the output layer. The layers $h_1,h_2,\ldots,h_{L-1}$ in between are called hidden layers. The neural network is a map $\Phi: h_0 \mapsto h_L$. In this work, we study multilayer perceptrons (MLPs), which belong to the most famous feed-forward neural network architectures, and which are defined by the iterative update rule 
		\begin{equation} \label{eq:intro_updaterule}
			h_{l} = f_l(h_{l-1},\theta_l) \coloneqq \widetilde{W}_l\sigma_l(W_lh_{l-1}+b_l)+\tilde{b}_l = \widetilde{W}_l\sigma_l(a_l)+\tilde{b}_l
		\end{equation}
		for $l \in \{1,\ldots,L\}$, where $W_l$, $\widetilde{W}_l$ are weight matrices and $b_l$, $\tilde{b}_l$ are biases of appropriate dimensions. We abbreviate all parameters by $\theta_l \coloneqq (W_l,\widetilde{W}_l,b_l,\tilde{b}_l)$. The function $\sigma_l$ is a component-wise applied non-linear activation function such as $\tanh$, soft-plus, sigmoid, or (normal, leaky, or parametric) $\textup{ReLU}$. The update rule~\eqref{eq:intro_updaterule} includes both the case of an outer nonlinearity if $\widetilde{W}_l = \text{Id}$, $\tilde{b}_l = 0$ and the case of an inner nonlinearity if $W_l = \text{Id}$ and $b_l = 0$. 
		
		Besides MLPs with finite depth $L< \infty$, we aim to study neural ODEs, which can be interpreted as an infinite network limit of residual neural networks (ResNets) \cite{Chen2018, He2016, Kuehn2023, Weinan2017}. ResNets are feed-forward neural networks with the specific property that all layers have the same width $m = n_l$ and the function $f_l$ of the update rule~\eqref{eq:intro_updaterule} is of the form
		\begin{equation}\label{eq:intro_ResNet}
			h_{l} = f_l(h_{l-1},\theta_l) = h_{l-1} + f_{\text{ResNet},l}(h_{l-1},\theta_l).
		\end{equation} 
		If $f_{\textup{ResNet},l}(\cdot,\cdot) =  f_\textup{ResNet}(\cdot,\cdot)$ for all $l \in \{1,\ldots,L\}$, the ResNet update rule~\eqref{eq:intro_ResNet} can be obtained as an Euler discretization of the ordinary differential equation (ODE) 
		\begin{equation} \label{eq:intro_ODE_1}
			\frac{\dd h}{\dd t} = f_\textup{ODE}(h(t),\theta(t)), \qquad h(0) = h_0,
		\end{equation}
		on the time interval $ [0,T]$ with step size $\delta = T/L$ and $f_\textup{ODE}(\cdot,\cdot) = \frac{1}{\delta} f_\textup{ResNet}(\cdot,\cdot)$. Hereby, the function $h$ can be interpreted as the hidden states, and the function $\theta$ as the weights. The output of the neural ODE, which corresponds to the output layer $h_L$ of the neural network, is the time-$T$ map $h_{h_0}(T)$ (cf.~\cite{Guckenheimer2002}) of~\eqref{eq:intro_ODE_1}. The smaller the step size~$\delta$, i.e., the larger the depth~$L$ of the neural network, the better the Euler approximation becomes for fixed $T$. In this work, we drop for most of our results the dependency on the parameter function~$\theta$ and study neural ODEs based on general non-autonomous initial value problems of the form 
		\begin{equation} \label{eq:intro_ODE_2}
			\frac{\dd h}{\dd t} = \tilde{f}(t,h(t)), \qquad h(0) = h_0
		\end{equation}
		with $\tilde{f}:\mathbb{R}\times\mathbb{R}^m\rightarrow \mathbb{R}^m$. To be flexible with respect to the input and output dimensions of the neural ODE, we study architectures defined by $\Phi: x \mapsto \widetilde{\lambda}(h_{\lambda(x)}(T))$, where $\lambda$ and $\widetilde{\lambda}$ are affine linear transformations applied before and after the initial value problem. The initial value problem~\eqref{eq:intro_ODE_2} is a classical ODE studied in many contexts, the main difference in machine learning is that the focus lies on input-output relations over finite time-scales. \medskip
		
		An important and common research area for feed-forward neural networks and neural ODEs is expressivity. For MLPs and ResNets, various universal approximation theorems exist~\cite{Hornik1989, Kidger2022, Lin2018, Pinkus1999, Schaefer2006}, stating that by increasing the width and depth or the number of parameters, any continuous function can be approximated arbitrarily well. If the width of the network is not larger than the input dimension, the approximation capability is restricted, as shown in \cite{Johnson2018}. For neural ODEs, initial explorations in the topic of universal approximation have been made in \cite{Kidger2022, Zhang2020a}. In \cite{Kuehn2023}, different neural ODE architectures are systematically studied with respect to the property of universal embedding. The restriction to this exact representation problem allows the mathematical arguments to gain clarity. As for MLPs, the general observation is that, also for neural ODEs, the expressivity mainly depends on the dimension of the phase space. The main distinction is whether the dimension of the initial value problem~\eqref{eq:intro_ODE_2} is smaller or larger compared to the dimension of the input data. The neural ODE architectures are then called non-augmented or augmented, respectively.  
		
		As first observed in \cite{Kuehn2023}, the universal embedding features of neural ODEs are related to the property of being a Morse function. A Morse function is a scalar function, where every critical point is non-degenerate~\cite{Hirsch1976, Morse1934}, i.e., the determinant of the Hessian matrix at critical points is non-zero. For scalar MLPs with smooth and monotonically increasing activation functions, it is claimed in \cite{Kratsios2021}, that MLPs without bottlenecks are almost surely Morse functions. A neural network has a bottleneck if three layers exist, where the middle layer has a strictly smaller dimension than the first and the third layer. We build upon the initial explorations of~\cite{Kratsios2021} and \cite{Kuehn2023} as a starting point to systematically study the property of being a Morse function depending on the architecture for both MLPs and neural ODEs. To that purpose, we define in both cases what a non-augmented, an augmented, and an architecture with a bottleneck is. Depending on the type of architecture, we prove comparable results about the input-output map of MLPs and neural ODEs. For non-augmented architectures, we can prove in both cases via rank arguments, an explicit calculation of the network gradient, and the usage of linear variational equations in the continuous case, that no critical points can exist. In the augmented and bottleneck cases, it is possible that the input-output map has critical points. Using differential geometry and Morse theory for augmented MLPs  and augmented neural ODEs, we are able to prove that generically, i.e., except for a set of measure zero with respect to the Lebesgue measure in the weight space, every critical point is non-degenerate. For finite-depth neural networks with a bottleneck, critical points can be degenerate or non-degenerate. We derive conditions classifying the regularity of the critical points depending on the specific weights chosen. For degenerate neural ODEs, where the affine linear transformations used do not have full rank, we prove that every critical point is degenerate. Furthermore, we explain why the results obtained for bottleneck/degenerate architectures are comparable in the finite and the infinite depth case.
		
		It is important to study the property of being a Morse function, as for sufficiently smooth functions, it is a generic property to be a Morse function. Hence, to represent general data, it is often a good idea to aim for neural network architectures, which also have the property of being generically a Morse function. Furthermore, examples of non-degenerate critical points are extreme points, which can be global properties. If the data to approximate has a global extreme point, the chosen neural network architecture should also have the possibility to have a non-degenerate critical point, such that a good approximation or an exact representation is possible. \medskip
		
		In Section~\ref{sec:classification}, we introduce Morse functions and collect our main theorems in Table~\ref{tab:results} to compare the results established for MLPs  and neural ODEs depending on the specific architecture. Afterwards, we discuss in Section~\ref{sec:embeddingapproximation} which implications our results have on the universal embedding and approximation property of MLPs and neural ODEs. The analysis of MLPs is situated in Section~\ref{sec:feedforward}, whereas the analysis of neural ODEs can be found in Section~\ref{sec:neuralODEs}. 
		
		We start in Section~\ref{sec:feedforward} by introducing  in Section~\ref{sec:nn_architecture} the different architectures non-augmented, augmented, and bottleneck. We continue in Section~\ref{sec:nn_weightmatrices} with developing a theory for equivalent neural network architectures. If a weight matrix does not have full rank, the size of the network can be decreased without changing the output, but with a possible change in the network architecture. The goal of this section is to derive a normal form of an MLP, where all weight matrices have full rank. The normal form of the neural network is then analyzed in Section~\ref{sec:nn_criticalpoints} with respect to the property of having critical points, and in Section~\ref{sec:nn_regularity}, the regularity of these critical points is studied. We show that non-augmented neural networks have no critical points and that it is a generic property of augmented neural networks to be a Morse function. As MLPs  with a bottleneck show the most complex behavior, we devote Section~\ref{sec:nn_bottleneck} to the analysis of bottleneck architectures. 
		
		In analogy to Section~\ref{sec:feedforward}, we start Section~\ref{sec:neuralODEs} with Section~\ref{sec:node_architecture} to introduce the different neural ODE architectures non-augmented, augmented, and degenerate. We analyze in Section~\ref{sec:node_criticalpoints} the existence of critical points in neural ODE architectures and study the regularity of the critical points in  Section~\ref{sec:node_regularity}. As for MLPs, we show that non-augmented neural ODEs have no critical points, and that it is a generic property of augmented neural ODEs to be a Morse function. Furthermore, we prove that every critical point of a degenerate neural ODE is degenerate, and explain why this result is comparable to the analysis of MLPs with bottlenecks. Finally, we discuss in Section~\ref{sec:NODE_universal} the relationship between general and parameterized neural ODEs and show the universal embedding property for general augmented neural ODEs and the universal approximation property for certain parameterized augmented neural ODEs.
		
		\vspace{3mm}
		
		\section{Overview and Results}
		\label{sec:overview_results}
		
		In this work, we aim to compare MLPs  and neural ODEs with respect to the existence and the regularity of critical points. These two properties can be characterized via Morse functions. One aspect of this work is to fully rigorously prove and fundamentally generalize results indicated by Kurochkin \cite{Kurochkin2021} about the relationship between Morse functions and MLPs. 
		
		\subsection{Classification of Critical Points}\label{sec:classification}
		
		To characterize the existence and regularity of critical points, we use Morse functions, which are defined in the following. 
		
		\begin{definition}[Morse Function~\cite{Hirsch1976, Morse1934}] \label{def:morse}
			A map $\Phi \in C^2(\mathcal{X}, \mathbb{R})$ with $\mathcal{X} \subset \mathbb{R}^n$ open is called a Morse function if all critical points of $\Phi$ are non-degenerate, i.e., for every critical point $p \in \mathcal{X}$ defined by a zero gradient $\nabla \Phi(p) = 0$, the Hessian matrix $H_{\Phi}(p)$ is non-singular.
		\end{definition}
		
		Morse functions are generic in the space $C^k(\mathcal{X},\mathbb{R})$, $\mathcal{X} \subset \mathbb{R}^n$ open and bounded, of $k$ times continuously differentiable scalar functions $\Phi: \mathcal{X} \rightarrow \mathbb{R}$, as the following theorem shows.
		
		\begin{theorem}[\cite{Kuehn2023}]\label{th:morsefunctionsdense}
			Let $\mathcal{X} \subset \mathbb{R}^n$ be open and bounded. For $k \in \mathbb{N}_0$, the vector space
			\begin{equation*}
				C^k(\overbar{\mathcal{X}},\mathbb{R}) \coloneqq \left\{ \Phi \in C^k(\mathcal{X},\mathbb{R}): \Phi^{(i)} \textup{ is continuously continuable on } \overbar{\mathcal{X}} \textup{ for all } i \leq k  \right\},
			\end{equation*}
			endowed with the $C^k$-norm
			\begin{equation*}
				\norm{\Phi}_{C^k(\overbar{\mathcal{X}},\R)} \coloneqq \sum_{\abs{s}\leq k} \norm{\partial^s \Phi}_\infty
			\end{equation*}
			is a Banach space. Hereby $\overbar{\mathcal{X}}$ denotes the closure of $\mathcal{X}$ and $\norm{f}_\infty \coloneqq \sup_{x \in \overbar{\mathcal{X}}} \abs{f(x)}$ the supremum norm  of a continuous function $f: \overbar{\mathcal{X}} \rightarrow \mathbb{R}$. If additionally $k \geq n+1\geq 2$, the set of Morse functions 
			\begin{equation*}
				M \coloneqq \left\{\Phi \in C^k(\overbar{\mathcal{X}},\mathbb{R}): \Phi \big\vert_{\mathcal{X}} \textup{ is a Morse function}\right\}
			\end{equation*}
			is dense in $\left(C^k(\overbar{\mathcal{X}},\mathbb{R}), \norm{\cdot}_{C^k(\overbar{\mathcal{X}},\R)} \right)$.
		\end{theorem} 
		
		The definition of a Morse function motivates us to subdivide the space $C^k(\mathcal{X},\mathbb{R})$, $\mathcal{X} \subset \mathbb{R}^n$ open, into three disjoint subsets.
		\begin{definition} \label{def:subsets}
			For $k \in \mathbb{N}_{\geq 1}$, the following  subspaces of $C^k(\mathcal{X},\mathbb{R})$, $\mathcal{X} \subset \mathbb{R}^n$ open, are defined:
			\begin{itemize}
				\item The set of functions without critical points: $$(\mathcal{C}1)^k(\mathcal{X},\mathbb{R}) \coloneqq \left\{\Phi \in C^k(\mathcal{X}, \mathbb{R}) : \nabla \Phi(x) \neq 0\;  \forall x\right\}.$$  By definition, these functions are Morse functions for $k \geq 2$.
				\item For $k\geq 2$, the set of functions, which have at least one critical point and where each critical point is non-degenerate: $$(\mathcal{C}2)^k (\mathcal{X},\mathbb{R})\coloneqq\left\{\Phi \in C^k(\mathcal{X}, \mathbb{R}) : \left(\exists \, p:  \nabla \Phi(p) = 0\right) \; \wedge \; \left(\nabla \Phi(q) = 0 \; \Rightarrow \; H_{\Phi}(q) \textup{ is non-singular}\right) \right\}.$$  By definition, these functions are Morse functions.
				\item For $k\geq 2$, the set of functions which have at least one degenerate critical point: $$(\mathcal{C}3)^k(\mathcal{X},\mathbb{R}) \coloneqq\left\{\Phi \in C^k(\mathcal{X}, \mathbb{R}) : \exists \, p:  \left(\nabla \Phi(p) = 0 \; \wedge \; H_{\Phi}(p) \textup{ is singular}\right) \right\}.$$ By definition, these functions are not Morse functions.
			\end{itemize}
			Clearly, it holds for $k\geq 2$ that the three defined subspaces are non-empty and that they are a disjoint subdivision of   $C^k(\mathcal{X},\mathbb{R})$, i.e., it holds $C^k(\mathcal{X},\mathbb{R}) = (\mathcal{C}1)^k(\mathcal{X},\mathbb{R})\, \dot{\cup} \, (\mathcal{C}2)^k(\mathcal{X},\mathbb{R}) \, \dot{\cup}\, (\mathcal{C}3)^k(\mathcal{X},\mathbb{R})$.
		\end{definition}

		\begin{figure}
			\begin{tabular*}{\textwidth}{p{2.1cm}|p{6.1cm}|p{6cm}}
				& \vspace{1pt}\hspace{11mm}\textbf{Multilayer Perceptron} & \vspace{1pt} \hspace{18mm} \textbf{Neural ODE} \\[15pt]
				\hline && \\[0pt]
				\textbf{General}
				&
				Architecture $\Xi^k(\mathcal{X},\mathbb{R})$, $k \geq 0$:
				\begin{itemize}[itemsep=-2pt,topsep=3pt]
					\item MLP  with~update~rule~\eqref{eq:nn_updaterule}
					\item All activation functions $\sigma_l$ are applied component-wise, are strictly monotone in each component, and $[\sigma_l]_i\in C^k(\mathbb{R},\mathbb{R})$
					\item The first layer is the input $h_0 = x \in \mathcal{X}\subset  \mathbb{R}^n$, the last layer is the output $h_L \in \mathbb{R}$
					\item We assume the generic case, that all weight matrices have full rank (cf.\ Lemma~\ref{lem:nn_weightspace})
				\end{itemize} \vspace{25pt} 
				&
				Architecture $\textup{NODE}^k(\mathcal{X},\mathbb{R})$, $k \geq 0$:
				\begin{itemize}[itemsep=-2pt,topsep=3pt]
					\item The scalar neural ODE is defined by~\eqref{eq:NODE}, based on the time-$T$ map of an initial value problem in $\mathbb{R}^m$ and two affine linear layers $\lambda$, $\widetilde{\lambda}$
					\item The ODE has initial condition $\lambda(x)$, $x \in\mathcal{X}\subset \mathbb{R}^n$ and a vector field  $f\in C^{0,k}([0,T]\times \mathbb{R}^m,\mathbb{R}^m)$
					\item The output of the neural ODE is $\widetilde{\lambda}$ applied to the time-$T$ map of the initial value problem
				\end{itemize}  \vspace{25pt} \\
				\hline &&  \\[0pt]
				\textbf{Non-Augmented}
				&
				Special Architecture $\Xi^k_\textup{N}(\mathcal{X},\mathbb{R})$:
				\begin{itemize}[itemsep=-2pt,topsep=2pt]
					\item The width of all layers is less than or equal to $n$ and monotonically decreasing from layer to layer
				\end{itemize} \vspace{15pt} 
				& 
				Special Architecture $\textup{NODE}^k_\textup{N}(\mathcal{X},\mathbb{R})$:
				\begin{itemize}[itemsep=-2pt,topsep=3pt]
					\item It holds $n\geq m$
					\item The weight matrices of the affine linear maps $\lambda$, $\widetilde{\lambda}$ have full rank
				\end{itemize} \vspace{15pt} \\
				& 
				Properties for $k \geq 1$:
				\begin{itemize}[itemsep=-2pt,topsep=3pt]
					\item The neural network is of class $(\mathcal{C}1)^k(\mathcal{X},\mathbb{R})$, see Theorem~\ref{th:nn_nocriticalpoints} 
				\end{itemize}  \vspace{15pt} 
				&
				Properties for $k \geq 1$: 
				\begin{itemize}[itemsep=-2pt,topsep=3pt]
					\item The scalar neural ODE is of class $(\mathcal{C}1)^k(\mathcal{X},\mathbb{R})$, see Theorem~\ref{th:node_nocriticalpoints}
				\end{itemize}  \vspace{15pt} 
				\\
				\hline && \\[0pt]
				\textbf{Augmented}
				&
				Special Architecture $\Xi^k_\textup{A}(\mathcal{X},\mathbb{R})$: 
				\begin{itemize}[itemsep=-2pt,topsep=3pt]
					\item The layer of maximal width has at least \mbox{$n+1$} nodes, before this layer the width is monotonically increasing and after this layer the width is monotonically decreasing from layer to layer
				\end{itemize}\vspace{15pt} 
				& 
				Special Architecture $\textup{NODE}^k_\textup{A}(\mathcal{X},\mathbb{R})$:
				\begin{itemize}[itemsep=-2pt,topsep=3pt]
					\item It holds $n< m$
					\item The weight matrices of the affine linear maps $\lambda$, $\widetilde{\lambda}$ have full rank
				\end{itemize} \vspace{15pt}  \\
				& 
				Properties for $k \geq 2$:
				\begin{itemize}[itemsep=-2pt,topsep=3pt]
					\item For all sets of weights, except for a zero set w.r.t.\ the Lebesgue measure, the network is of class $(\mathcal{C}1)^k(\mathcal{X},\mathbb{R})$ or $(\mathcal{C}2)^k(\mathcal{X},\mathbb{R})$, see Theorem~\ref{th:nn_morse}
				\end{itemize} \vspace{15pt} 
				&
				Properties for $k \geq 2$:
				\begin{itemize}[itemsep=-2pt,topsep=3pt]
					\item For all sets of weights, except for a zero set w.r.t.\ the Lebesgue measure, the neural ODE is of class $(\mathcal{C}1)^k(\mathcal{X},\mathbb{R})$ or $(\mathcal{C}2)^k(\mathcal{X},\mathbb{R})$, see  Theorem~\ref{th:node_augmented}
				\end{itemize} \vspace{15pt} \\
				\hline && \\[0pt]
				\textbf{Bottleneck/ Degenerate} 
				&
				Special Architecture  $\Xi^k_\textup{B}(\mathcal{X},\mathbb{R})$: 
				\begin{itemize}[itemsep=-2pt,topsep=3pt]
					\item There exist three layers $g_i$, $g_j$, $g_l$ with $i<j<l$, such that the width of $g_j$ is strictly smaller than the width of $g_i$ and $g_l$ each
				\end{itemize}\vspace{15pt} 
				& 
				Special Architecture $\textup{NODE}^k_{\textup{D}}(\mathcal{X},\mathbb{R})$:
				\begin{itemize}[itemsep=-2pt,topsep=3pt]
					\item At least one of the weight matrices of the affine linear maps $\lambda$, $\widetilde{\lambda}$ has not full rank
				\end{itemize} \vspace{15pt}  \\
				& 
				Properties for $k \geq 2$:
				\begin{itemize}[itemsep=-2pt,topsep=3pt]
					\item The network can be of all classes, for a detailed classification depending on the type of bottleneck, see Theorem~\ref{th:nn_bottleneck}
				\end{itemize} \vspace{15pt} 
				&
				Properties for $k \geq 2$:
				\begin{itemize}[itemsep=-2pt,topsep=3pt]
					\item The scalar neural ODE is of class $(\mathcal{C}1)^k(\mathcal{X},\mathbb{R})$ or $(\mathcal{C}3)^k(\mathcal{X},\mathbb{R})$, see Theorem~\ref{th:node_bottleneck}
				\end{itemize} \vspace{20pt} 
			\end{tabular*}
			\vspace{3mm}
			\captionof{table}{Comparison of the main results for generic multilayer perceptrons and neural ODEs.} \label{tab:results}
		\end{figure}
		
		The existence of critical points and their regularity is studied in Section~\ref{sec:feedforward} for MLPs  and in Section~\ref{sec:neuralODEs} for neural ODEs. The main results regarding the classification of MLPs  and neural ODEs into the classes $(\mathcal{C}1)^k(\mathcal{X},\mathbb{R})$, $(\mathcal{C}2)^k(\mathcal{X},\mathbb{R})$, and $(\mathcal{C}3)^k(\mathcal{X},\mathbb{R})$ are collected in Table~\ref{tab:results}. The special architectures, non-augmented, augmented, and bottleneck/degenerate, are defined in detail in Section~\ref{sec:nn_architecture}  for MLPs and in Section~\ref{sec:node_architecture} for neural ODEs, but the most important properties of the different architectures are sketched in Table~\ref{tab:results}. For MLPs, we assume the generic case, that all weight matrices have full rank. In Section~\ref{sec:nn_weightmatrices} we show how neural networks with not full rank matrices can be transformed into equivalent neural networks with full rank matrices, which we call the normal form of the neural network. For non-augmented architectures, Table~\ref{tab:results} shows that the network is in the finite and in the infinite depth case of class $(\mathcal{C}1)^k(\mathcal{X},\mathbb{R})$, hence no critical points exist. In the case of an augmented architecture it is for MLPs  and for neural ODEs possible, that critical points exist. Yet, in this case, the critical points are generically non-degenerate, such that except for a Lebesgue measure zero set in the weight space, the input-output map is a Morse function of class $(\mathcal{C}1)^k(\mathcal{X},\mathbb{R})$ or $(\mathcal{C}2)^k(\mathcal{X},\mathbb{R})$. For MLPs  with a bottleneck, the output can be of all classes $(\mathcal{C}1)^k(\mathcal{X},\mathbb{R})$, $(\mathcal{C}2)^k(\mathcal{X},\mathbb{R})$ and $(\mathcal{C}3)^k(\mathcal{X},\mathbb{R})$, whereas degenerate neural ODEs can only be of class $(\mathcal{C}1)^k(\mathcal{X},\mathbb{R})$ or $(\mathcal{C}3)^k(\mathcal{X},\mathbb{R})$. Remark~\ref{rem:node_bottleneck} explains why, also in the bottleneck/degenerate case, the established theorems are comparable. 
		
		The overall result of our full classification in Table~\ref{tab:results} is that MLPs and neural ODEs have comparable properties regarding the existence and regularity of critical points. Furthermore, we can prove which cases may occur depending on the architecture. This provides a precise mathematical link between the underlying structure of the neural network or neural ODE and the geometry of the function it represents, i.e., whether it is a Morse function or not.

		\subsection{Universal Embedding and Universal Approximation} \label{sec:embeddingapproximation}
		
		A desirable property of neural network architectures is universal approximation, i.e., the property to represent every function of a given function space arbitrarily well. In an abstract context, universal approximation can be defined as follows. 
		
		\begin{definition}[Universal Approximation \cite{Kratsios2021}]
			A neural network $\Phi_\theta: \mathcal{X} \rightarrow \mathcal{Y}$ with parameters $\theta$, topological space $\mathcal{X}$ and metric space~$\mathcal{Y}$ has the universal approximation property with respect to the space $C^k(\mathcal{X},\mathcal{Y})$, $k \geq 0$, if for every $\varepsilon >0$ and for each function $\Psi \in C^k(\mathcal{X},\mathcal{Y})$, there exists a choice of parameters~$\theta$, such that $\textup{dist}_\mathcal{Y}(\Phi_\theta(x),\Psi(x)) < \varepsilon$ for all $x \in \mathcal{X}$.
		\end{definition}
		
		The parameters $\theta$ are all possible weights and biases, for MLPs, additionally, the activation functions, and for neural ODEs, the vector field. The property of universal approximation can depend on the metric of the space~$\mathcal{Y}$. For classical feed-forward neural networks like MLPs, ResNets, and recurrent neural networks (RNNs), various universal approximation theorems exist~\cite{Hornik1989, Kidger2022, Lin2018, Pinkus1999, Schaefer2006}. These theorems state, given a suitable activation function, that by increasing the width or depth of the neural network and the number of parameters, any function $\Psi \in C^k( \mathcal{X}, \mathcal{Y})$, $k \geq 0$ can be approximated arbitrarily well. A common property of universal approximation theorems is that the neural network architecture is augmented.
		
		Even though the established universal approximation theorems are extremely powerful, they do not explain why architectures without augmentation or with obstructions like bottlenecks have no universal approximation property. If we ask for an exact representation instead of an approximation, we can directly use the results established in this work to make statements about the universal embedding property of MLPs and neural~ODEs. 
		
		\begin{definition}[Universal Embedding]
			A neural network $\Phi_\theta: \mathcal{X} \rightarrow \mathcal{Y}$ with parameters $\theta$ and topological spaces $\mathcal{X}$ and $\mathcal{Y}$ has the universal embedding property with respect to the space $C^k(\mathcal{X},\mathcal{Y})$, $k \geq 0$, if for every function $ \Psi \in C^k(\mathcal{X},\mathcal{Y})$, there exists a choice of parameters~$\theta$, such that $\Phi_\theta(x) = \Psi(x)$ for all $x \in \mathcal{X}$.
		\end{definition}
		
		In the following, we first explain which direct implications our results have for the universal embedding property. Afterwards, we show how we can use our results about universal embedding to make statements about the universal approximation property.}
		
		\subsubsection*{Implications on the Universal Embedding Property}
		
		\textbf{Neural ODEs:} As the subsets $(\mathcal{C}1)^k(\mathcal{X},\mathbb{R})$,  $(\mathcal{C}2)^k(\mathcal{X},\mathbb{R})$ and  $(\mathcal{C}3)^k(\mathcal{X},\mathbb{R})$ of Definition~\ref{def:subsets} are all non-empty subsets of $C^k(\mathcal{X},\mathbb{R})$, it follows directly from Table~\ref{tab:results}, that non-augmented and degenerate neural ODEs cannot have the universal embedding property with respect to the space $C^k(\mathcal{X},\mathbb{R})$, $k \geq 2$. In particular, this means that non-degenerate critical points, such as non-degenerate extreme points, can only be represented by augmented neural ODEs. In Section~\ref{sec:NODE_universal} in Theorem~\ref{th:node_universalembedding}, we show via an explicit construction that augmented neural ODEs have the universal embedding property with respect to the space $C^1(\mathcal{X},\R)$. To exactly represent any given map $\Psi\in C^1(\mathcal{X},\R)$, the phase space needs to be augmented by at least one additional dimension, and we need to have the property to freely choose the vector field of the underlying initial value problem. \\[-2mm]
		
		\noindent \textbf{Multilayer Perceptrons:} For MLPs, we can also directly conclude from Table~\ref{tab:results}, that non-augmented MLPs with full rank weight matrices cannot have the universal embedding property. In the augmented case, the situation is more complicated, as Theorem~\ref{th:nn_morse} only holds for all weights except for a zero set with respect to the Lebesgue measure in the weight space. As for large classes of augmented MLPs, there exist already various universal approximation theorems, we do not treat the question of whether universal embedding can be proven in that case. If the neural network has a bottleneck, all classes $(\mathcal{C}1)^k(\mathcal{X},\mathbb{R})$,  $(\mathcal{C}2)^k(\mathcal{X},\mathbb{R})$ and  $(\mathcal{C}3)^k(\mathcal{X},\mathbb{R})$ are possible. In Theorem~\ref{th:nn_bottleneck}, we distinguish four different cases of bottlenecks. In the cases~\ref{th:nn_bottleneck_a} and~\ref{th:nn_bottleneck_b} we assume that the first bottleneck is non-augmented and prove that the input-output map cannot be of class $(\mathcal{C}2)^k(\mathcal{X},\mathbb{R})$, such that the neural network cannot have the universal embedding property. In the other two cases the first bottleneck is augmented and we formulate a condition distinguishing in case~\ref{th:nn_bottleneck_c} neural networks with a bottleneck, which show similar behavior like augmented neural networks and in case~\ref{th:nn_bottleneck_d} neural networks with a bottleneck, which cannot be of class $(\mathcal{C}1)^k(\mathcal{X},\mathbb{R})$. Only in case~\ref{th:nn_bottleneck_c}, it is possible that a universal embedding or approximation theorem as for augmented MLPs exists.
		
		\subsubsection*{Implications on the Universal Approximation Property}
		
		To better understand the implications of our results when using MLPs and neural ODEs in practice, we aim to study the universal approximation property. In Table~\ref{tab:results}, we have seen that a major restriction of non-augmented architectures is the non-existence of critical points. To show, that neural network architectures without critical points cannot have the universal approximation property, we estimate in the following theorem the distance between the function $\Psi_{y}: \mathcal{X} \rightarrow \R$, $\Psi_y(x) = \sum_{i = 1}^n (x_i-y_i)^2$, $\mathcal{X}\subset \R^n$ open, $y\in \mathcal{X}$, and a neural network of class $(\mathcal{C}1)^1(\mathcal{X},\mathbb{R})$.
			
			\begin{theorem}\label{th:nocriticalpoint}
				Consider a neural network $\Phi \in C^1(\mathcal{X},\R)$, $\mathcal{X}\subset \R^n$ open, of class $(\mathcal{C}1)^1(\mathcal{X},\mathbb{R})$. Fix $r>0$ and $y\in\mathcal{X}$, such that $K_r(y) \coloneqq \{x \in \R^n: \sum_{i =1}^n (x_i-y_i)^2 \leq r^2\}\subset  \mathcal{X}$. Then it holds for the approximation of the function $\Psi_y: \mathcal{X} \rightarrow \R$, $\Psi_y(x) = \sum_{i = 1}^n (x_i-y_i)^2$ that
				\begin{equation*}
					\sup_{x \in \mathcal{X}} \; \norm{\Phi(x)-\Psi_y(x)}_\infty \geq \frac{r^2}{2}.
				\end{equation*}
			\end{theorem}
			
			\begin{proof}
				Assume by contradiction that 
				\begin{equation*}
					\sup_{x \in \mathcal{X}} \; \norm{\Phi(x)-\Psi_y(x)}_\infty <\frac{r^2}{2}.
				\end{equation*}
				As $\Psi_y(y) = 0$ and $\Psi_y(x) = r^2$ for all $x \in \partial K_r(y) \coloneqq  \{x \in \R^n: \sum_{i =1}^n (x_i-y_i)^2 = r^2\}$, our contradiction assumption implies that
				\begin{equation*}
					-\frac{r^2}{2}<\Phi(y)<\frac{r^2}{2}, \qquad \text{and} \qquad \frac{r^2}{2}<\Phi(x)<\frac{3r^2}{2} \quad \text{ for all } x \in    \partial K_r(y),
				\end{equation*}
				so especially it holds $\Phi(y) < \Phi(x)$ for all $x \in \partial K_r(y)$.
				As the neural network map $\Phi$ is continuous, it attains its minimum and maximum on the compact set $K_r(y)$. The point $x_\text{min}\in K_r(y)$ with $\Phi(x_{\min})=\min_{x \in K_r(y)} \Phi(x)$ cannot lie on the boundary $\partial K_r(y)$, as we already showed that $\Phi(x_\text{min}) \leq \Phi(y)<\Phi(x)$ for all $x \in \partial K_r(y)$. Hence, it follows $x_\text{min}\in B_r(y)\coloneqq \{x \in \R^n: \sum_{i =1}^n (x_i-y_i)^2 < r^2\}$ and $\Phi(x_\text{min}) = \min_{x \in B_r(y)} \Phi(x)$. A necessary condition for existence of the minimum $x_\text{min}$ of the function $\Phi: B_r(y) \rightarrow \R$ is that $\nabla \Phi(x_\text{min}) = 0$, as $ B_r(y) \subset \R^n$ is an open set \cite{Forster2017}. As this is a contradiction to the assumption $\nabla \Phi(x) \neq 0$ for all $x \in \mathcal{X}$, the statement of the theorem follows.
			\end{proof}
			
			A direct consequence of this theorem is that every neural network architecture of class $(\mathcal{C}1)^1(\mathcal{X},\mathbb{R})$ cannot have the universal approximation property, as the following corollary shows.
			
			\begin{corollary}\label{cor:uniapprox}
				Consider a neural network architecture, wherein every neural network is of class $(\mathcal{C}1)^1(\mathcal{X},\mathbb{R})$, $\mathcal{X}\subset \R^n$ open. Then the considered class of neural network architectures cannot have the universal approximation property with respect to the space $C^k(\mathcal{X},\R)$ for every $k \geq 0$.
			\end{corollary}
			
			\begin{proof}
				Every neural network $\Phi$ of the considered class of architectures fulfills the assumptions of Theorem~\ref{th:nocriticalpoint} with common and fixed $r>0$ and $y\in \mathcal{X}$. As the function $\Psi_y: \mathcal{X} \rightarrow \R$, $\Psi_y(x) = \sum_{i = 1}^n (x_i-y_i)^2$ is an element of the space  $C^k(\mathcal{X},\R)$ for every $k \geq 0$, the considered class of neural network architectures cannot have the universal approximation property, as $\Psi_y$ cannot be approximated by any neural network $\Phi$ with precision $\eps < \frac{r^2}{2}$.
			\end{proof}

			We can use the statement of Corollary~\ref{cor:uniapprox} together with the results of Table~\ref{tab:results} to make statements about the (non-)universal approximation property of MLPs and neural ODEs. \\[-2mm]
			
			\noindent \textbf{Neural ODEs:} It follows directly from Table~\ref{tab:results} and Corollary~\ref{cor:uniapprox} that non-augmented and degenerate neural ODEs cannot have the universal approximation property with respect to the space $C^k(\mathcal{X},\mathbb{R})$ for $k \geq 0$. For augmented neural ODEs, Theorem~\ref{th:node_universalembedding} shows the universal embedding property with respect to the space $C^1(\mathcal{X},\R)$, $\mathcal{X}\subset \R^n$ open, in the case of a general vector field. In the case of a parameterized vector field, we show in Theorem~\ref{th:node_approx}, under which conditions the universal embedding property of general augmented neural ODEs transfers to a universal approximation property of parameterized augmented neural ODEs. \\[-2mm]
			
			\noindent \textbf{Multilayer Perceptrons:} As for neural ODEs, it follows directly from Table~\ref{tab:results} and Corollary~\ref{cor:uniapprox} that non-augmented MLPs cannot have the universal approximation property. The same statement holds for MLPs with a bottleneck in the cases~\ref{th:nn_bottleneck_a},~\ref{th:nn_bottleneck_b}, and~\ref{th:nn_bottleneck_d} of Theorem~\ref{th:nn_bottleneck}. As discussed earlier, there already exist many universal approximation theorems for augmented MLPs in the literature~\cite{Hornik1989,Pinkus1999}. As typical universal approximation theorems make no statement about the rank of the weight matrices and the existence of bottlenecks, the case~\ref{th:nn_bottleneck_c} of Theorem~\ref{th:nn_bottleneck} about MLPs with a specific augmented bottleneck can be included in the class of MLP architectures, for which universal approximation theorems can exist. 
			
			\begin{remark}
				In \cite[Theorem 1]{Johnson2018} it is shown that MLPs with full rank weight matrices and strictly monotone activation functions, where the dimension of the hidden layer does not exceed the input dimension $n\geq 2$, cannot approximate functions $\Psi:\R^n\rightarrow \R$, which have at least one bounded level set. This result gives only implicit information about the input-output map of non-augmented MLPs. In contrast to that, our results directly characterize, under the additional assumption that the activation functions are continuously differentiable, the geometric obstruction that non-augmented MLPs cannot have critical points. The non-existence of critical points implies that the input-output map of non-augmented MLPs has only unbounded level sets, as a bounded level set would define the boundary of a compact set, in whose interior a critical point exists \cite{Forster2017}. 
			\end{remark}

			Overall, our analysis shows which geometric obstructions, such as the non-existence or the degeneracy of critical points, prevent MLPs and neural ODEs from having the universal embedding and approximation property. Additionally, we show positive results in the case of augmented neural ODEs when general neural ODEs have the universal embedding and parameterized neural ODEs have the universal approximation property.

		\section{Multilayer Perceptrons} \label{sec:feedforward}

		A classical feed-forward neural network is structured in main layers $h_l$, $l \in \{0,1,\ldots, L\}$, $L \geq 1$ with width~$n_l$, i.e., $h_l \in \mathbb{R}^{n_l}$. The network is initialized by input data $h_0 = x \in \mathcal{X}$, $\mathcal{X} \subset \mathbb{R}^n$ open with $n \coloneqq n_0$ and the layers are in the case of multilayer perceptrons (MLPs) for $l \in \{1,\ldots,L\}$ iteratively updated by
		\begin{equation} \label{eq:nn_updaterule} 
			h_{l} = \widetilde{W}_l\sigma_l(W_lh_{l-1}+b_l)+\tilde{b}_l = \widetilde{W}_l\sigma_l(a_l)+\tilde{b}_l,
		\end{equation}
		where $\widetilde{W}_l \in \mathbb{R}^{n_{l} \times m_l}$ and $W_l \in \mathbb{R}^{m_{l}\times n_{l-1}}$ are weight matrices, $b_l \in \mathbb{R}^{m_{l}}$ and $\tilde{b}_l \in \mathbb{R}^{n_{l}}$ are biases and $\sigma_l: \mathbb{R}^{m_{l}}\rightarrow \mathbb{R}^{m_{l}}$ is a component-wise applied activation function, i.e., $\sigma_l = ([\sigma_l]_1,\ldots, [\sigma_l]_{m_{l}})^\top$ with $[\sigma_l]_i:\mathbb{R}\rightarrow \mathbb{R}$ for $i \in \{1,\ldots,m_{l}\}$. $m_l$ is the dimension of the intermediate layers, also called pre-activations (cf.\ \cite{Goodfellow2016}), $a_l \coloneqq W_l h_{l-1}+b_l$, $l \in \{1,\ldots,L\}$, to which the activation function $\sigma$ is applied. For an easier notation, we additionally enumerate the main and intermediate layers together by $g_j \in \mathbb{R}^{d_j}$, $j\in\{0,\ldots,2L
		\}$, where $g_{2l} = h_l \in \mathbb{R}^{d_{2l}} = \mathbb{R}^{n_l}$ for $l \in \{0,\ldots,L\}$ and $g_{2l-1}=a_l \in \mathbb{R}^{d_{2l-1}}=\mathbb{R}^{m_l}$ for $l \in \{1,\ldots,L\}$. Additionally, we enumerate the weight matrices together by $V_j \in \mathbb{R}^{d_j\times d_{j-1}}$, $j\in\{1,\ldots,2L\}$, such that $(W_1, \widetilde{W}_1, \ldots, W_L, \widetilde{W}_L) = (V_1,\ldots,V_{2L})$ with $W_l = V_{2l-1}$ and $\widetilde{W}_l = V_{2l}$ for $l\in\{1,\ldots,L\}$. The output of the neural network is the last layer $h_L$. The structure of the neural network is visualized in Figure~\ref{fig:nn_architecture}. The general neural network update rule~\eqref{eq:nn_updaterule} contains MLPs with an outer nonlinearity $$h_l = \sigma_l(W_lh_{l-1}+b_l)$$ by choosing $m_l = n_l$, $\widetilde{W}_l = \textup{Id}_{n_l}$ and $\tilde{b}_l = 0$ for $l\in \{1,\ldots,L\}$ and MLPs with an inner nonlinearity $$h_l = \widetilde{W}_l\sigma_l(h_{l-1})+\tilde{b}_l$$ by choosing $m_l = n_{l-1}$, $W_l = \textup{Id}_{n_{l-1}}$ and $b_l = 0$ for $l\in \{1,\ldots,L\}$. Both architectures, with outer and inner nonlinearity, as well as their combination with update rule~\eqref{eq:nn_updaterule}, are well-known in machine learning theory~\cite{Esteve2020}. The distinction between the different types of nonlinearities also exists in classical neural field models of mathematical neuroscience: the class of Wilson-Cowan models has outer nonlinearities and the class of Amari models has inner nonlinearities \cite{Cook2022}. 
		
		In the main theorems of this work we assume, that all components of $\sigma_l$ are continuous and strictly monotone on $\mathbb{R}$, i.e., $[\sigma_l]_i'>0$ or $[\sigma_l]_i'<0$ on $\mathbb{R}$ in the case that $[\sigma_l]_i$ is differentiable, $i \in \{1,\ldots,m_l\}$. As we restrict our analysis to single components of the output, we can, without loss of generality, assume that the output $h_L \in \mathbb{R}$ is one-dimensional, i.e., $n_L = 1$. Hence, the neural network mapping $h_0$ to $h_L$ is a function $\Phi: \mathcal{X} \rightarrow \mathbb{R}$.
		
		\begin{lemma} \label{lem:nn_regularity}
			Let $\Phi: \mathcal{X} \rightarrow \mathbb{R}$, $\mathcal{X}\subset \mathbb{R}^n$ open, be an MLP with update rule~\eqref{eq:nn_updaterule}. If the activation functions fulfill for $l \in \{1,\ldots,L\}$, $i \in \{1,\ldots,m_l\}$ that $[\sigma_l]_i\in C^k(\mathbb{R},\mathbb{R})$ with $k \geq 0$, then $\Phi \in C^k(\mathcal{X},\mathbb{R})$. 
		\end{lemma}
		
		\begin{proof}
			As a composition of $k$ times continuously differentiable functions, the chain rule implies that $\Phi$ is also $k$ times continuously differentiable. 
		\end{proof}
		
		\begin{definition}[Multilayer Perceptron] \label{def:feedforward}
			For $k \geq 0$, the set of all MLPs $\Phi: \mathcal{X} \rightarrow \mathbb{R}$, $\mathcal{X}\subset \mathbb{R}^n$ open with update rule~\eqref{eq:nn_updaterule} and component-wise applied strictly monotone activation functions $[\sigma_l]_i\in C^k(\mathbb{R},\mathbb{R})$, $l \in \{1,\ldots,L\}$, $i \in \{1,\ldots,m_{l}\}$ is denoted by $\Xi^k(\mathcal{X},\mathbb{R}) \subset C^k(\mathcal{X},\mathbb{R})$.
		\end{definition}
		
		The set of MLPs $\Xi^k(\mathcal{X},\mathbb{R})$ we study in this work is larger than the class of networks considered in the work by Kurochkin \cite{Kurochkin2021}, where the analysis is restricted to outer nonlinearities with strictly monotonically increasing activation functions. For most of the upcoming theorems, we provide, for illustration purposes, low-dimensional examples. If explicit calculations of gradients and Hessian matrices are needed in the examples, we work with the soft-plus action function $\sigma(x) = \ln(1+\exp(x))$, which has as first derivative the sigmoid function $\sigma'(x) = \frac{1}{1+\exp(-x)}$.

		\begin{figure}
			\centering
			\includegraphics[width=\textwidth]{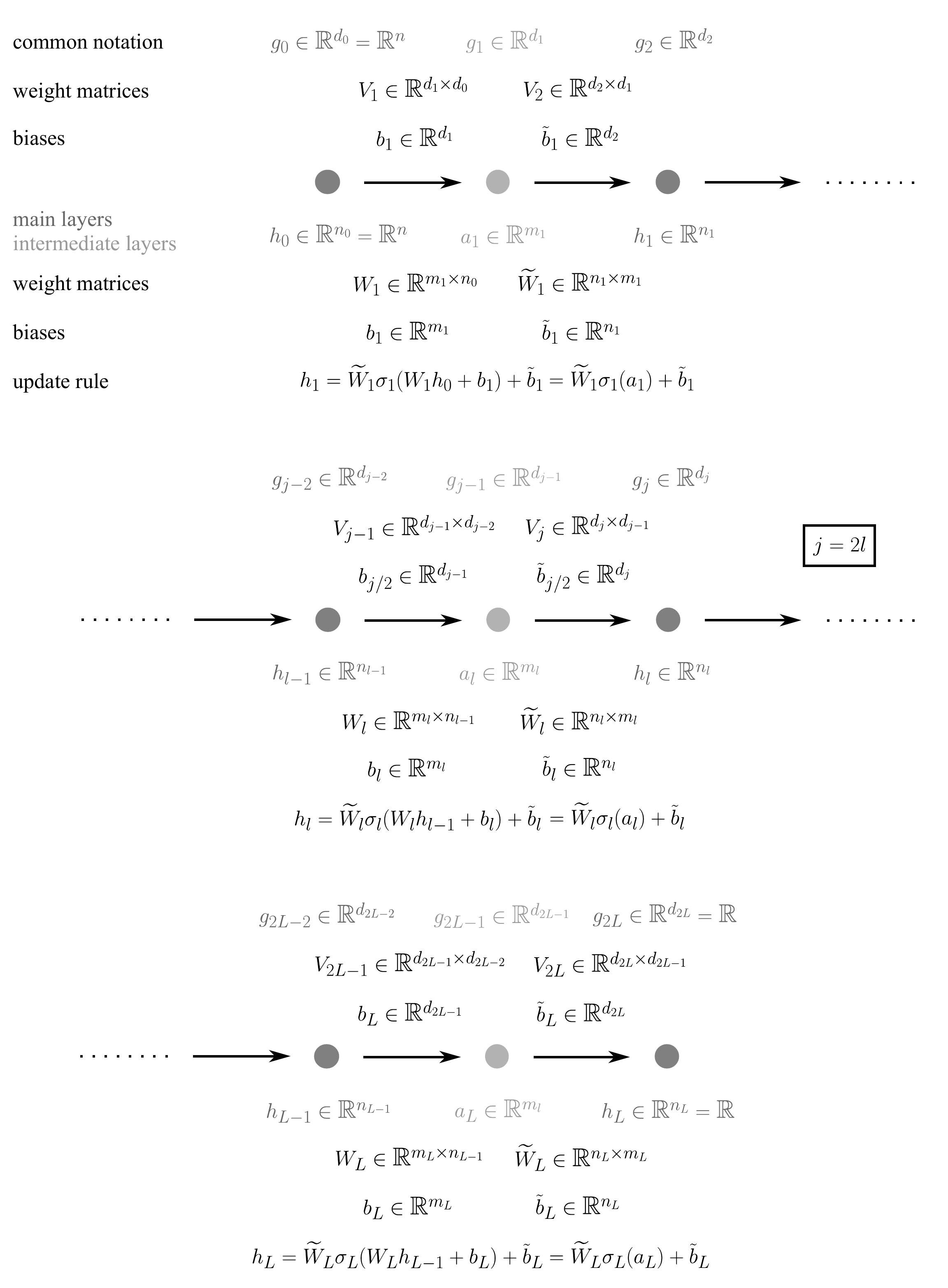}
			\caption{Structure of an MLP with update rule~\eqref{eq:nn_updaterule} in standard notation  $(l\in \{0,\ldots,L\})$ and in common notation $(j \in \{0,\ldots,2L\})$ for main and intermediate layers, visualized in dark and light gray respectively.}
			\label{fig:nn_architecture}
		\end{figure}
		
		\subsection{Special Architectures}
		\label{sec:nn_architecture}
		
		In the following, we subdivide the set of all classical feed-forward neural networks, such as the MLPs $\Xi^k(\mathcal{X},\mathbb{R})$ introduced in Definition~\ref{def:feedforward}, in three different classes: non-augmented neural networks, augmented neural networks, and neural networks with a bottleneck.
		
		\subsubsection*{Non-Augmented}
		
		We call a classical feed-forward neural network non-augmented if it holds for the widths of the layers that $d_j \leq n$ for $j \in \{0,\ldots,2L\}$. The space is non-augmented, as all layers have a width smaller than or equal to the width $n$ of the layer $h_0$, which corresponds to the input data $x \in \mathcal{X}\subset \mathbb{R}^n$. Additionally, we require that the widths of the layers are monotonically decreasing from layer to layer, i.e., $d_{j-1} \geq d_{j} $ for $j \in \{1,\ldots,2L\}$. The subset of non-augmented MLPs is denoted by
		\begin{equation*}
			\Xi^k_\textup{N}(\mathcal{X},\mathbb{R}) \coloneqq \left\{\Phi \in \Xi^k(\mathcal{X},\mathbb{R}) : \Phi \textup{ is non-augmented}\right\}.
		\end{equation*}
		
		\subsubsection*{Augmented}
		
		We call a classical feed-forward neural network $\Phi \in \Xi^k(\mathcal{X},\mathbb{R}) $ augmented if the first layer of maximal width $g_{l^\ast}$ has at least $n+1$ nodes, i.e., $d_{l^\ast}\geq n+1$ and $d_{l^\ast-1}<d_{l^\ast}$. The space is augmented, as there exists a layer with width larger than the dimension of the input data $x \in \mathbb{R}^n$. Furthermore, we require that between layer $g_0$ and $g_{l^\ast}$, the width of the layers is monotonically increasing, i.e., $d_{j-1} \leq d_{j}$ for $0<j \leq l^\ast$ and between layer $g_{l^\ast}$ and $g_{2L}$, the width of the layers is monotonically decreasing, i.e., $d_{j-1} \geq d_{j}$ for $l^\ast < j \leq  2L$. The subset of augmented MLPs is denoted by
		\begin{equation*}
			\Xi^k_\textup{A}(\mathcal{X},\mathbb{R}) \coloneqq \left\{\Phi \in \Xi^k(\mathcal{X},\mathbb{R}) : \Phi \textup{ is augmented}\right\}.
		\end{equation*}

		\subsubsection*{Bottleneck}
		
		We say that a classical feed-forward neural network $\Phi \in \Xi^k(\mathcal{X},\mathbb{R}) $ has a bottleneck in layer $g_{j^\ast}$ if there exists three layers $g_{i^\ast}$, $g_{j^\ast}$ and $g_{l^\ast}$ with $0 \leq i^\ast < j^\ast < l^\ast \leq 2L$, such that $d_{i^\ast}>d_{j^\ast}$ and $d_{j^\ast}<d_{l^\ast}$. Typical neural network architectures that have a bottleneck are auto-encoders, where the dimension is first reduced to extract specific features and then augmented again to the dimension of the initial data. The subset of MLPs with a bottleneck is denoted by
		\begin{equation*}
			\Xi^k_\textup{B}(\mathcal{X},\mathbb{R}) \coloneqq \left\{\Phi \in \Xi^k(\mathcal{X},\mathbb{R}) : \Phi \textup{ has a bottleneck}\right\}.
		\end{equation*}  
		
		The three different types of feed-forward neural networks are visualized in Figure~\ref{fig:nn_architectures}. In the following, we show that these three types of architectures build a disjoint subdivision of all classical feed-forward neural networks.  
		
		\begin{figure}
			\centering
			\begin{subfigure}{0.9\textwidth}
				\centering
				\includegraphics[width=0.9\textwidth]{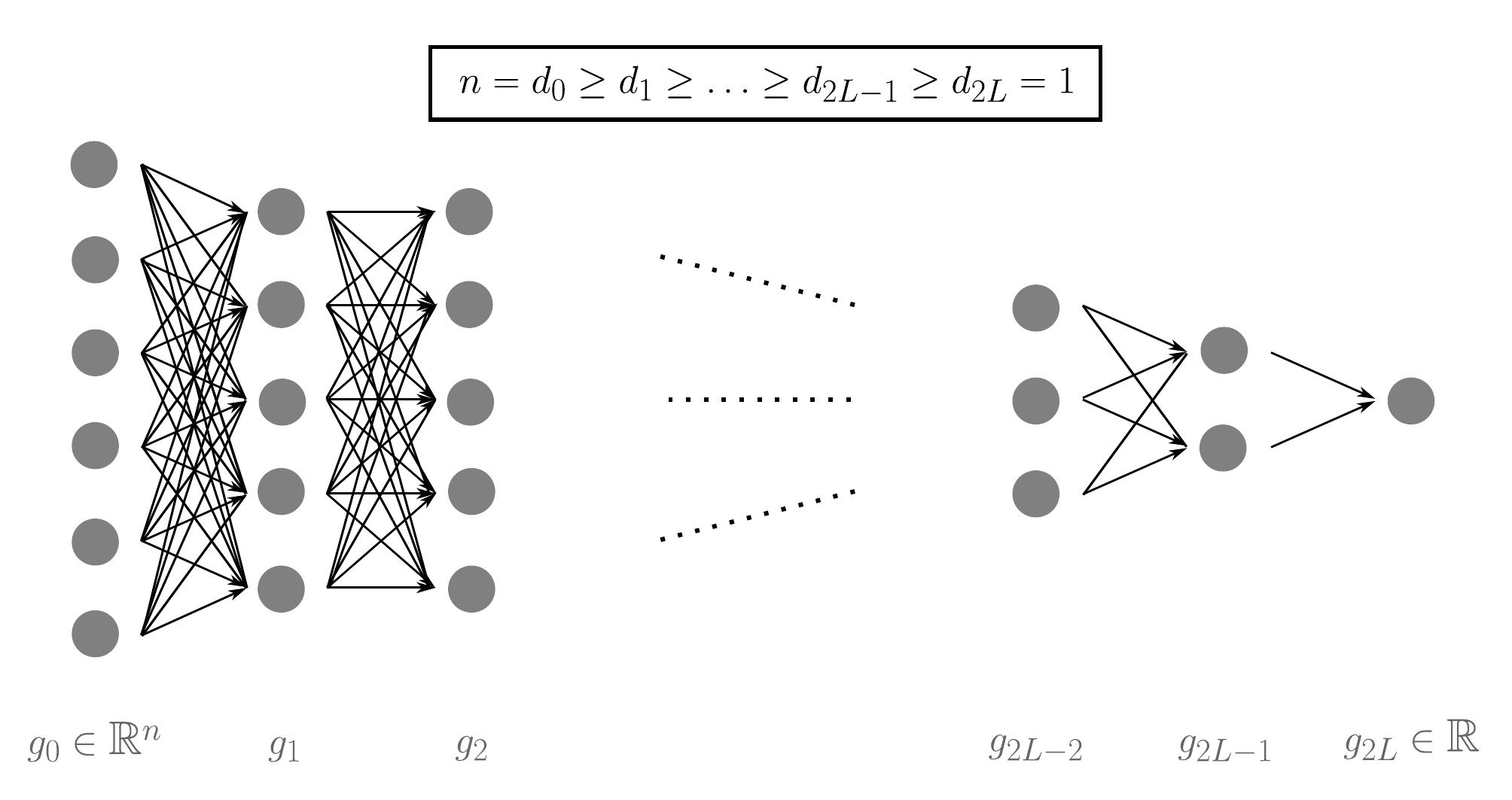}
				\caption{Structure of a non-augmented neural network, like for example the MLP $\Phi \in \Xi_\textup{N}^k(\mathcal{X},\mathbb{R})$.}
				\label{fig:nn_arch_nonaugmented}
			\end{subfigure}
			\begin{subfigure}{0.9\textwidth}
				\vspace{7mm}
				\centering
				\includegraphics[width=0.9\textwidth]{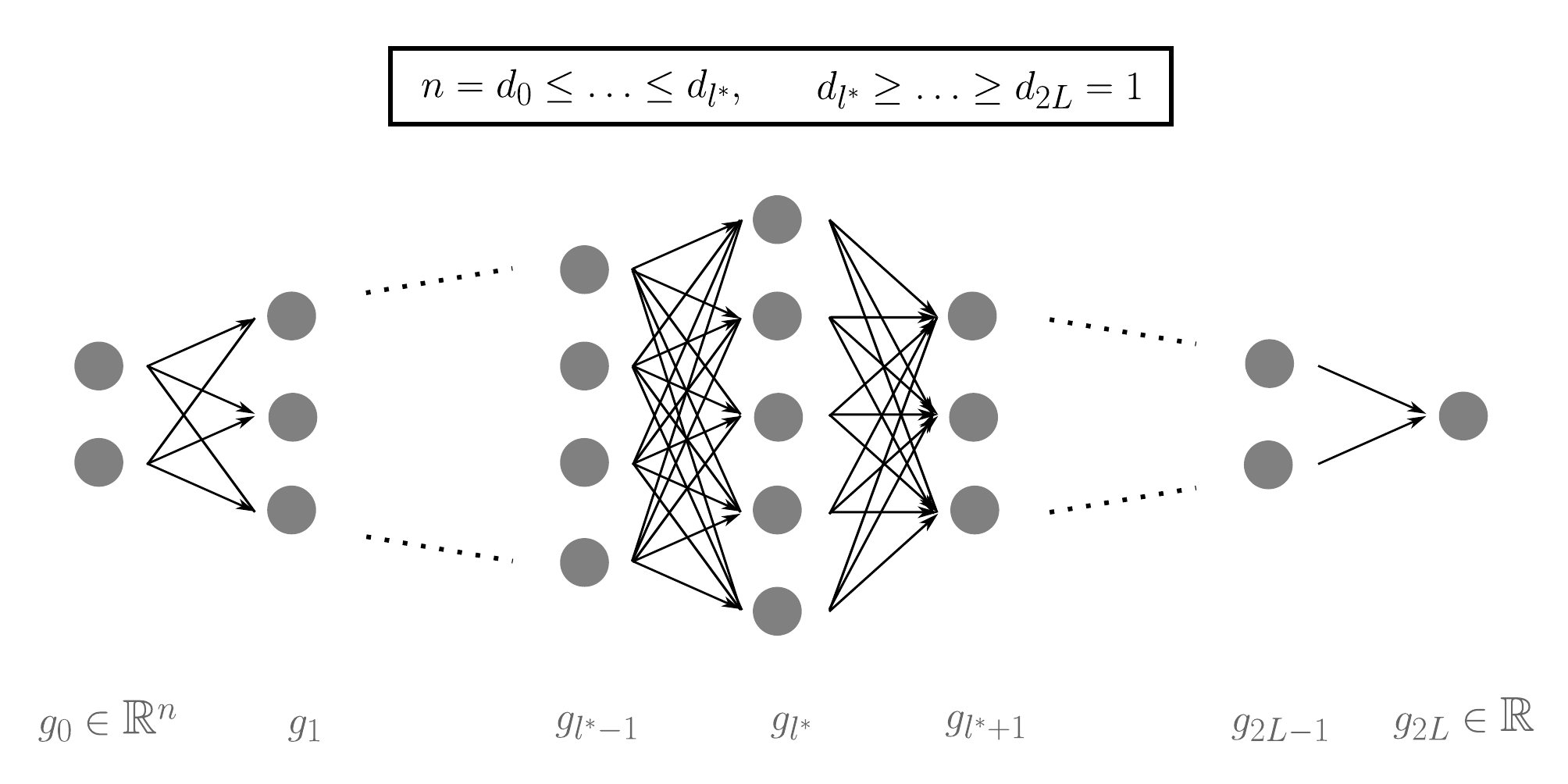}
				\caption{Structure of an augmented neural network, like for example the MLP $\Phi \in \Xi_\textup{A}^k(\mathcal{X},\mathbb{R})$.}
				\label{fig:nn_arch_augmented}
			\end{subfigure}
			\begin{subfigure}{0.9\textwidth}
				\vspace{7mm}
				\centering
				\includegraphics[width=0.9\textwidth]{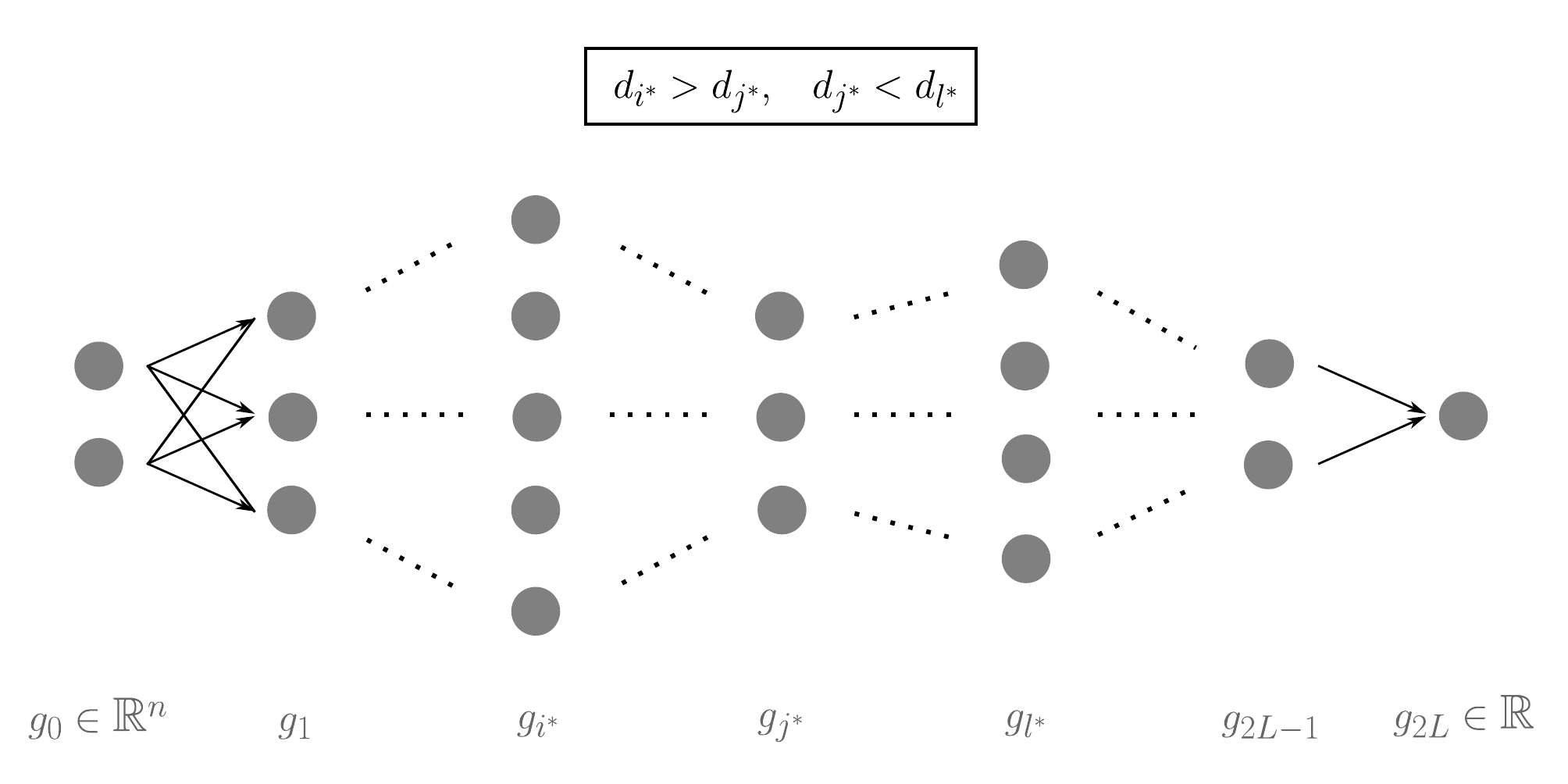}
				\caption{Structure of a neural network with a bottleneck, like for example the MLP $\Phi \in \Xi_\textup{B}^k(\mathcal{X},\mathbb{R})$.}
				\label{fig:nn_arch_bottleneck}
			\end{subfigure}
			\caption{The three different types of architectures non-augmented, augmented, and bottleneck, for classical feed-forward neural networks.}
			\label{fig:nn_architectures}
		\end{figure}

		\begin{proposition} \label{prop:nn_disjointpartition}
			The subdivision of classical feed-forward neural networks in non-augmented neural networks, augmented neural networks, and neural networks with a bottleneck, is a complete partition in three disjoint sub-classes of feed-forward neural networks, i.e., in the case of MLPs it holds
			\begin{equation*}
				\Xi^k(\mathcal{X},\mathbb{R}) = \Xi^k_\textup{N}(\mathcal{X},\mathbb{R})\; \dot{\cup} \;\, \Xi^k_\textup{A}(\mathcal{X},\mathbb{R}) \; \dot{\cup} \;\, \Xi^k_\textup{B}(\mathcal{X},\mathbb{R})
				.		\end{equation*}
		\end{proposition}
		
		\begin{proof}
			We define the three index sets
			\begin{align*}
				\mathcal{J}_{<}&\coloneqq \left\{j \in \{1,\ldots,2L\}: d_{j-1}<d_j \right\}, \\
				\mathcal{J}_{=}&\coloneqq \left\{j \in \{1,\ldots,2L\}: d_{j-1}= d_j \right\}, \\
				\mathcal{J}_{>}&\coloneqq \left\{j \in \{1,\ldots,2L\}: d_{j-1}>d_j \right\},
			\end{align*}
			which are a disjoint subdivision of all indices, i.e., $\mathcal{J}_{<} \, \dot{\cup} \, \mathcal{J}_{=} \, \dot{\cup} \, \mathcal{J}_{>} =  \{1,...,2L\}$. We distinguish now different cases depending on the partition of the indices in the sets $\mathcal{J}_{<}$, $\mathcal{J}_{=}$, and $\mathcal{J}_{>}$.
			\begin{itemize}[itemsep=-1pt,topsep=4pt]
				\item Case 1: $\mathcal{J}_{<} = \varnothing$. In this case it holds $d_{j-1}\geq d_j$ for all $j \in \{1,\ldots,2L\}$. As $d_0 = n$ this implies that $d_j \leq n$ for all $j \in \{0,\ldots,2L\}$. By definition, the feed-forward neural network is non-augmented.
				\item Case 2: $\mathcal{J}_{<} \neq \varnothing$. As the set $\mathcal{J}_{<}$ is finite, there exists a maximal element $l^\ast = \max\{j: j \in \mathcal{J}_{<}\}$.
				\item Case 2.1: $\mathcal{J}_{<} \neq \varnothing$ and $\{1,\ldots,l^\ast\}\subset \left(\mathcal{J}_{<} \cup \mathcal{J}_{=}\right)$. In this case it holds $d_{j-1} \leq d_{j}$ for $0<j \leq l^\ast$. This implies $d_{l^\ast}>d_{l^\ast-1}\geq \ldots \geq d_0 = n$, hence $d_{l^\ast} \geq n+1$. As $l^\ast = \max\{j: j \in \mathcal{J}_{<}\}$ it holds $\{l^\ast+1,\ldots,2L\}\subset \left(\mathcal{J}_{=} \cup \mathcal{J}_{>}\right)$. Consequently $d_{j-1} \geq d_{j}$ for $l^\ast < j \leq  2L$. By definition, the feed-forward neural network is augmented.
				\item Case 2.2: $\mathcal{J}_{<} \neq \varnothing$ and $\{1,\ldots,l^\ast\}\not\subset \left(\mathcal{J}_{<} \cup \mathcal{J}_{=}\right)$. Define now the index $j^\ast$ as the maximal element of $\{1,\ldots,l^\ast\}$ which is not contained in $\mathcal{J}_{<} \cup \mathcal{J}_{=}$: $j^\ast= \max\{j \in \{1,\ldots,l^\ast\}: j \in \mathcal{J}_{>}\}$. It holds $j^\ast < l^\ast$ as $l^\ast \in \mathcal{J}_{<}$. For $i^\ast = j^\ast-1$ it holds $0 \leq i^\ast < j^\ast<l^\ast \leq 2L$ and $d_{i^\ast} = d_{j^\ast-1}> d_{j^\ast}$. As $\{j^\ast+1,\ldots,l^\ast\} \subset\mathcal{J}_{<} \cup \mathcal{J}_{=} $ and $l^\ast \in \mathcal{J}_{<}$ it also holds $d_{j^\ast}\leq \ldots \leq d_{l^\ast-1}<d_{l^\ast}$. This implies that the considered feed-forward neural network has a bottleneck as $d_{i^\ast}> d_{j^\ast}$ and $d_{j^\ast}<d_{l^\ast}$.
			\end{itemize}
			As the three cases 1, 2.1 and 2.2 are a complete and disjoint subdivision of possible partitions of the indices $1,\ldots,2L$ in the sets $\mathcal{J}_{<}$, $\mathcal{J}_{=}$ and $\mathcal{J}_{>}$, also the subdivision of classical feed-forward neural networks, such as the MLPs $\Xi^k(\mathcal{X},\mathbb{R})$ in non-augmented architectures $\Xi^k_\textup{N}(\mathcal{X},\mathbb{R})$, augmented architectures $\Xi^k_\textup{A}(\mathcal{X},\mathbb{R})$ and architectures with a bottleneck $\Xi^k_\textup{B}(\mathcal{X},\mathbb{R})$, is a complete partition in three disjoint sub-classes of feed-forward neural networks.
		\end{proof}

		\subsection{Equivalent Neural Network Architectures}
		\label{sec:nn_weightmatrices}

		In this section, we first define the weight and biases spaces for MLPs. Afterwards, we introduce a definition of equivalent neural network architectures and show via an explicit algorithm how to obtain, under all equivalent neural network architectures, the one with the fewest number of nodes. The resulting architecture has only full rank matrices and is called the normal form of the neural network, which is analyzed in the following Sections~\ref{sec:nn_criticalpoints} and~\ref{sec:nn_regularity}.
		
		To define the weight space, we identify a matrix in $\mathbb{R}^{p\times q}$ with a point in $\mathbb{R}^{pq}$ by stacking the columns of the matrix. This operation is also known as vectorization.
		
		\begin{definition}[Vectorization \cite{Magnus2019}] \label{def:vectorization}
			Let $A \in \mathbb{R}^{p\times q}$ be a matrix and denote the $j$-th column by $a_j \in \mathbb{R}^{p}$, $j\in \{1,\ldots,q\}$. Then the bijective vectorization operator $\ve: \mathbb{R}^{p\times q} \rightarrow \mathbb{R}^{pq}$ is defined as
			\begin{equation*}
				\ve(A) = \begin{pmatrix}
					a_1 \\ \vdots \\ a_q
				\end{pmatrix}
			\end{equation*}
			with inverse  $\ve^{-1}: \mathbb{R}^{pq} \rightarrow \mathbb{R}^{p \times q} $.
		\end{definition} 
		
		In the following, we define weight and bias spaces for MLPs, generalize the vectorization operator of Definition~\ref{def:vectorization} to multiple weight matrices, and introduce a stacking operator for biases.
		
		\begin{definition}[MLP Parameter Space]\label{def:nn_weightspace}
			Let $\Phi \in \Xi^k(\mathcal{X},\mathbb{R})$, $k \geq 0$ be an MLP with weight matrices $W = (W_1, \widetilde{W}_1, \ldots, W_L, \widetilde{W}_L) = (V_1,\ldots,V_{2L})$ and biases $b = (b_1, \tilde{b}_1 \ldots, b_L, \tilde{b}_L)$. 
			\begin{enumerate}[label=(\alph*), font=\normalfont]
				\item 	The weight space $\mathbb{W}$ is defined as the space of all weight matrices,
				\begin{align*}
					\mathbb{W} &\coloneqq \mathbb{R}^{m_1 \times n_0}  \times \mathbb{R}^{n_1 \times m_1}\times \ldots \times \mathbb{R}^{m_L \times n_{L-1}} \times \mathbb{R}^{n_L \times m_L} \\
					& = \mathbb{R}^{d_1 \times d_0}  \times \mathbb{R}^{d_2 \times d_1}\times  \ldots \times \mathbb{R}^{d_{2L-1} \times d_{2L-2}} \times \mathbb{R}^{d_{2L} \times d_{2L-1}} .
				\end{align*}
				\item \label{def:nn_weightspace_b} For $W \in \mathbb{W}$ the bijective multiple vectorization operator  $\vem: \mathbb{W} \rightarrow \mathbb{R}^{N_W}$ with $N_W \coloneqq \sum_{l=1}^L m_l(n_{l-1} + n_l) = \sum_{l=1}^{2L} d_l d_{l-1}$ is defined as $$ \vem(W) \coloneqq \begin{pmatrix}
					\ve(W_1) \\ \vdots \\ \ve(\widetilde{W}_L)
				\end{pmatrix}=\begin{pmatrix}
					\ve(V_1) \\ \vdots \\ \ve(V_{2L})
				\end{pmatrix}\in \mathbb{R}^{N_W}.$$ The inverse of the operator $\vem$ is denoted by $\vem^{-1}: \mathbb{R}^{N_W} \rightarrow \mathbb{W} $. 
				\item  The subset of the weight space $\mathbb{W}$, in which all weight matrices $W$ have full rank is denoted by
				\begin{align*}
					\mathbb{W}^\ast \coloneqq &\Bigl\{W \in \mathbb{W}: \textup{ rank}(W_l) = \min\{m_l,n_{l-1}\},\textup{ rank}(\widetilde{W}_l) = \min\{m_l,n_l\} \; \forall \; l \in \{1,\ldots,L\} \Bigr\} \\
					=&\Bigl\{W \in \mathbb{W}: \textup{ rank}(V_j) = \min\{d_j,d_{j-1}\} \; \forall \; j \in \{1,\ldots,2L\}\Bigr\}.
				\end{align*}
				The set of MLPs $\Phi \in \Xi^k_i(\mathcal{X},\mathbb{R})$, $i \in \{\varnothing,\textup{N},\textup{A},\textup{B}\}$ with weights $W \in \mathbb{W}^\ast$ is denoted by $\Xi^k_{i,\mathbb{W}^\ast}(\mathcal{X},\mathbb{R})$.
				\item The bias space $\mathbb{B}$ is defined as the space of all biases $b = (b_1, \tilde{b}_1 \ldots, b_L, \tilde{b}_L)$,
				\begin{equation*}
					\mathbb{B} \coloneqq \mathbb{R}^{m_1} \times \mathbb{R}^{n_1}  \times \ldots \times \mathbb{R}^{m_L} \times \mathbb{R}^{n_L} = \mathbb{R}^{d_1}\times \mathbb{R}^{d_2}  \times \ldots \times \mathbb{R}^{d_{2L-1}} \times \mathbb{R}^{d_{2L}}.
				\end{equation*}
				\item For $b = (b_1, \tilde{b}_1 \ldots, b_L, \tilde{b}_L) \in \mathbb{B}$ the bijective stacking operator  $\stack: \mathbb{B} \rightarrow \mathbb{R}^{N_B}$ with $N_B \coloneqq \sum_{l=1}^L (m_l + n_l) = \sum_{l=1}^{2L} d_l$ is defined as $$ \stack(b) \coloneqq \begin{pmatrix}
					b_1 \\ \vdots \\ \tilde b_L
				\end{pmatrix}\in \mathbb{R}^{N_B}.$$ The inverse of the operator $\stack$ is denoted by $\stack^{-1}: \mathbb{R}^{N_B} \rightarrow \mathbb{B} $. 
			\end{enumerate}
		\end{definition}
		
		The following lemma shows that the subset of all weight matrices $W \in \mathbb{W}$, where at least one matrix has not full rank, i.e., the set $\mathbb{W}_0 \coloneqq \mathbb{W}\setminus \mathbb{W}^\ast$, is a zero set with respect to the Lebesgue measure in the identified space $\mathbb{R}^{N_W}$.

		\begin{lemma} \label{lem:nn_weightspace}
			For all weights $w \in \mathbb{R}^{N_W}$, except possibly for a zero set in $\mathbb{R}^{N_W}$ with respect to the Lebesgue measure, all corresponding weight matrices $W = \vem^{-1}(w)$ have full rank, i.e., the set $$\vem(\mathbb{W}_0) \coloneqq \left\{\vem(W) \in \mathbb{R}^{N_W}:  W \in \mathbb{W}_0 \right\},$$ where $\mathbb{W}_0 \coloneqq \mathbb{W}\setminus \mathbb{W}^\ast$, is a zero set in $\mathbb{R}^{N_W}$.
		\end{lemma}
		
		\begin{proof}
			By \cite[Theorem 5.15]{Prasolov2006}, the set of $p \times q$ matrices with rank $k$
			\begin{equation*}
				M_{p,q,k} \coloneqq \left\{M \in \mathbb{R}^{p \times q}: \textup{rank}(M) = k \right\} \subset \mathbb{R}^{p \times q},
			\end{equation*}
			is in the identified space $\mathbb{R}^{pq}$ a manifold of dimension $k(p+q-k)$. For matrices which do not have full rank, i.e., the rank $k$ fulfills $k < k^\ast \coloneqq \min\{p,q\}$, the dimension of $M_{p,q,k}$ is strictly smaller than $pq$ as for $k = k^\ast - i$, $i \geq 1$ it holds
			\begin{align*}
				k(p+q-k)&= (k^\ast-i)(p+q-k^\ast + i) = (\min\{p,q\}-i)(\max\{p,q\}+i) \\
				&= pq-i(\max\{p,q\}-\min\{p,q\})-i^2 < pq.
			\end{align*}
			Consequently $\ve(M_{p,q,k})\coloneqq  \{\ve(M) \in \mathbb{R}^{pq}:  M \in M_{p,q,k}\}$ has for $k < k^\ast$ Lebesgue measure zero in~$\mathbb{R}^{pq}$. We denote the set of all matrices that do not have full rank by
			\begin{equation*}
				M_{p,q,k<k^\ast} \coloneqq \bigcup_{k=0}^{k^\ast -1} M_{p,q,k}.
			\end{equation*} 
			Since
			\begin{equation*}
				\ve(M_{p,q,k<k^\ast}) = \ve\left(\bigcup_{k=0}^{k^\ast -1} M_{p,q,k}\right) = \bigcup_{k=0}^{k^\ast -1} \ve( M_{p,q,k})
			\end{equation*}
			is a finite union of sets of Lebesgue measure zero in $\mathbb{R}^{pq}$, the Lebesgue measure of $\ve(M_{p,q,k<k^\ast})$ is also zero in $\mathbb{R}^{pq}$. Transferred to the weight space $\mathbb{W}$ it follows that for $j \in \{1,\ldots,2L\}$, the set
			\begin{equation*}
				\vem\left(\mathbb{R}^{d_1 \times d_0} \times \ldots \times \mathbb{R}^{d_{j-1} \times d_{j-2}}\times  M_{d_j,d_{j-1},k<k^\ast(j)} \times \mathbb{R}^{d_{j+1} \times d_{j}} \times \ldots \times \mathbb{R}^{d_{2L} \times d_{2L-1}}\right) 
			\end{equation*}
			is a zero set with respect to the Lebesgue measure in $\mathbb{R}^{N_W}$. Hereby $k^\ast(j) = \min\{d_j,d_{j-1}\}$ denotes the maximal rank of a matrix in $\mathbb{R}^{d_j \times d_{j-1}}$. Consequently, also the finite union  
			\begin{align*}
				\vem( \mathbb{W}_0)  
				=\; & \vem \left( \bigcup_{j=1}^{2L} \mathbb{R}^{d_1 \times d_0} \times \ldots \times \mathbb{R}^{d_{j-1} \times d_{j-2}}\times  M_{d_j,d_{j-1},k<k^\ast(j)} \times \mathbb{R}^{d_{j+1} \times d_{j}} \times \ldots \times \mathbb{R}^{d_{2L} \times d_{2L-1}}\right) \\
				= \; &\bigcup_{j=1}^{2L} \vem\Big(\mathbb{R}^{d_1 \times d_0} \times \ldots \times \mathbb{R}^{d_{j-1} \times d_{j-2}}\times  M_{d_j,d_{j-1},k<k^\ast(j)} \times \mathbb{R}^{d_{j+1} \times d_{j}} \times \ldots \times \mathbb{R}^{d_{2L} \times d_{2L-1}}\Big) 
			\end{align*}
			is a zero set with respect to the Lebesgue measure in $\mathbb{R}^{N_W}$. This implies that the set of weights $w \in \mathbb{R}^{N_W}$, where at least one of the weight matrices of $W = \vem^{-1}(w)$ has not full rank is a zero set in $\mathbb{R}^{N_W}$ with respect to the Lebesgue measure and the result follows.
		\end{proof}
		
		A first reason to study mainly MLPs with full rank weight matrices is that the set of full rank matrices has full measure in $\mathbb{R}^{N_W}$. A second reason is that every MLP, which has at least one non-full rank weight matrix, is equivalent to a MLP in normal form with fewer nodes and only full rank matrices. By equivalent, we mean that both neural network architectures have the same output for every input $x \in \mathcal{X}$.
		
		\begin{definition}[Equivalent MLP] \label{def:nn_equivalence}
			Let $k\geq 0$ and $\Phi \in \Xi^k(\mathcal{X},\mathbb{R})$ be an MLP. We call another MLP $\overline{\Phi} \in \Xi^k(\mathcal{X},\mathbb{R})$ equivalent to $\Phi$ if $\Phi(x) = \overline{\Phi}(x)$ for all $x \in \mathcal{X}$. We say that the neural network $\overline{\Phi}$ has a smaller (larger) architecture than $\Phi$ if the number of nodes per layer fulfills $\overline{d}_j \leq d_j$ ($\overline{d}_j \geq d_j$) for $j \in \{0,\ldots,2L\}$ and the total number of nodes is strictly smaller (larger), i.e., $\overline{N}_B < N_B$ ($\overline{N}_B > N_B$). It is possible that the number of layers $\overline{L}$ of $\overline{\Phi}$ is smaller than the number of layers $L$ in $\Phi$, then we set $\overline{d}_j = 0$ for $j \in \{2 \overline{L} + 1, \ldots, 2L\}$.
		\end{definition}
		
		We always denote the equivalent neural network and all corresponding weight matrices, biases, and dimensions with an overbar. In the following Lemmata~\ref{lem:nn_notfullrank1} and~\ref{lem:nn_notfullrank2}, we study MLPs, which have at least one non-full rank matrix and explicitly construct an equivalent neural network, where the considered matrix has full rank. Afterwards, we show in Theorem~\ref{th:nn_fullrank_equivalent} how to combine the two lemmata to construct an equivalent neural network architecture with only full rank weight matrices $W \in \mathbb{W}^\ast$. As the construction of the equivalent MLPs is quite technical, we postpone the proofs of the two lemmata to Appendix~\ref{app:normalform}.
		
		First, we consider the case that an inner weight matrix $W_{l+1}\in \mathbb{R}^{m_{l+1} \times n_{l}}$ has not full rank for some $l \in \{1,\ldots,L-1\}$ and show that the MLP is equivalent to a network with a smaller architecture, which fulfills certain rank properties.
		
		\begin{lemma} \label{lem:nn_notfullrank1}
			Let $\Phi \in \Xi^k_{ \mathbb{W}_0}(\mathcal{X},\mathbb{R})$ with a non-zero, non-full rank matrix $W_{l+1} \in \mathbb{R}^{m_{l+1} \times n_{l}}$ for some $l \in \{1,\ldots,L-1\}$, i.e., $0<\textup{rank}(W_{l+1}) < \min\{m_{l+1},n_{l}\}$. Then there exists an equivalent MLP $\overline{\Phi} \in \Xi^k(\mathcal{X},\mathbb{R})$ with smaller architecture. Especially only the number of nodes in layer $h_{l}$ is reduced by one, i.e., $\overline{n}_{l} = n_{l}-1\geq 1$, and $W_{l+1}$, $\widetilde{W}_{l}$ and $\tilde{b}_{l}$ are replaced by new matrices $\overline{W}_{l+1} \in \mathbb{R}^{m_{l+1} \times \overline{n}_{l}}$ and $\overline{\widetilde{W}}_{l} \in \mathbb{R}^{\overline{n}_{l} \times m_{l}}$ with $ \textup{rank}(\overline{W}_{l+1}) = \textup{rank}(W_{l+1})$, $ \textup{rank}(\overline{\widetilde{W}}_{l}) \leq \textup{rank}(\widetilde{W}_{l})$, $\overline{W}_{l+1} \overline{\widetilde{W}}_{l} = W_{l+1} \widetilde{W}_{l}$ and a new bias $\overline{\tilde{b}}_{l} \in \mathbb{R}^{\overline{n}_{l}}$.
		\end{lemma}
		
		\begin{proof}
			Lemma~\ref{lem:nn_notfullrank1} is proven as Lemma~\ref{lem:nn_notfullrank1_app} in Appendix~\ref{app:normalform}.
		\end{proof}

		Also in the case that an outer weight matrix $\widetilde{W}_l\in \mathbb{R}^{m_l \times n_{l-1}}$ has for some $l \in \{1,\ldots,L-1\}$ not full rank, we show that the neural network is equivalent to a 
		smaller architecture with certain rank properties. 
		
		\begin{lemma} \label{lem:nn_notfullrank2}
			Let $\Phi \in \Xi^k_{ \mathbb{W}_0}(\mathcal{X},\mathbb{R})$ with a non-zero, non-full rank matrix $\widetilde{W}_l \in \mathbb{R}^{n_l \times m_l}$ for some $l \in \{1,\ldots,L-1\}$, i.e., $0<\textup{rank}(\widetilde{W}_l) < \min\{m_l,n_{l}\}$. Then there exists an equivalent MLP $\overline{\Phi} \in \Xi^k(\mathcal{X},\mathbb{R})$ with smaller architecture. Especially only the number of nodes in layer $h_l$ is reduced by one, i.e., $\overline{n}_{l} = n_{l}-1\geq 1$, and $\widetilde{W}_{l}$, $W_{l+1}$ and $\tilde{b}_{l}$ are replaced by new matrices  $\overline{\widetilde{W}}_{l} \in \mathbb{R}^{\overline{n}_{l} \times m_{l}}$ and $\overline{W}_{l+1} \in \mathbb{R}^{m_{l+1} \times \overline{n}_{l}}$ with $ \textup{rank}(\overline{\widetilde{W}}_l) = \textup{rank}(\widetilde{W}_l)$, $ \textup{rank}(\overline{W}_{l+1}) \leq \textup{rank}(W_{l+1})$, $\overline{W}_{l+1}\overline{\widetilde{W}}_{l} = W_{l+1} \widetilde{W}_l$ and a new bias $\overline{\tilde{b}}_{l} \in \mathbb{R}^{\overline{n}_{l}}$. If $W_{l+1} \in \mathbb{R}^{m_{l+1}\times n_l}$ has full rank, i.e., $\textup{rank}(W_{l+1}) = \min\{m_{l+1},n_l\}$, then $\overline{W}_{l+1}$ has also full rank, i.e., $ \textup{rank}(\overline{W}_{l+1}) = \min\{m_{l+1},\overline{n}_l\}$.
		\end{lemma}
		
		\begin{proof}
			Lemma~\ref{lem:nn_notfullrank2} is proven as Lemma~\ref{lem:nn_notfullrank2_app} in Appendix~\ref{app:normalform}.
		\end{proof}

		Finally, we aim to apply in the upcoming theorem Lemma~\ref{lem:nn_notfullrank1} and Lemma~\ref{lem:nn_notfullrank2} to find smaller equivalent architectures, until all weight matrices of the MLP have full rank. Theorem~\ref{th:nn_fullrank_equivalent} is the main reason why we restrict ourselves to the analysis of neural networks with full rank weight matrices $W \in \mathbb{W}^\ast$. It shows that the neural network is either up to a linear change of coordinates equivalent to a smaller architecture with full rank matrices, or the neural network is a constant function. The resulting architecture is called the MLP normal form. To prove the following theorem, we additionally need some basic results from linear algebra stated in Appendix \ref{sec:app_LA}.
		
		\begin{theorem}[MLP Normal Form] \label{th:nn_fullrank_equivalent}
			Let $\Phi \in \Xi^k_{ \mathbb{W}_0}(\mathcal{X},\mathbb{R})$, $k \geq 0$, $\mathcal{X}\subset \mathbb{R}^n$ open, be an MLP with weight matrices $W = (W_1, \widetilde{W}_1, \ldots, W_L, \widetilde{W}_L) \in  \mathbb{W}_0$ and biases $b = (b_1,\tilde{b}_1,\ldots,b_L,\tilde{b}_L) \in \mathbb{B}$.
			\begin{enumerate}[label=(\alph*), font=\normalfont]
				\item \label{th:nn_fullrank_equivalent_a} If $\textup{rank}(W_1)>0$, $\textup{rank}(\widetilde{W}_L)>0$ and $\textup{rank}(W_{l+1} \widetilde{W}_l) > 0 $ for all $l\in \{1,\ldots,L-1\}$, there exists an MLP $\overline{\Phi} \in \Xi^k_{\mathbb{W}^\ast}(\mathcal{Y},\mathbb{R})$, $\mathcal{Y} \subset \mathbb{R}^{\bar{n}}$ for some $\bar{n} \leq n$, which has a smaller architecture, only full rank weight matrices and which is up to a linear change of coordinates, represented by a full rank matrix $A \in \mathbb{R}^{\bar{n} \times n}$, equivalent to $\Phi$. If $\bar{n} = n$, then $A$ is the identity matrix and $\mathcal{Y} = \mathcal{X}$, if $\bar{n}<n$, then $\mathcal{Y}=\{y\in \mathbb{R}^{\bar{n}}: y = Ax, x \in \mathcal{X}\}$.
				\item  \label{th:nn_fullrank_equivalent_b} If $\textup{rank}(W_1)=0$ or $\textup{rank}(\widetilde{W}_L)=0$  or $\textup{rank}(W_{l+1} \widetilde{W}_l) = 0 $ for some $l\in \{1,\ldots,L-1\}$, it holds $\Phi(x) = c_{W,b}$ for all $x \in \mathcal{X}$, where $c_{W,b}$ is a scalar independent of $x$, only depending on $W$ and~$b$. $\Phi$ is equivalent to  $\overline{\Phi}\in \Xi^k(\mathcal{X},\mathbb{R})$ defined by $\overline{L} = 1$, $\overline{W}_1 = 0 \in \mathbb{R}^{1\times n}$, $\overline{b}_1 = 0 \in \mathbb{R}$, $\overline{\widetilde{W}}_1 = 0 \in \mathbb{R}$ and $\overline{\tilde{b}}_1 = c_{W,b} \in \mathbb{R}$, which has a smaller architecture than $\Phi$ if $\overline{\Phi} \neq \Phi$.
			\end{enumerate}
		\end{theorem}
		
		\begin{proof}
			For case~\ref{th:nn_fullrank_equivalent_a} we consider an MLP $\Phi \in \Xi^k_{ \mathbb{W}_0}(\mathcal{X},\mathbb{R})$, where at least one weight matrix has not full rank and denote the weight matrices by $W^{(0)} = (W_1^{(0)}, \widetilde{W}_1^{(0)}, \ldots, W_L^{(0)}, \widetilde{W}_L^{(0)})$ and its input by $x^{(0)} \in \mathcal{X}$. The assumption $\textup{rank}(W_1^{(0)})>0$, $\textup{rank}(\widetilde{W}_L^{(0)})>0$ and $\textup{rank}(W_{l+1}^{(0)} \widetilde{W}_l^{(0)}) > 0 $ for all $l\in \{1,\ldots,L-1\}$ implies, that all weight matrices of $W^{(0)}$ have at least rank one and every weight matrix, which has not full rank, has at least width and height two. In the following, we explicitly construct equivalent MLPs with smaller architectures, until all weight matrices have full rank. The structure of the algorithm is visualized in Figure~\ref{fig:nn_algorithm}. We start with step~\ref{th:nn_step_i} and $r = 0$.
			
			\begin{enumerate}[label=(\roman*)]
				\item \label{th:nn_step_i} If this is the first iteration of step~\ref{th:nn_step_i}, denote the considered MLP by $\Phi^{(\textup{(i)},0)} \in \Xi^k_{ \mathbb{W}_0}(\mathcal{X},\mathbb{R})$, the weight matrices by $(W_1^{(0)}, \widetilde{W}_1^{(0)},W_2^{(0)}, \ldots, W_L^{(0)}, \widetilde{W}_L^{(0)})$ and set $r = 0$. 
				
				If  $W_1^{(r)}\in \mathbb{R}^{m_1 \times (n-r)}$ has full rank, go to step~\ref{th:nn_step_ii} and in the case that $r = 0$ set $r_0 = 0$, $W_1^{(\ast)} \coloneqq  W_1^{(r_0)} = W_1^{(0)}$,  $\mathcal{Y} \coloneqq \mathcal{X}$, $\Phi^{(\textup{(i)},r_0)} = \Phi^{(\textup{(i)},0)}$ and define $A \coloneqq \text{Id}_n$, which is the identity transformation on~$\mathbb{R}^{n}$. 
				
				If $W_1^{(r)}\in \mathbb{R}^{m_1 \times (n-r)}$ has not full rank, it holds $0<\textup{rank}(W_1^{(r)}) < \min\{m_1,n-r\}$. The columns of $W_1^{(r)} $ are linearly dependent, hence there exist scalars $\tilde{\alpha}_j^{(r)}$, $j \in \{1,\ldots,n-r\}$ not all equal to zero, such that
				\begin{equation*}
					\sum_{j = 1}^{n-r} \tilde{\alpha}_j^{(r)} [W_1^{(r)}]_j = 0 \in \mathbb{R}^{m_1}.
				\end{equation*}
				Since $\tilde{\alpha}_i^{(r)} \neq 0$ for some $i \in \{1,\ldots,n-r\}$, it holds 
				\begin{equation*}
					[W_1^{(r)}]_i = \sum_{j = 1, j \neq i}^{n-r} \alpha_j^{(r)} [W_1^{(r)}]_j
				\end{equation*}
				with $\alpha_j^{(r)} \coloneqq-\tilde{\alpha}_j^{(r)}/\tilde{\alpha}_i^{(r)}$. We replace the matrix $W_1^{(r)}$ by the matrix  
				\begin{equation*}
					W_1^{(r+1)} = \begin{pmatrix}
						[W_1^{(r)}]_{1} &\cdots& [W_1^{(r)}]_{i-1} & [W_1^{(r)}]_{i+1} \cdots &[W_1^{(r)}]_{n-r} 
					\end{pmatrix} \in \mathbb{R}^{m_1 \times (n-r-1)}
				\end{equation*}
				which has the same rank as $W_1^{(r)}$,  $W_1^{(r)}$ arises from $W_1^{(r+1)}$ by adding a new column which is a linear combination of the other columns and hence does not increase the maximal number of linearly independent columns. As $\text{rank}(W_1^{(r+1)}) = \text{rank}(W_1^{(r)})$, it holds \mbox{$\text{rank}(W_1^{(r+1)})>0$}. 
				Furthermore we make a linear change of coordinates $A^{(r)}: \mathbb{R}^{n-r} \rightarrow \mathbb{R}^{n-r-1}$ of the input data $x^{(r)} \subset \mathbb{R}^{n-r}$ defined by
				\begin{equation*}
					x^{(r+1)}  = A^{(r)}(x^{(r)}) = \begin{pmatrix}
						1 & \ldots & 0 &\alpha_1 & 0 & \ldots & 0 \\
						\vdots & &0& \ldots & 0 && \vdots \\
						0 & \ldots & 1 & \alpha_{i-1} & 0 & \ldots &0 \\
						0 & \ldots & 0 & \alpha_{i+1} & 1 &\ldots &0 \\
						\vdots & &0& \ldots & 0 && \vdots \\
						0 & \ldots & 0 &\alpha_{n-r} & 0 & \ldots & 1
					\end{pmatrix} x^{(r)} = 
					\begin{pmatrix}
						x_1 + \alpha_1 x_i \\
						\vdots \\
						x_{i-1} + \alpha_{i-1} x_i \\
						x_{i+1} + \alpha_{i+1} x_i \\
						\vdots \\
						x_{n-r} + \alpha_{n-r} x_i
					\end{pmatrix} \in \mathbb{R}^{n-r-1},
				\end{equation*}
				such that $W_1^{(r)} x^{(r)} = W_1^{(r+1)} x^{(r+1)} =  W_1^{(r+1)} A^{(r)}(x^{(r)}) $ for all $x^{(r)} \in \mathbb{R}^{n-r}$. By construction, the matrix $A^{(r)}$ has full rank. The new neural network $\Phi^{(\textup{(i)},r+1)}$ with $W_1^{(r)}$ replaced by $W_1^{(r+1)}$ and $x^{(r)}$ replaced by $x^{(r+1)}$ is up to the linear change of coordinates $A^{(r)}$ equivalent to $\Phi^{(\textup{(i)},r)}$ and has a smaller architecture. 
				
				If $W_1^{(r+1)}$ has full rank, continue with step~\ref{th:nn_step_ii} and define $r_0 = r+1$ and
				\begin{equation*}
					A:\mathbb{R}^n \rightarrow \mathbb{R}^{\bar{n}}, \quad A \coloneqq A^{(r)} \circ \ldots \circ A^{(0)},
				\end{equation*}
				$\mathcal{Y} \coloneqq A(\mathcal{X}) \subset \mathbb{R}^{\bar{n}}$, $\bar{n}\coloneqq n-r-1$, and $W_1^{(\ast)} \coloneqq  W_1^{(r_0)}$. The matrix $A$ has by Lemma~\ref{lem:nn_productfullrank} full rank $\bar{n}$, as all matrices $A^{(i)}$, $i \in \{0,\ldots,r\}$ have full rank and the dimensions of the matrices are monotone. The neural network $\Phi^{(\textup{(i)},r_0)} \in \Xi^k(\mathcal{Y},\mathbb{R})$ is up to the linear change of coordinates $A$ equivalent to  $\Phi^{(\textup{(i)},0)} \in \Xi^k(\mathcal{X},\mathbb{R})$.
				
				\begin{figure}
					\centering
					\includegraphics[width=0.9\textwidth]{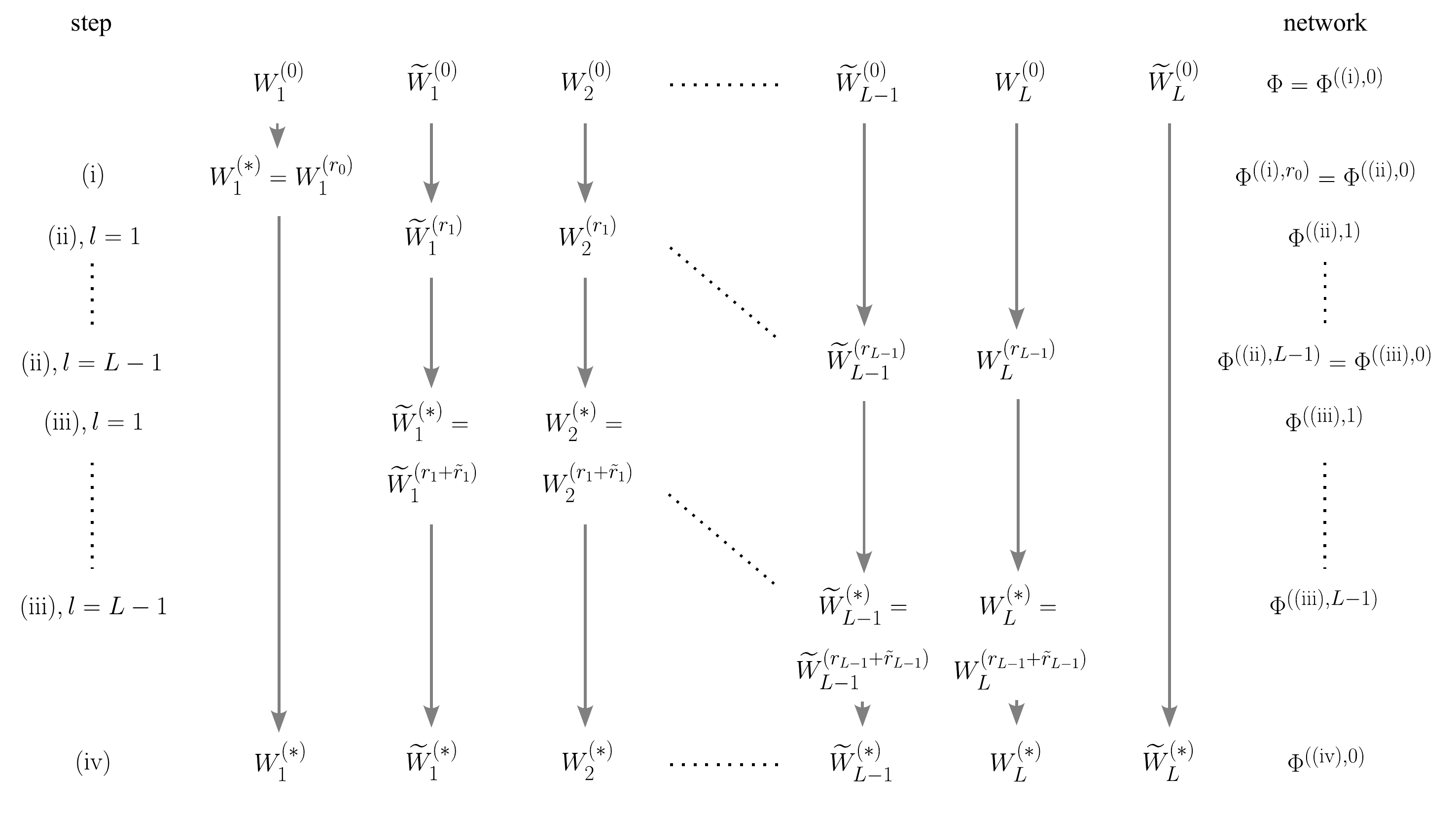}
					\caption{Structure of the algorithm defined in  the proof of Theorem~\ref{th:nn_fullrank_equivalent}\ref{th:nn_fullrank_equivalent_a}.}
					\label{fig:nn_algorithm}
				\end{figure}
				
				If $W_1^{(r+1)}$ has not full rank, repeat step~\ref{th:nn_step_i} for the neural network $\Phi^{(\textup{(i)},r)}$ with the counter~$r$ increased by one. Lemma~\ref{lem:nn_iterations} guarantees that after a finite number of iterations, $W_1^{(r+1)}$ has full rank and step~\ref{th:nn_step_i} ends.
				
				\item \label{th:nn_step_ii} If this is the first iteration of step~\ref{th:nn_step_ii}, denote the MLP $\Phi^{(\textup{(i)},r_0)} \in \Xi^k(\mathcal{Y},\mathbb{R})$ obtained from~\ref{th:nn_step_i} with weight matrices $(W_1^{(\ast)}, \widetilde{W}_1^{(0)},W_2^{(0)}, \ldots, W_L^{(0)}, \widetilde{W}_L^{(0)})$ by $\Phi^{(\textup{(ii)},0)}$ and set $l = 1$. If $L = 1$, go to step~\ref{th:nn_step_iii}.
				
				If $W_{l+1}^{(0)}\in \mathbb{R}^{m_{l+1}  \times n_{l}}$ has full rank, set $r_l = 0$ and define  $\Phi^{(\textup{(ii)},l)} \coloneqq \Phi^{(\textup{(ii)},l-1)}$. If $l<L-1$, increase the counter $l$ by one and repeat step~\ref{th:nn_step_ii}. If $l=L-1$, notice that  $\Phi^{(\textup{(ii)},L-1)}$ is equivalent to $\Phi^{(\textup{(ii)},0)}$ and  continue with step~\ref{th:nn_step_iii}. 
				
				For $r\geq 0$, if $W_{l+1}^{(r)} \in \mathbb{R}^{m_{l+1} \times (n_{l}-r)}$ has not full rank, it has by assumption at least rank one and Lemma~\ref{lem:nn_notfullrank1} guarantees the existence of an MLP, which is equivalent to $\Phi^{(\textup{(ii)},l-1)}$, where the matrix $W_{l+1}^{(r)}$ is replaced by a matrix $W_{l+1}^{(r+1)} \in \mathbb{R}^{m_{l+1} \times (n_{l}-r-1)}$ with the same rank, $\widetilde{W}_{l}^{(r)}$ is replaced by a matrix $\widetilde{W}_{l}^{(r+1)} \in \mathbb{R}^{(n_{l}-r-1)\times m_{l}}$, the bias $\tilde{b}_{l}$ is replaced by a new bias and it holds $W_{l+1}^{(r+1)}\widetilde{W}_{l}^{(r+1)} = W_{l+1}^{(r)}\widetilde{W}_{l}^{(r)}$. As $\textup{rank}(W_{l+1}^{(r+1)}) = \textup{rank}(W_{l+1}^{(r)})$, Lemma~\ref{lem:nn_iterations} implies that there exist an index $r_{l}$, such that after $r_{l}$ applications of Lemma~\ref{lem:nn_notfullrank1}, the matrix $W_{l+1}^{(r_{l})} \in  \mathbb{R}^{m_{l+1} \times (n_{l}-r_{l})}$ has full rank. We denote the equivalent neural network with smaller architecture obtained after applying Lemma~\ref{lem:nn_notfullrank1} $r_{l}$ times to $\Phi^{(\textup{(ii)},l-1)}$ by $\Phi^{(\textup{(ii)},l)}$. $\Phi^{(\textup{(ii)},l)}$ has the matrix $W_{l+1}^{(0)}$ replaced by the full rank matrix  $W_{l+1}^{(r_{l})}\in  \mathbb{R}^{m_{l+1} \times (n_{l}-r_{l})}$, the matrix $\widetilde{W}_{l}^{(0)}$ replaced by a matrix $\widetilde{W}_{l}^{(r_{l})}\in  \mathbb{R}^{(n_{l}-r_{l}) \times m_{l}}$ and the bias $\tilde{b}_{l}$ replaced by a new bias. It holds $W_{l+1}^{(r_l)}\widetilde{W}_{l}^{(r_{l})} = W_{l+1}^{(0)}\widetilde{W}_{l}^{(0)}$ and as $\text{rank}(W_{l+1}^{(0)}\widetilde{W}_{l}^{(0)})>0$, the matrix $\widetilde{W}_{l}^{(r_{l})}$ has at least rank one. If $l<L-1$, increase the counter $l$ by one and repeat step~\ref{th:nn_step_ii}. If $l=L-1$, notice that  $\Phi^{(\textup{(ii)},L-1)}$ is equivalent to $\Phi^{(\textup{(ii)},0)}$ and continue with step~\ref{th:nn_step_iii}. 
				
				\item \label{th:nn_step_iii} If this is the first iteration of step~\ref{th:nn_step_iii}, denote the MLP $\Phi^{(\textup{(ii)},L-1)}(\mathcal{Y},\mathbb{R})\in \Xi^k(\mathcal{Y},\mathbb{R})$ obtained from~\ref{th:nn_step_ii} with weight matrices $(W_1^{(\ast)}, \widetilde{W}_1^{(r_1)},W_2^{(r_1)}, \ldots, W_L^{(r_{L-1})}, \widetilde{W}_L^{(0)})$ by $\Phi^{(\textup{(iii)},0)}$ and set $l = 1$. If $L = 1$, go to step~\ref{th:nn_step_iv}.
				
				If $\widetilde{W}_l^{(r_{l})}\in \mathbb{R}^{(n_l-r_{l}) \times m_{l}}$ has full rank, set $\tilde{r}_l = 0$, $\widetilde{W}_l^{(\ast)} = \widetilde{W}_l^{(r_{l})}$, $W_{l+1}^{(\ast)} = W_{l+1}^{(r_l)}$ and define $\Phi^{(\textup{(iii)},l)} = \Phi^{(\textup{(iii)},l-1)}$. If $l<L-1$, increase the counter $l$ by one and repeat step~\ref{th:nn_step_iii}. If $l=L-1$, notice that  $\Phi^{(\textup{(iii)},L-1)}$ is equivalent to $\Phi^{(\textup{(iii)},0)}$ and continue with step~\ref{th:nn_step_iv}. 
				
				For $r\geq 0$, if $\widetilde{W}_l^{(r_{l}+r)}\in \mathbb{R}^{(n_l-r_{l}-r) \times m_{l}}$ has not full rank, it has by step~\ref{th:nn_step_ii} at least rank one and Lemma~\ref{lem:nn_notfullrank2} guarantees the existence of an  MLP, which is equivalent to $\Phi^{(\textup{(iii)},l-1)}$, where the matrix $\widetilde{W}_l^{(r_{l}+r)}$ is replaced by a matrix $\widetilde{W}_l^{(r_{l}+r+1)} \in \mathbb{R}^{(n_l-r_{l}-r-1) \times m_l}$ with the same rank, the full rank matrix $W_{l+1}^{(r_l+r)}$ is replaced by a full rank matrix $W_{l+1}^{(r_l+r+1)} \in \mathbb{R}^{m_{l+1} \times (n_{l}-r_{l}-r-1)}$ and the bias $\tilde{b}_{l}$ is replaced by a new bias. As the rank of $\widetilde{W}_l^{(r_{l}+r+1)}$ is the same as the rank of $\widetilde{W}_l^{(r_{l}+r)}$, Lemma~\ref{lem:nn_iterations} implies that there exist an index $\tilde{r}_l$, such that after $\tilde{r}_l$ applications of Lemma~\ref{lem:nn_notfullrank2}, the matrix $\widetilde{W}_l^{(r_{l}+\tilde{r}_l)} \in  \mathbb{R}^{(n_l-r_{l}-\tilde{r}_l) \times m_l}$ has full rank. We denote the equivalent neural network obtained after applying Lemma~\ref{lem:nn_notfullrank2} $\tilde{r}_l$ times to $\Phi^{(\textup{(iii)},l-1)}$ by $\Phi^{(\textup{(iii)},l)}$. $\Phi^{(\textup{(iii)},l)}$ has the matrix $\widetilde{W}_l^{(r_{l})}$ replaced by the full rank matrix $\widetilde{W}_l^{(\ast)} \coloneqq \widetilde{W}_l^{(r_{l}+\tilde{r}_l)}\in  \mathbb{R}^{(n_l-r_{l}-\tilde{r}_l) \times m_l}$, the full rank matrix $W_{l+1}^{(r_l)}$ replaced by the full rank matrix $W_{l+1}^{(\ast)}\coloneqq W_{l+1}^{(r_l + \tilde{r}_l)}\in \mathbb{R}^{m_{l+1} \times (n_{l}-r_{l}-\tilde{r}_l)}$ and the bias $\tilde{b}_{l}$ replaced by a new bias. If $l<L-1$, increase the counter $l$ by one and repeat step~\ref{th:nn_step_iii}. If $l=L-1$, notice that  $\Phi^{(\textup{(iii)},L-1)}$ is equivalent to $\Phi^{(\textup{(iii)},0)}$ and continue with step~\ref{th:nn_step_iv}.
				
				\item \label{th:nn_step_iv} Denote the MLP $\Phi^{(\textup{(iii)},L-1)}\in  \Xi^k(\mathcal{Y},\mathbb{R})$ obtained from~\ref{th:nn_step_iii} with weight matrices $$W^\ast \coloneqq \left(W_1^{(\ast)}, \widetilde{W}_1^{(\ast)},W_2^{(\ast)}, \ldots, \widetilde{W}_{L-1}^{(\ast)}, W_L^{(\ast)}, \widetilde{W}_L^{(0)}\right)$$ by $\Phi^{(\textup{(iv)},0)}$. By step~\ref{th:nn_step_i}, the matrix $W_1^{(\ast)}$ has full rank and by steps~\ref{th:nn_step_ii} and~\ref{th:nn_step_iii}, all the matrices $\widetilde{W}_1^{(\ast)}$, $W_2^{(\ast)}$, $\ldots$, $\widetilde{W}_{L-1}^{(\ast)}$, $W_L^{(\ast)}$ have full rank. As $\widetilde{W}_L^{(0)} \in \mathbb{R}^{n_L \times m_L} = \mathbb{R}^{1\times m_L}$ and $\widetilde{W}_L^{(0)}$ has by assumption at least rank one, it holds $\text{rank}(\widetilde{W}_L^{(0)}) = 1$, such that $\widetilde{W}_L^{(\ast)} = \widetilde{W}_L^{(0)}$ has full rank. Consequently, all matrices of $W^\ast$ are full rank matrices, such that $\Phi^{(\textup{(iv)},0)}\in  \Xi^k_{\mathbb{W}^\ast}(\mathcal{Y},\mathbb{R})$. $\Phi^{(\textup{(iv)},0)}=\Phi^{(\textup{(iii)},L-1)}$ is equivalent to $\Phi^{(\textup{(iii)},0)} = \Phi^{(\textup{(ii)},L-1)}$, which is equivalent to $\Phi^{(\textup{(ii)},0)}= \Phi^{(\textup{(i)},r_0)}$, which is up to the linear change of coordinates $A$ defined in step~\ref{th:nn_step_i} equivalent to $\Phi^{(\textup{(i)},0)}= \Phi \in \Xi^k_{ \mathbb{W}_0}(\mathcal{X},\mathbb{R})$. The architecture $\Phi^{(\textup{(iv)},0)}$ is smaller than the architecture $\Phi$, as at least one weight matrix of $\Phi$ has not full rank, such that in at least one of the steps~\ref{th:nn_step_i},~\ref{th:nn_step_ii} and~\ref{th:nn_step_iii} the number of nodes in one of the layers $h_l$, $l \in \{0,\ldots,L-1\}$ is reduced by one, which implies the result.
			\end{enumerate}
			For case~\ref{th:nn_fullrank_equivalent_b} we consider an MLP $\Phi \in \Xi^k_{ \mathbb{W}_0}(\mathcal{X},\mathbb{R})$ with  weight matrices $W = (W_1, \widetilde{W}_1,W_2, \ldots, W_L, \widetilde{W}_L)$, biases $b = (b_1,\tilde{b}_1,\ldots,b_L,\tilde{b}_L)$ and assume that \mbox{$\textup{rank}(W_1)=0$} or $\textup{rank}(\widetilde{W}_L)=0$ or $\textup{rank}(W_{l+1} \widetilde{W}_l) = 0 $ for some $l\in \{1,\ldots,L-1\}$. Consequently, at least one of the vectors $h_0$ or $\sigma_l(a_l)$, $l\in \{1,\ldots,L\}$ is multiplied with a zero matrix, such that $\Phi(x) = c_{W,b}$ for all $x \in \mathcal{X}$, where $c_{W,b}$ is a scalar independent of $x$, only depending on $W$ and $b$. $\Phi$ is equivalent to the MLP $\overline{\Phi}\in \Xi^k(\mathcal{X},\mathbb{R})$ defined by $\overline{L} = 1$, $\overline{W}_1 = 0 \in \mathbb{R}^{1 \times n}$, $\overline{b}_1 = 0 \in \mathbb{R}$, $\overline{\widetilde{W}}_1 = 0 \in \mathbb{R}$ and $\overline{\tilde{b}}_1 = c_{W,b} \in \mathbb{R}$ as $\overline{\Phi}(x) = c_{W,b}$ for all $x \in \mathcal{X}$. If $\overline{\Phi} \neq \Phi$, $\overline{\Phi}$ has a smaller architecture than $\Phi$, as $\overline{\Phi}$ only consists of the $n$-dimensional input layer, the one-dimensional intermediate layer, and the one-dimensional output layer.
		\end{proof}
		
		\begin{remark}
			The rank of the weight matrices of the MLPs $\Phi$ in Theorem~\ref{th:nn_fullrank_equivalent} determines the widths of the layers of the equivalent neural network $\overline{\Phi}$. Hence, we refer to the term normal form of an MLP as the geometry of the smallest equivalent neural network, i.e., to the number and the widths of the layers. The proof of Theorem~\ref{th:nn_fullrank_equivalent} is constructive, but the chosen weights and biases of the equivalent neural network architectures are not unique. For a given neural network $\Phi$, there exists a unique geometry of the normal form architecture $\overline{\Phi}$, but multiple choices of the weights and biases, which define equivalent neural networks, are possible.
		\end{remark}

		In the following example, we illustrate the explicit algorithm of the proof of Theorem~\ref{th:nn_fullrank_equivalent}\ref{th:nn_fullrank_equivalent_a} for a two-layer neural network architecture.
		
			\begin{figure}[b!]
			\centering
			\begin{subfigure}[t]{0.9\textwidth}
				\centering
				\includegraphics[width=0.52\textwidth]{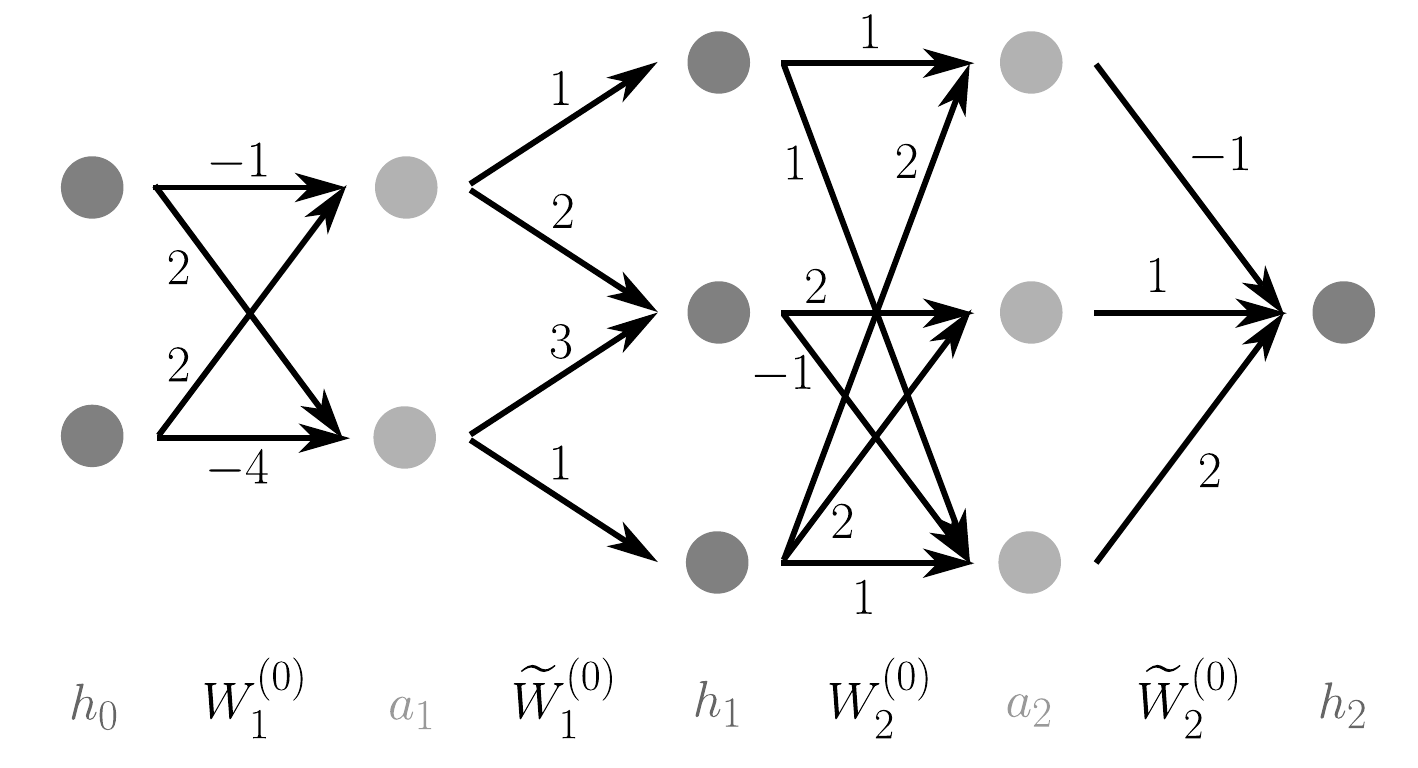}
				\caption{Example of a two-layer augmented neural network $\Phi \in \Xi^k_{\textup{A}, \mathbb{W}_0}(\mathcal{X},\mathbb{R})$, where the two weight matrices $W_1^{(0)}$ and $W_2^{(0)}$ have both not full rank.}
				\label{fig:nn_normalform_a}
			\end{subfigure}
			\begin{subfigure}[t]{0.9\textwidth}
				\centering
				\includegraphics[width=0.52\textwidth]{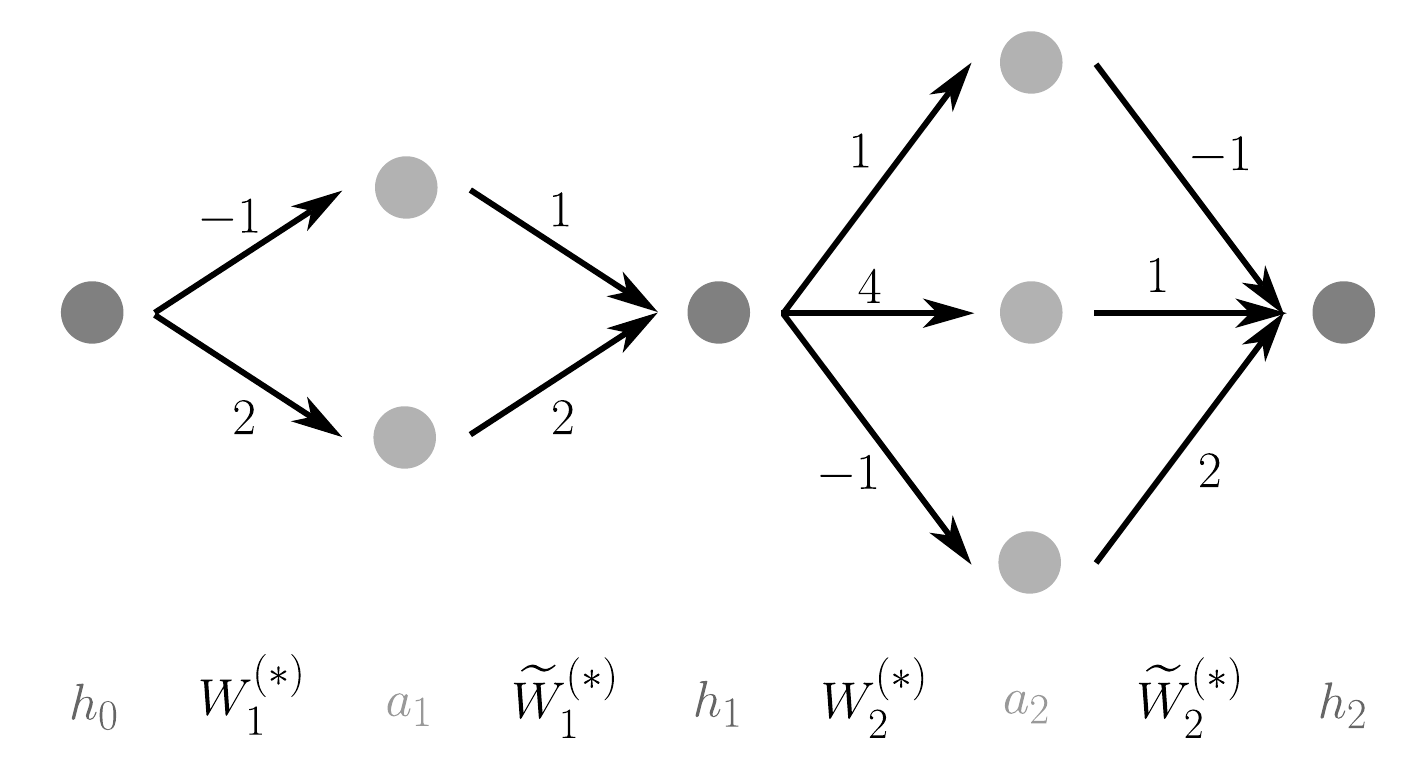}
				\caption{Two-layer neural network $\overline{\Phi} \in \Xi^k_{\textup{B},\mathbb{W}^\ast}(\mathcal{Y},\mathbb{R})$ with a bottleneck, which has only full rank matrices and which is up to a linear change of coordinates equivalent to the neural network $\Phi$ of  \subref{fig:nn_normalform_a}.}
				\label{fig:nn_normalform_b}
			\end{subfigure}
			\caption{Equivalent neural network architectures of Example~\ref{ex:nn_normalform} to illustrate Theorem~\ref{th:nn_fullrank_equivalent}\ref{th:nn_fullrank_equivalent_a}.}
			\label{fig:nn_normalform}
		\end{figure}

		\begin{example} \label{ex:nn_normalform}
			Consider the neural network $\Phi = \Phi^{(\textup{(i)},0)} \in \Xi^k_{\textup{A},  \mathbb{W}_0}(\mathcal{X},\mathbb{R})$ with $\mathcal{X} = (-1,1)\times (-1,1)\subset \mathbb{R}^2$ and $L = 2$ defined by the weight matrices
			\begin{align*}
				W_1^{(0)} = \begin{pmatrix}
					-1 & 2 \\ 2 & -4
				\end{pmatrix}, \quad \widetilde{W}_1^{(0)} = \begin{pmatrix}
					1 & 0 \\ 2 & 3 \\ 0 & 1
				\end{pmatrix}, \quad W_2^{(0)} = \begin{pmatrix}
					1 & 0 & 2 \\ 0 & 2 & 2 \\ 1 & -1 & 1
				\end{pmatrix}, \quad \widetilde{W}_2^{(0)} = \begin{pmatrix}
					-1 & 1 & 2
				\end{pmatrix},
			\end{align*}
			and arbitrary biases, see also Figure~\ref{fig:nn_normalform}\subref{fig:nn_normalform_a}. It holds $\textup{rank}(W_1^{(0)}) = 1<2$, $\textup{rank}(\widetilde{W}_1^{(0)})=2$, $\textup{rank}(W_2^{(0)})  =2<3$ and $\textup{rank}(\widetilde{W}_2^{(0)}) = 1$, such that $W^{(0)} = (W_1^{(0)},\widetilde{W}_1^{(0)},W_2^{(0)},\widetilde{W}_2^{(0)}) \in  \mathbb{W}_0$. As $\textup{rank}(W_1^{(0)})>0$, $\textup{rank}(W_2^{(0)} \widetilde{W}_1^{(0)}) = 1> 0$ and $\textup{rank}(\widetilde{W}_2^{(0)})>0$, Theorem~\ref{th:nn_fullrank_equivalent}\ref{th:nn_fullrank_equivalent_a} implies, that there exists a neural network~$\overline{\Phi}$, which is up to a linear change of coordinates equivalent to $\Phi$ and which has only full rank matrices. For step~\ref{th:nn_step_i} of the proof of Theorem~\ref{th:nn_fullrank_equivalent}\ref{th:nn_fullrank_equivalent_a}, we find $[W_1^{(0)}]_2 = -2 \cdot [W_1^{(0)}]_1$, such that
			\begin{equation*}
				A: \mathbb{R}^2 \rightarrow \mathbb{R},  \qquad A\cdot \begin{pmatrix}
					x_1 \\ x_2
				\end{pmatrix} = \begin{pmatrix}
					1 & -2
				\end{pmatrix}\cdot \begin{pmatrix}
					x_1 \\ x_2
				\end{pmatrix} = x_1 - 2x_2, \qquad \mathcal{Y} = A(\mathcal{X}) = (-3,3) \subset \mathbb{R},
			\end{equation*}
			and $\Phi^{(\textup{(i)},0)}$ is up to the linear change of coordinates $A$ equivalent to the neural network $\Phi^{(\textup{(i)},1)}$, where the matrix $W_1^{(0)}$ is replaced by the matrix $W_1^{(\ast)} = \begin{pmatrix} -1 \\ 2 \end{pmatrix}$. In step~\ref{th:nn_step_ii} Lemma~\ref{lem:nn_notfullrank1} is applied and as $[W_2^{(0)}]_3 = 2 \cdot [W_2^{(0)}]_1 + 1 \cdot [W_2^{(0)}]_2$, the matrices $\widetilde{W}_1^{(0)}$ and $W_2^{(0)}$ are replaced by
			\begin{equation*}
				\widetilde{W}_1^{(1)} =  \begin{pmatrix}
					1 + 2\cdot 0 & 0 + 2\cdot 1 \\ 2+ 1\cdot 0 & 3 + 1\cdot 1 \end{pmatrix} = \begin{pmatrix}
					1 & 2 \\ 2 & 4
				\end{pmatrix}, \qquad W_2^{(1)} = \begin{pmatrix}
					1 & 0  \\ 0 & 2  \\ 1 & -1 
				\end{pmatrix},
			\end{equation*}
			and the bias $\tilde{b}_1$ is adapted accordingly. The obtained architecture $\Phi^{(\textup{(ii)},1)}$ is equivalent to $\Phi^{(\textup{(i)},1)}$.  We continue with step~\ref{th:nn_step_iii} since $W_2^{(1)}$ is a full rank matrix, but $\textup{rank}(\widetilde{W}_1^{(1)}) = 1<2$. As $[[\widetilde{W}_1^{(1)}]^\top]_2 = 2 \cdot [[\widetilde{W}_1^{(1)}]^\top]_1$, the application of Lemma~\ref{lem:nn_notfullrank2} in step~\ref{th:nn_step_iii} yields
			\begin{equation*}
				\widetilde{W}_1^{(\ast)} = \widetilde{W}_1^{(2)} = (1,2), \qquad W_2^{(\ast)} = W_2^{(2)} = \begin{pmatrix}
					1 +2\cdot 0  \\ 0 +2\cdot 2  \\ 1 -2\cdot1 
				\end{pmatrix} = \begin{pmatrix} 1 \\ 4 \\ -1 \end{pmatrix},
			\end{equation*}
			and the bias $\tilde{b}_1$ is again adapted accordingly. The resulting architecture is denoted by $\Phi^{(\textup{(iii)},1)}$. As $\widetilde{W}_1^{(\ast)}$ and $ W_2^{(\ast)} $ have both full rank, we continue with step~\ref{th:nn_step_iv}, which implies with $\widetilde{W}_2^{(\ast)} = \widetilde{W}_2^{(0)}$ that $\Phi^{(\textup{(iv)},0)} =  \Phi^{(\textup{(iii)},1)}$ has only full rank matrices and is up to the linear change of coordinates $A$ equivalent to the neural network architecture  $\Phi=\Phi^{(\textup{(i)},0)}  \in \Xi^k_{\textup{A}, \mathbb{W}_0}(\mathcal{X},\mathbb{R})$. The equivalent normal form $\overline{\Phi} \coloneqq\Phi^{(\textup{(iv)},0)}\in \Xi_{\textup{B},\mathbb{W}^\ast}^k(\mathcal{Y},\mathbb{R})$ has a bottleneck and is visualized in Figure~\ref{fig:nn_normalform}\subref{fig:nn_normalform_b}.
		\end{example}

		Theorem~\ref{th:nn_fullrank_equivalent} is an explicit algorithm showing equivalence up to a linear change of coordinates. Due to the linear change of coordinates, it is not guaranteed that the regularity of the critical points of equivalent neural networks remains the same. The following theorem shows that if a coordinate transformation is necessary to obtain an equivalent neural network with full rank matrices, it is possible that degenerate critical points become non-degenerate.
		
		\begin{theorem} \label{th:nn_classes_coordinatechange}
			Let $\Phi \in \Xi^k(\mathcal{X},\mathbb{R})$, $k \geq 2$, $\mathcal{X} \subset \mathbb{R}^n$ be an MLP, which is up to the linear change of coordinates $A \in \mathbb{R}^{\bar{n}\times n}$, $\bar{n} \leq n$, $\textup{rank}(A) = \bar{n}$ equivalent to  $\overline{\Phi}\in \Xi^k(\mathcal{Y},\mathbb{R})$, $\mathcal{Y} = A(\mathcal{X})\subset \mathbb{R}^{\bar{n}}$. Then
			\begin{equation*}
				\overline{\Phi} \in (\mathcal{C}1)^k(\mathcal{Y},\mathbb{R})\quad \Leftrightarrow \quad  \Phi \in (\mathcal{C}1)^k(\mathcal{X},\mathbb{R}).
			\end{equation*}
			If $\bar{n} = n$, it holds additionally
			\begin{align*}
				\overline{\Phi} \in (\mathcal{C}2)^k(\mathcal{Y},\mathbb{R})\quad &\Leftrightarrow \quad \Phi \in (\mathcal{C}2)^k(\mathcal{X},\mathbb{R}), \\
				\overline{\Phi} \in (\mathcal{C}3)^k(\mathcal{Y},\mathbb{R})\quad &\Leftrightarrow \quad \Phi \in (\mathcal{C}3)^k(\mathcal{X},\mathbb{R}).
			\end{align*}
			If $\bar{n}<n$, then $ \Phi \notin(\mathcal{C}2)^k(\mathcal{X},\mathbb{R}) $ and it holds
			\begin{equation*}
				\overline{\Phi} \in(\mathcal{C}2)^k(\mathcal{X},\mathbb{R}) \cup (\mathcal{C}3)^k(\mathcal{Y},\mathbb{R}) 	\quad \Leftrightarrow \quad\Phi \in (\mathcal{C}3)^k(\mathcal{X},\mathbb{R}).
			\end{equation*}
		\end{theorem}
		
		\begin{proof}
			As $\Phi \in \Xi^k(\mathcal{X},\mathbb{R})$ is up to the linear change of coordinates $A$ equivalent to $\overline{\Phi}\in \Xi^k(\mathcal{Y},\mathbb{R})$, it holds
			\begin{equation*}
				\Phi(x) = \overline{\Phi}(Ax) \qquad \text{for all } x \in \mathcal{X}.
			\end{equation*}
			Calculating the gradient leads to 
			\begin{equation*}
				\nabla\Phi(x) = A^\top \cdot \nabla\overline{\Phi}(Ax) \qquad \text{for all } x \in \mathcal{X}.
			\end{equation*}
			If $\nabla\overline{\Phi}(Ax) = 0$ for some $x \in \mathcal{X}$ with $y = Ax \in \mathcal{Y}$, then  $\nabla\overline{\Phi}(Ax) = 0$ for all $x \in\mathcal{X}_y$, where
			\begin{equation*}
				\mathcal{X}_y \coloneqq \left\{x\in\mathcal{X}: Ax = y \in \mathcal{Y}\right\}.
			\end{equation*}
			Given $y \in \mathcal{Y}$, $\nabla\overline{\Phi}(y) = 0$ implies that  $\nabla \Phi(x) = 0$ for all $x \in \mathcal{X}_y$.
			As the matrix $A^\top \in \mathbb{R}^{n \times \bar{n}}$ has  rank $\bar{n}\leq n$, the linear system $A^\top \cdot \nabla\overline{\Phi}(y) = 0$ has only the solution $\nabla\overline{\Phi}(y) = 0$, such that from $\nabla \Phi(x) = 0$ for some $x \in \mathcal{X}$ it follows $ \nabla\overline{\Phi}(Ax) = 0$ for $Ax \in \mathcal{Y}$. Hence, there exists a bijection between the critical points $y \in \mathcal{Y}$ of $\overline{\Phi}$  and the set of critical points $\mathcal{X}_y\subset \mathcal{X}$ of   $\Phi$. Especially, $\overline{\Phi}$ has critical points if and only if $\Phi$ has critical points, such that the first assertion follows.
			
			To determine the regularity of the critical points, we calculate the Hessian matrices	 of $\Phi$ and $\overline{\Phi}$: 
			\begin{equation*}
				H_\Phi(x) = \nabla(\nabla\Phi(x)^\top) = \nabla( \nabla\overline{\Phi}(Ax)^\top \cdot A) = A^\top \cdot H_{\overline{\Phi}}(x) \cdot A  \qquad \text{for all } x \in \mathcal{X}.
			\end{equation*}
			As the matrix $A$ has rank $\bar{n}$ and $H_{\overline{\Phi}}(x)$ is a $\bar{n}\times \bar{n}$ matrix, it holds
			\begin{equation*}
				\text{rank}(H_\Phi(x)) = \text{rank}(  A^\top \cdot H_{\overline{\Phi}}(x) \cdot A) = \text{rank}(A^\top \cdot H_{\overline{\Phi}}(x)) =\text{rank}(H_{\overline{\Phi}}(x))  \qquad \text{for all } x \in \mathcal{X},
			\end{equation*}
			which implies with the fact that $H_{\Phi}(x)$ is a $n\times n$ matrix the remaining assertions of the theorem.
		\end{proof}
		
		The following example shows a neural network with degenerate critical points, which is up to a linear change of coordinates equivalent to a neural network with only non-degenerate critical points. It illustrates the case of Theorem~\ref{th:nn_classes_coordinatechange}, where under a linear change of coordinates represented by a full rank matrix $A$, the class of the neural network can change.
		
			\begin{figure}[b!]
			\centering
			\begin{subfigure}[t]{0.04\textwidth}
				\textcolor{white}{.}
			\end{subfigure}
			\begin{subfigure}[t]{0.42\textwidth}
				\centering
				\includegraphics[width=0.8\textwidth]{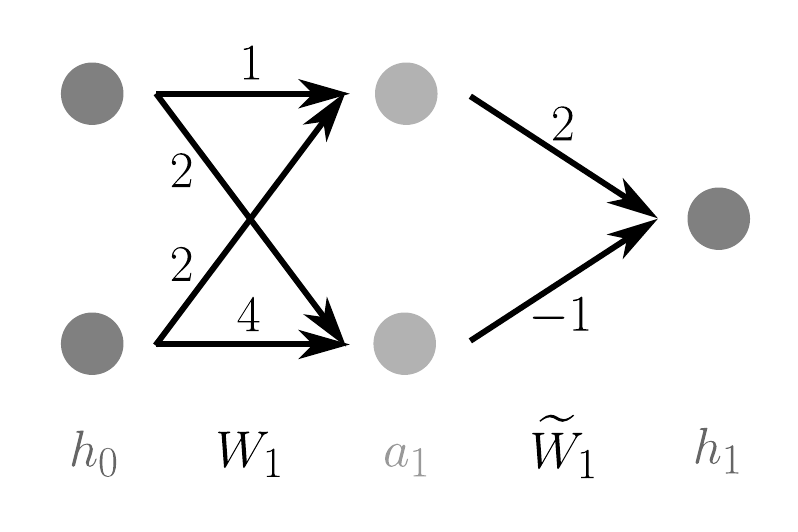}
				\caption{Example of a one-layer neural network $\Phi \in \Xi^\infty_{\mathbb{W}_0}(\mathbb{R}^2,\mathbb{R})$ of class $(\mathcal{C}3)^\infty(\mathbb{R}^2,\mathbb{R})$, which has one line of degenerate equilibria.}
				\label{fig:nn_example_classes_a}
			\end{subfigure}
			\begin{subfigure}[t]{0.04\textwidth}
				\textcolor{white}{.}
			\end{subfigure}
			\begin{subfigure}[t]{0.42\textwidth}
				\centering
				\includegraphics[width=0.8\textwidth]{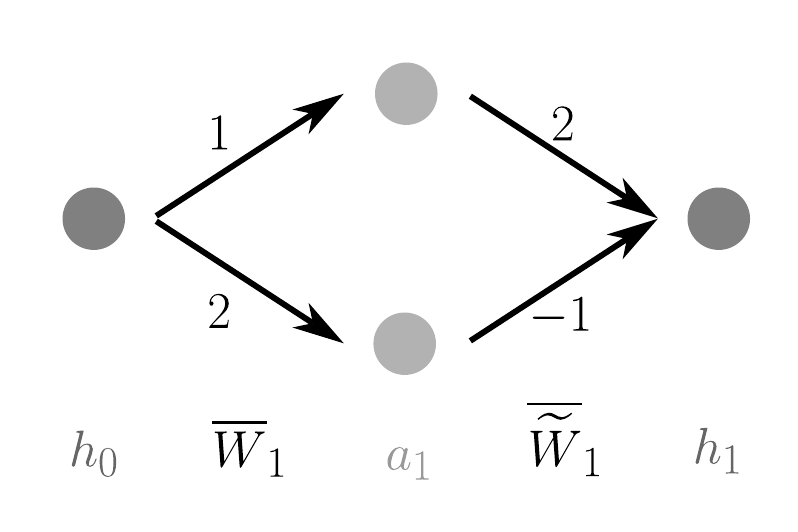}
				\caption{One-layer neural network $\overline{\Phi} \in \Xi^\infty_{\mathbb{W}^\ast}(\mathbb{R},\mathbb{R})$ of class $(\mathcal{C}2)^\infty(\mathbb{R},\mathbb{R})$, which is up to a linear change of coordinates equivalent to $\Phi$ of  \subref{fig:nn_example_classes_a}. $\overline{\Phi}$ has one non-degenerate critical point.}
				\label{fig:nn_example_classes_b}
			\end{subfigure}
			\begin{subfigure}[t]{0.04\textwidth}
				\textcolor{white}{.}
			\end{subfigure}
			\caption{Equivalent neural network architectures of Example~\ref{ex:nn_classes} to illustrate Theorem~\ref{th:nn_classes_coordinatechange}.}
			\label{fig:nn_example_classes}
		\end{figure}
		
		\begin{example} \normalfont \label{ex:nn_classes}
			Consider the neural network $\Phi \in \Xi^\infty_{\mathbb{W}_0}(\mathbb{R}^2,\mathbb{R})$ with $L = 1$ defined by the weight matrices
			\begin{equation*}
				W_1 = \begin{pmatrix}
					1 & 2 \\ 2 & 4
				\end{pmatrix}, \qquad \widetilde{W}_1 = \begin{pmatrix}
					2 & -1
				\end{pmatrix}, \qquad b_1 = \begin{pmatrix}
					0 \\ 0
				\end{pmatrix}, \qquad \tilde{b}_1 = 0,
			\end{equation*}
			arbitrary biases and soft-plus activation functions $[\sigma_1]_1(x) = [\sigma_1]_2(x) = \ln(1+\exp(x))$, see Figure~\ref{fig:nn_example_classes}\subref{fig:nn_example_classes_a}. As $\textup{rank}(W_1) = 1<2$ and $\textup{rank}(\widetilde{W}_1) = 1$, Theorem~\ref{th:nn_fullrank_equivalent}\ref{th:nn_fullrank_equivalent_a} implies that there exists a neural network architecture $\overline{\Phi} \in \Xi^\infty_{\mathbb{W}^\ast}(\mathbb{R},\mathbb{R})$, which is up to a linear change of coordinates $A:\mathbb{R}^2\rightarrow \mathbb{R}$ equivalent to $\Phi$. As $[W_1]_2 = 2 [W_1]_1$, it follows by the proof of Theorem~\ref{th:nn_fullrank_equivalent}\ref{th:nn_fullrank_equivalent_a}  that the weights of $\overline{\Phi}$ and the matrix $A$ are given by
			\begin{equation*}
				\overline{W}_1 = \begin{pmatrix}
					1 \\ 2
				\end{pmatrix}, \qquad \overline{\widetilde{W}}_1 = \begin{pmatrix}
					2 & -1
				\end{pmatrix}, \qquad \bar{b}_1 = \begin{pmatrix}
					0 \\ 0
				\end{pmatrix}, \qquad \bar{\tilde{b}}_1 = 0, \qquad A = \begin{pmatrix}
					1 & 2
				\end{pmatrix},
			\end{equation*}
			such that $\Phi(x_1,x_2) = \overline{\Phi}(x_1+2x_2) = \Phi(y)$ for $y = x_1 + 2x_2$. We verify now the properties of Theorem~\ref{th:nn_classes_coordinatechange}. It holds
			\begin{equation*}
				\overline{\Phi}(y) = 2\ln(1+\exp(y))-\ln(1+\exp(2y))
			\end{equation*}
			with gradient
			\begin{equation*}
				\nabla \overline{\Phi}(y) = \frac{2\exp(y)(1-\exp(y))}{(1+\exp(y))(1+\exp(2y))},
			\end{equation*}
			such that $\overline{\Phi}$ has exactly one critical point at $y^\ast = 0$. The critical point is non-degenerate as
			\begin{equation*}
				H_{\overline{\Phi}}(y^\ast) = -0.5 \neq 0.  
			\end{equation*}
			Consequently $\overline{\Phi}$ is of class $(\mathcal{C}2)^\infty(\mathbb{R},\mathbb{R})$. By Theorem~\ref{th:nn_classes_coordinatechange} the original neural network $\Phi$ has to be of class $(\mathcal{C}3)^\infty(\mathbb{R}^2,\mathbb{R})$. It holds
			\begin{equation*}
				\Phi(x_1,x_2) = 2\ln(1+\exp(x_1+2x_2))-\ln(1+\exp(2x_1+4x_2)),
			\end{equation*}
			with gradient
			\begin{equation*}
				\nabla \overline{\Phi}(x_1,x_2) = 	\frac{\exp(x_1+2x_2)(1-\exp(x_1+2x_2))}{(1+\exp(x_1+2x_2))(1+\exp(2x_1+4x_2))} \cdot \begin{pmatrix}
					2 \\4
				\end{pmatrix} ,
			\end{equation*}
			such that $\Phi$ has a line of equilibria defined by $x_1 + 2x_2 = 0$. As the entries of $\nabla \overline{\Phi}(x_1,x_2)$ are linearly dependent, the Hessian matrix $H_\Phi(x_1,x_2)$ is for every $(x_1,x_2)\in\mathbb{R}^2$ singular, which implies that all equilibria of $\Phi$ are degenerate and $\Phi$ is of class  $(\mathcal{C}3)^\infty(\mathbb{R}^2,\mathbb{R})$.
		\end{example}

		If we aim to have a large expressivity of MLPs with respect to the space of  $k$ times continuously differentiable functions, Theorem~\ref{th:nn_classes_coordinatechange} implies that we should choose the input dimension as small as possible. If the input dimension $n$ can be reduced, the considered architecture cannot be of class $(\mathcal{C}2)^k(\mathcal{X},\mathbb{R})$, so at least one critical point is degenerate. Consequently, the generic class of Morse functions, where every critical point is non-degenerate, cannot be completely represented.

		Due to Lemma~\ref{lem:nn_weightspace} and Theorems~\ref{th:nn_fullrank_equivalent} and~\ref{th:nn_classes_coordinatechange}, we assume in the upcoming analysis, that the considered neural network architecture has only full rank matrices, as every neural network with non-trivial dynamics is up to a linear change of coordinates equivalent to a smaller architecture in normal form and the set of matrices, where at least one weight matrix has not full rank, is a zero set in the weight space.

		\subsection{Existence of Critical Points}
		\label{sec:nn_criticalpoints}
		
		In this section, we study the existence of critical points dependent on the special architecture of the neural network. To that purpose, we first calculate the network gradient.
		
		\begin{lemma}[MLP Network Gradient] \label{lem:nn_gradient}
			Let $\Phi \in \Xi^k(\mathcal{X},\mathbb{R})$, $k \geq 1$ be an MLP with weight matrices $W = (W_1, \widetilde{W}_1, \ldots, W_L, \widetilde{W}_L) \in \mathbb{W}$ and biases $b = (b_1,\tilde{b}_1,\ldots,b_L,\tilde{b}_L) \in \mathbb{B}$. Then 
			\begin{equation*}
				\nabla \Phi(x) = \left[\widetilde{W}_L\Psi_L(a_{L})W_{L} \ldots \widetilde{W}_1\Psi_1(a_{1})W_1\right]^\top \in \mathbb{R}^{n},
			\end{equation*}
			where $\Psi_l(a_{l}) = \textup{diag}([\sigma_l]_i'([a_l]_i)) \in \mathbb{R}^{m_{l}\times m_{l}}$, $i \in \{1,\ldots,m_l\}$ is a diagonal matrix.
		\end{lemma}
		
		\begin{proof}
			Due to the layer structure of the MLP it holds for $l \in \{1,.\ldots,L-1\}$:
			\begin{equation*}
				\Phi(x) = h_L = \widetilde{W}_L\sigma_L(W_L(\ldots\sigma_2(W_2(\widetilde{W}_1\sigma_1(W_1x+b_1)+\tilde{b}_1)+b_2)\ldots)+b_L)+\tilde{b}_L.
			\end{equation*}
			By Lemma~\ref{lem:nn_regularity} it holds $\Phi \in C^1(\mathcal{X},\mathbb{R})$, such that the multi-dimensional chain rule implies 
			\begin{equation*}
				\frac{\dd \Phi}{\dd x} = \widetilde{W}_L\Psi_L(a_{L})W_{L} \ldots \widetilde{W}_1\Psi_1(a_{1})W_1 \in \mathbb{R}^{1\times n},
			\end{equation*}
			where $\Psi_l(a_{l}) = \textup{diag}([\sigma_l]_i'([a_l]_i)) \in \mathbb{R}^{m_{l}\times m_{l}}$, $i \in \{1,\ldots,m_l\}$ is a diagonal matrix. The result follows by taking the transpose. 
		\end{proof}
		
		Due to the MLP normal form derived in Section~\ref{sec:nn_weightmatrices}, we restrict our upcoming analysis to neural networks $\Phi \in \Xi^k_{\mathbb{W}^\ast}(\mathcal{X},\mathbb{R})$ with full rank weight matrices. The upcoming theorems established for networks in $\Xi^k_{\mathbb{W}^\ast}(\mathcal{X},\mathbb{R})$ can be generalized to general MLPs $\Xi^k(\mathcal{X},\mathbb{R})$ in the following way: given an MLP $\Phi \in \Xi^k_{\mathbb{W}_0}(\mathcal{X},\mathbb{R})$, where at least one weight matrix has not full rank, we can apply the algorithm of Theorem~\ref{th:nn_fullrank_equivalent} and find an MLP  $\overline{\Phi} \in \Xi^k_{\mathbb{W}^\ast}(\mathcal{X},\mathbb{R})$, which is (up to a linear change of coordinates) equivalent to $\Phi$. Hence, the input-output dynamics of $\Phi$ can be analyzed via the input-output dynamics of the normal form MLP $\overline{\Phi}$.
			
		For non-augmented MLPs $\Phi \in \Xi^k_{\textup{N},\mathbb{W}^\ast}(\mathcal{X},\mathbb{R})$, the following theorem shows that $\Phi$ has no critical points if all weight matrices have full rank. Using the statement of Lemma~\ref{lem:nn_weightspace}, that $\vem(\mathbb{W}_0)$ is a zero set in $\mathbb{R}^{N_W}$, the result of the upcoming theorem can also be stated for all MLPs $\Phi \in \Xi^k_{\textup{N}}(\mathcal{X},\mathbb{R})$ by adding the restriction, that the statement holds except possibly for a Lebesgue measure zero set in the weight space.
		
		\begin{theorem} \label{th:nn_nocriticalpoints}
			Any non-augmented MLP $\Phi \in \Xi^k_{\textup{N},\mathbb{W}^\ast}(\mathcal{X},\mathbb{R})$, $k \geq 1$, with weight matrices $W\in \mathbb{W}^\ast$ has no critical point, i.e., $\nabla \Phi(x) \neq 0$ for all $x \in \mathcal{X}$. Consequently it holds for $k \geq 1$ $$ \Xi^k_{\textup{N},\mathbb{W}^\ast}(\mathcal{X},\mathbb{R}) \subset (\mathcal{C}1)^k (\mathcal{X},\mathbb{R}).$$
			This implies that for all weights $(\vem(W);\stack(b)) \in \mathbb{R}^{N_W+N_B} $, $W \in \mathbb{W}$, $b \in \mathbb{B}$ except possibly for a zero set in $\mathbb{R}^{N_W+N_B}$ with respect to the Lebesgue measure, a non-augmented MLP $\Phi \in \Xi^k_{\textup{N}}(\mathcal{X},\mathbb{R})$, $k\geq 1$ is of class $(\mathcal{C}1)^k (\mathcal{X},\mathbb{R})$ and hence a Morse function. 
		\end{theorem}
		
		\begin{proof}
			Let $\Phi \in \Xi^k_{N,\mathbb{W}^\ast}(\mathcal{X},\mathbb{R})$ with weight matrices $W = (W_1, \widetilde{W}_1, \ldots, W_L, \widetilde{W}_L)$. By Lemma~\ref{lem:nn_gradient} it holds 
			\begin{equation*}
				\nabla \Phi(x) = \left[\widetilde{W}_L\Psi_L(a_{L})W_{L} \ldots \widetilde{W}_1\Psi_1(a_{1})W_1\right]^\top \in \mathbb{R}^{n},
			\end{equation*}
			where $\Psi_l(a_{l}) = \textup{diag}([\sigma_l]_i'([a_l]_i)) \in \mathbb{R}^{m_{l}\times m_{l}}$, $i \in \{1,\ldots,m_l\}$, is a diagonal matrix with $a_l = W_lh_{l-1}+b_l$. As $[\sigma_l]_i\in C^k(\mathbb{R},\mathbb{R})$, $k \geq 1$  is by Definition~\ref{def:feedforward} strictly monotone for $l \in \{1,\ldots,L\}$, $i \in \{1,\ldots,m_{l}\}$, it either holds $[\sigma_l]_i'>0$ or $[\sigma_l]_i'<0$. Consequently each matrix $\Psi_l(a_{l})$, $l \in \{1,\ldots,L\}$ has non-zero diagonal entries and hence full rank. Due to the assumption $W\in \mathbb{W}^\ast$ all matrices $W_l$ and $\widetilde{W}_l$ have full rank for $l \in \{1,\ldots,L\}$. As the neural network is non-augmented, the dimensions of the full rank matrices $\widetilde{W}_L,\Psi_L(a_{L}), \ldots,\Psi_1(a_{1}), W_1$ are monotonically increasing, such that Lemma~\ref{lem:nn_productfullrank} implies that the vector $[\nabla \Phi(x)]^\top \in \mathbb{R}^{1 \times n}$ has independently of $x \in \mathcal{X}$ always rank one. Consequently $\nabla \Phi(x) \neq 0$ for all $x \in \mathcal{X}$. By Definition~\ref{def:subsets} it follows $ \Xi^k_{\textup{N},\mathbb{W}^\ast}(\mathcal{X},\mathbb{R}) \subset (\mathcal{C}1)^k (\mathcal{X},\mathbb{R})$. The second statement is a consequence of the first part and Lemma~\ref{lem:nn_weightspace}: as $\vem( \mathbb{W}_0)$ is a zero set with respect to the Lebesgue measure in $\mathbb{R}^{N_W}$, it follows that the set of non-full rank matrices and arbitrary biases $\vem( \mathbb{W}_0) \times \stack(\mathbb{B})$ is a zero set in $\mathbb{R}^{N_W+N_B}$ with respect to the Lebesgue measure. 
		\end{proof}
		
		\begin{example} \label{ex:nn_nonaugmented}
			\normalfont Consider the one-layer neural network architecture $\Phi \in \Xi^k_\textup{N}(\mathbb{R}^3,\mathbb{R})$, $k\geq 1$ with $L = 1$ defined by the weight matrices
			\begin{equation*}
				W_1 = \begin{pmatrix}
					w_{11} & w_{12} & w_{13} \\
					w_{21} & w_{22} & w_{23}
				\end{pmatrix} \in \mathbb{R}^{2 \times 3}, \qquad \widetilde{W}_1 = (\widetilde{w}_1,\widetilde{w}_2) \in \mathbb{R}^{1\times 2},
			\end{equation*}
			and arbitrary biases, see Figure~\ref{fig:nn_nonaugmented}. By Lemma~\ref{lem:nn_gradient} it holds
			\begin{align*}
				\nabla \Phi(x) &= \begin{pmatrix}
					w_{11} & w_{21} \\
					w_{12} & w_{22} \\
					w_{13} & w_{23}
				\end{pmatrix}
				\begin{pmatrix}
					[\sigma_1]_1'([a_1]_1) & 0 \\
					0 & [\sigma_1]_2'([a_1]_2)
				\end{pmatrix}
				\begin{pmatrix}
					\widetilde{w}_1 \\ \widetilde{w}_2
				\end{pmatrix}\\ &= \begin{pmatrix}
					\widetilde{w}_1 w_{11} [\sigma_1]_1'([a_1]_1) + \widetilde{w}_2 w_{21}  [\sigma_1]_2'([a_1]_2)\\
					\widetilde{w}_1 w_{12} [\sigma_1]_1'([a_1]_1) + \widetilde{w}_2 w_{22}  [\sigma_1]_2'([a_1]_2)\\
					\widetilde{w}_1 w_{13} [\sigma_1]_1'([a_1]_1) + \widetilde{w}_2 w_{23}   [\sigma_1]_2'([a_1]_2)
				\end{pmatrix}.
			\end{align*}
			By Theorem~\ref{th:nn_nocriticalpoints}, there exists no critical point if $W_1$ and $\widetilde{W}_1$ have both full rank. This is the case, as if $\nabla \Phi(x) = 0$ for some $x \in \mathcal{X}$, then  
			\begin{equation*}
				\widetilde{w}_1 [\sigma_1]_1'([a_1]_1)	\begin{pmatrix}
					w_{11} \\ w_{12} \\ w_{13}
				\end{pmatrix}+\widetilde{w}_2 [\sigma_1]_2'([a_1]_2) 
				\begin{pmatrix}
					w_{21} \\ w_{22} \\ w_{23}
				\end{pmatrix} = \begin{pmatrix}
					0 \\ 0 \\ 0
				\end{pmatrix},
			\end{equation*}
			such that either $W_1$ has only rank one as its columns are linearly dependent or the matrix $\widetilde{W}_1$ is the zero matrix and has rank zero.
		\end{example}
		
		\begin{figure}[b!]
			\centering
			\includegraphics[width=0.28\textwidth]{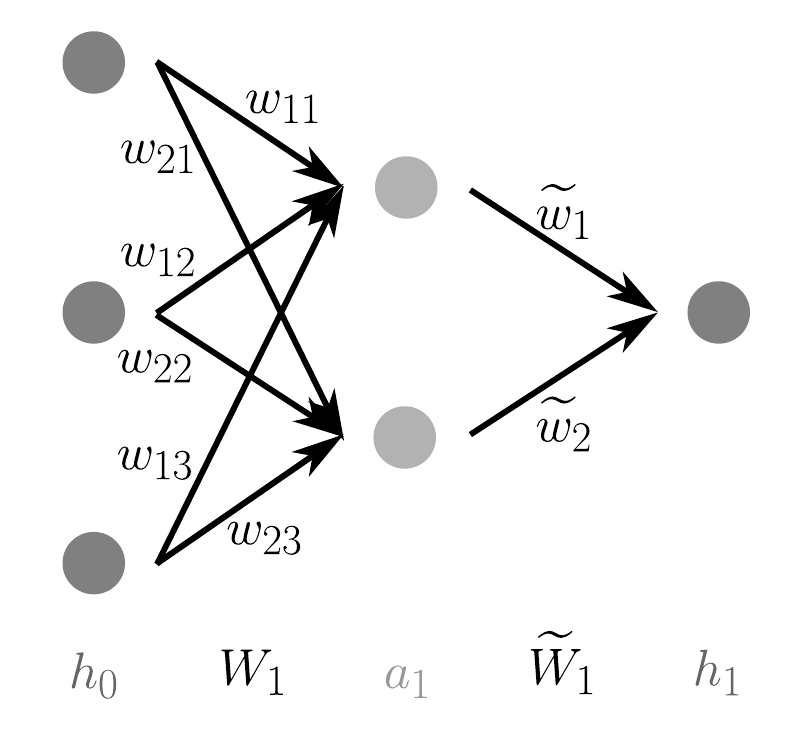}
			\caption{One-layer non-augmented neural network architecture $\Phi \in \Xi^k_\textup{N}(\mathbb{R}^3,\mathbb{R})$ of Example~\ref{ex:nn_nonaugmented}.}
			\label{fig:nn_nonaugmented}
		\end{figure}

		In contrast to non-augmented neural networks, critical points can exist in the case of augmented or bottleneck architectures. The following proposition shows, that for augmented and bottleneck architectures in normal form, there exist for every $x\in \mathcal{X}$ a choice of the full rank matrices $W \in \mathbb{W}^\ast$, such that the neural network has a critical point at $x \in \mathcal{X}$. For the explicit construction in the following theorem, results from linear algebra stated in Appendix \ref{sec:app_LA} are used.		
		
		\begin{theorem} \label{th:nn_criticalpoints}
			Given $L\geq 1$ and $\mathcal{X}\subset \mathbb{R}^n$, for every $x \in \mathcal{X}$ there exist weight matrices $W_x\in \mathbb{W}^\ast$ corresponding to an augmented architecture $\Phi \in \Xi^k_{\textup{A},\mathbb{W}^\ast}(\mathcal{X},\mathbb{R})$ and weight matrices $W_x\in \mathbb{W}^\ast$ corresponding to a bottleneck architecture $\Phi \in \Xi^k_{\textup{B},\mathbb{W}^\ast}(\mathcal{X},\mathbb{R})$, $k \geq 1$,  such that $\nabla \Phi(x) = 0$, i.e., $\Phi$ has a critical point at $x \in \mathcal{X}$.
		\end{theorem}
		
		\begin{proof}
			For $\Phi \in \Xi^k_{\textup{A}}(\mathcal{X},\mathbb{R})$, the first layer of maximal width $g_{l^\ast}$, $l^\ast \in \{1,\ldots,2L-1\}$ fulfills $d_{l^\ast-1} < d_{l^\ast}$ and it holds $d_{j-1} \geq d_j$ for $l^\ast < j \leq 2L$. For $\Phi \in \Xi^k_{\textup{B}}(\mathcal{X},\mathbb{R})$, consider the last bottleneck of $\Phi$. Then there exist layers $g_i^\ast$, $g_j^\ast$ and $g_l^\ast$ with $0 \leq i^\ast < j^\ast < l^\ast \leq 2L$, such that $d_i^\ast>d_j^\ast$, $d_j^\ast\leq d_{j^\ast +1} \leq \ldots \leq d_{l^\ast-1} <d_l^\ast$ and $d_{j-1} \geq d_j$ for $l^\ast < j \leq 2L$. In both cases there exists a layer $g_{l^\ast}$ such that $d_{l^\ast-1}<d_{l^\ast}$, $d_{l^\ast}\geq d_{l^\ast+1}$ and from layer $g_{l^\ast}$ on-wards the width of the layers is monotonically decreasing.
			
			Let $\Phi \in \Xi^k_{\textup{A}}(\mathcal{X},\mathbb{R}) \cup \Xi^k_{\textup{B}}(\mathcal{X},\mathbb{R})$ be an augmented or bottleneck MLP, fix some $x \in \mathcal{X}$ and denote the weight matrices by $W =(W_1, \widetilde{W}_1, \ldots, W_L, \widetilde{W}_L) = (V_1,\ldots,V_{2L})\in \mathbb{W}$, where $V_1,\ldots,V_{l^\ast},V_{l^\ast+2},\ldots,V_{2L}$ are arbitrary but fixed full rank matrices. As $d_{l^\ast-1} < d_{l^\ast}$, the corresponding weight matrix $V_{l^\ast} \in \mathbb{R}^{d_{l^\ast} \times d_{l^\ast -1}}$ has rank $d_{l^\ast-1}$, where $V_{l^\ast} = \widetilde{W}_{l^\ast/2}$ if $l^\ast$ is even and $V_{l^\ast} = W_{(l^\ast+1)/2}$ if $l^\ast$ is odd. As $d_{l^\ast}>d_{l^\ast-1}$ and $\text{rank}(V_{l^\ast}) = d_{l^\ast-1}$, the $d_{l^\ast}$ rows of $V_{l^\ast}$ are linearly dependent. Hence, there exists a non-trivial linear combination of the rows of $V_{l^\ast}$, which results in a zero row, i.e., there exists $v_{l^\ast} \in \mathbb{R}^{1 \times d_{l^\ast}}$, $\text{rank}(v_{l^\ast}) = 1$, such that
			\begin{equation*}
				v_{l^\ast} V_{l^\ast} = 0 \in \mathbb{R}^{1 \times d_{l^\ast-1}}.
			\end{equation*}
			If $l^\ast$ is even, then
			\begin{equation*}
				(\nabla \Phi(x))^\top = V_{2L}\Psi_L(a_{L})V_{2L-1} \ldots \Psi_{l^\ast/2+1}(a_{l^\ast/2+1}) V_{l^\ast+1} V_{l^\ast} \ldots V_2\Psi_1(a_{1})V_1 \in \mathbb{R}^{1 \times n}.
			\end{equation*}
			As $V_{2L}$, $\Psi_L(a_{L})$, \ldots, $\Psi_{l^\ast/2+1}(a_{l^\ast/2+1}) $ are full rank matrices with monotonically increasing width, Lemma~\ref{lem:nn_productfullrank} implies that the matrix product $V_{2L} \Psi_L(a_{L}) \ldots \Psi_{l^\ast/2+1}(a_{l^\ast/2+1}) \in \mathbb{R}^{1 \times d_{l^\ast +1}} $ has rank one, independently of the choice of $x \in \mathcal{X}$. Since $d_{l^\ast}\geq d_{l^\ast+1}$, Lemma~\ref{lem:nn_linearsystem} shows via an explicit construction that there exist a full rank solution $V_{l^\ast +1} \in \mathbb{R}^{d_{l^\ast+1}\times d_{l^\ast}}$ of the linear system
			\begin{equation*}
				\left[V_{2L}\ldots \Psi_{l^\ast/2+1}(a_{l^\ast/2+1})\right] V_{l^\ast +1} = v_{l^\ast} \in \mathbb{R}^{1 \times d_{l^\ast}}.
			\end{equation*}
			Hence, there exist weight matrices $W_x \in \mathbb{W}^\ast$, which depend on $x$, such that $\nabla \Phi(x) = 0 \in \mathbb{R}^n$ for the fixed $x \in \mathcal{X}$. If $l^\ast$ is odd, then
			\begin{equation*}
				(\nabla \Phi(x))^\top = V_{2L}\Psi_L(a_{L})V_{2L-1} \ldots V_{l^\ast+2}V_{l^\ast+1} \Psi_{(l^\ast+1)/2}(a_{(l^\ast+1)/2})  V_{l^\ast} \ldots V_2\Psi_1(a_{1})V_1 \in \mathbb{R}^{1 \times n}.
			\end{equation*}
			As $V_{2L}$, \ldots, $V_{l^\ast+2}$ are full rank matrices with monotonically increasing width, Lemma~\ref{lem:nn_productfullrank} guarantees that the matrix product $\widetilde{W}_L \ldots V_{l^\ast+2} \in \mathbb{R}^{1 \times d_{l^\ast +1}} $ has rank one, independently of the choice of $x \in \mathcal{X}$. As $\Psi_{(l^\ast+1)/2}(a_{(l^\ast+1)/2}) $ is a diagonal, invertible full rank matrix, $v_{l^\ast} \left[\Psi_{(l^\ast+1)/2}(a_{(l^\ast+1)/2})\right]^{-1}$ has rank one. Since $d_{l^\ast}\geq d_{l^\ast+1}$, Lemma~\ref{lem:nn_linearsystem} implies that there exist a full rank solution $V_{l^\ast +1} \in \mathbb{R}^{d_{l^\ast+1}\times d_{l^\ast}}$ of the linear system 
			\begin{equation*}
				\left[V_{2L}\ldots V_{l^\ast+2}\right] V_{l^\ast+1}= v_{l^\ast} \left[\Psi_{(l^\ast+1)/2}(a_{(l^\ast+1)/2})\right]^{-1}\in \mathbb{R}^{1 \times d_{l^\ast}}.
			\end{equation*}
			Hence there exist weight matrices $W_x \in \mathbb{W}^\ast$, which depend on $x$, such that $\nabla \Phi(x) = 0 \in \mathbb{R}^n$ for the fixed $x \in \mathcal{X}$. As $\Phi$ was an augmented or a bottleneck neural network, the result holds for both architectures.
		\end{proof}
		
		Examples of augmented and bottleneck architectures, which have critical points, are shown in the upcoming Examples~\ref{ex:nn_augmented} and~\ref{ex:nn_bottleneck}, when the regularity of the critical points is analyzed. The last theorem has shown that for every $x \in \mathcal{X}$, the weight matrices of augmented neural networks and neural networks with a bottleneck can be chosen in such a way that $x \in \mathcal{X}$ is a critical point. From this statement, we cannot deduce if it is a generic property of augmented and bottleneck architectures to have a critical point or not. In the following, we show that if there exists for a fixed architecture a non-degenerate critical point, then the set of weights, such that the corresponding neural network has a non-degenerate critical point, has non-zero Lebesgue measure in the weight space. Hence, if a non-degenerate critical point exists for one set of weights, it is locally a generic property of the chosen architecture to have at least one non-degenerate critical point.
		
		\begin{theorem} \label{th:nn_criticalpoint_fulllebesgue}
			Let $\Phi \in \Xi^k(\mathcal{X},\mathbb{R})$, $k \geq 2$ be an MLP with weight matrices $W^\ast \in \mathbb{W}$. If $\Phi$ has a non-degenerate critical point $x^\ast \in \mathcal{X}$, then there exist an open neighborhood $\mathcal{V}$ of $w^\ast = \vem(W^\ast) \in \mathbb{R}^{N_W}$, which has non-zero Lebesgue measure in $\mathbb{R}^{N_W}$, such that for every $w \in \mathcal{V}$, the corresponding neural network has at least one critical point. By choosing $\mathcal{V}$ small enough, it can be guaranteed that at least one critical point of the corresponding neural network is non-degenerate.
		\end{theorem}
		
		\begin{proof}
			Let $w^\ast = \vem(W^\ast) \in \mathbb{R}^{N_W}$ be the vector of stacked weight matrices of the neural network~$\Phi$. Define the function
			\begin{equation*}
				F:  \mathcal{X} \times \mathbb{R}^{N_W} \rightarrow \mathbb{R}^n, \qquad (x,w) \mapsto F(x,w) = \nabla \Phi(x)
			\end{equation*}
			with $\mathcal{X}\subset \mathbb{R}^n $ open  and $\nabla \Phi(x)$ defined in Lemma~\ref{lem:nn_gradient}, which also depends on $w  \in \mathbb{R}^{N_W}$. As $k\geq 2$, $\Phi$ is a composition of functions, which are at least twice differentiable in $x$ and $w$, such that $F$ is a continuously differentiable function. By assumption, the neural network $\Phi$ has a critical point at $x^\ast \in \mathcal{X}$, such that $F(x^\ast,w^\ast) = 0$. The $n \times n$ matrix
			\begin{equation*}
				\frac{\partial F}{\partial x} = \frac{\nabla \Phi(x)}{\partial x} = H_{\Phi}(x)
			\end{equation*}
			is non-singular at the point $(x^\ast,w^\ast)$ as by assumption $x^\ast$ is a non-degenerate critical point. The Implicit Function Theorem (cf.\ \cite{Forster2017}) now implies that there exist open neighborhoods $\mathcal{U}\subset \mathcal{X}$ and $\mathcal{V}\subset \mathbb{R}^{N_W}$ with $(x^\ast,w^\ast) \in \mathcal{U} \times \mathcal{V}$ and a continuously differentiable function $g: \mathcal{V} \rightarrow \mathcal{U}$ with $g(w^\ast) = x^\ast$ such that
			\begin{equation*}
				F(g(w),w) = 0 \qquad \text{for all } w \in \mathcal{V}.
			\end{equation*}
			Consequently, for every $w \in \mathcal{V}$, the corresponding neural network has a critical point at $x = g(w)$. As $\mathcal{V}$ is an open set in $\mathbb{R}^{N_W}$, it has non-zero Lebesgue measure. As the determinant of the matrix $\frac{\partial F}{\partial x} = H_{\Phi}(x)$ is a continuous function in $w$, the neighborhood $\mathcal{V}$ can be chosen small enough, such that the critical point $x = g(w)$ is for every $w \in \mathcal{V}$ non-degenerate.
		\end{proof}
		
		\subsection{Regularity of Critical Points}
		\label{sec:nn_regularity}
		
		To prove the main results of this section, we use the following lemma of differential geometry, characterizing Morse functions not via their Hessian, but via the mixed second derivatives with respect to the weights and the input data. 
		
		\begin{lemma}[\cite{Nicolaescu2011}]\label{lem:nn_morsefunction}
			Let $\widehat{\Lambda}\in C^{2}(\mathcal{X} \times \mathcal{V},\mathbb{R})$, $\mathcal{X}\subset \mathbb{R}^n$, $\mathcal{V}\subset \mathbb{R}^p$ with $n \leq p$ and  denote by $\Lambda_v\in C^2(\mathcal{X},\mathbb{R})$ the family of functions $\Lambda_v(x) \coloneqq \widehat{\Lambda}(x,v)$, which depend continuously on the parameter $v \in \mathcal{V}$. If
			\begin{equation*}
				\frac{\partial^2 \widehat{\Lambda}}{\partial v \partial x }(x,v)\in \mathbb{R}^{n \times p}, \quad  \left[\frac{\partial^2 \widehat{\Lambda}}{\partial v \partial x }(x,v)\right]_{ij} = \frac{\partial^2 \widehat{\Lambda}}{\partial v_j \partial x_i}(x,v), \quad i \in \{1,\ldots,n\},\;  j \in \{1,\ldots,p\},
			\end{equation*}
			is for every $(x,v) \in \mathcal{X}\times \mathcal{V}$ surjective, i.e., the matrix $\frac{\partial^2 \widehat{\Lambda}}{\partial v \partial x }(x,v)$ has for every $(x,v) \in \mathcal{X}\times \mathcal{V}$ full rank~$n$, then there exists a set $\mathcal{V}_0 \subset \mathcal{V}$ of Lebesgue measure zero in $\mathbb{R}^p$, such that the function $\Lambda_v$ is for all weights $v \in \mathcal{V}\setminus\mathcal{V}_0 $  a Morse function. 
		\end{lemma}
		
		\begin{example}\normalfont
			Let $\widehat{\Lambda}\in C^2(\mathbb{R} \times \mathbb{R}^2,\mathbb{R})$ be defined by $\widehat{\Lambda}(x,v_1,v_2) = v_1x^2+v_2x$. As the matrix $$\frac{\partial^2 \widehat{\Lambda}}{\partial v \partial x }(x,v) = \left(2x , 1\right)$$ has for all $(x,v) \in \mathbb{R} \times \mathbb{R}^2$ rank one, Lemma~\ref{lem:nn_morsefunction} implies that $\Lambda_v(x) = v_1x^2+v_2x$ is for all $v \in \mathbb{R}^2$, except possibly for a set $\mathcal{V}_0 \subset \mathbb{R}$ of Lebesgue measure zero, a Morse function. In this example $\mathcal{V}_0 = (0,0)$ as for $v_1 \neq 0$ it holds for the Hessian $H_{\Lambda_v}(x) = 2v_1 \neq 0$, such that every possible critical point is non-degenerate. For $v_1 = 0$ and $v_2 \neq 0$, the map $\Lambda_v$ has no critical point and hence is a Morse function, only for $v_1 = v_2 = 0$ the map $\Lambda_v$ has degenerate critical points and is not a Morse~function.
		\end{example}
		
		The idea to apply Theorem~\ref{lem:nn_morsefunction} to MLPs was already used by Kurochkin in \cite{Kurochkin2021}. In this work, we rigorously prove an analogous result to the main theorem of \cite{Kurochkin2021} for our more general setting. To use Lemma~\ref{lem:nn_morsefunction}, the mixed partial second derivatives with respect to the input $x$ and the parameters~$v$ need to be calculated. For MLPs, this means to differentiate the network gradient of Lemma~\ref{lem:nn_gradient} with respect to the weights $W$, $\widetilde{W}$, and the biases $b$, $\tilde{b}$. As these computations are lengthy, the proof is postponed to Appendix~\ref{app:regularity}. Due to the chain rule, the partial derivatives with respect to $W$ and $b$ or with respect to $\widetilde{W}$ and $\tilde{b}$ have a similar form, which is used to show together with the structure of an augmented MLP the surjectivity of the matrix of mixed partial second derivatives. The result we obtain for augmented MLPs is stated below.
		
		\begin{theorem} \label{th:nn_morse}
			Any augmented MLP $\Phi \in \Xi^k_{\textup{A},\mathbb{W}^\ast}(\mathcal{X},\mathbb{R})$, $k \geq 2$, with weight matrices $W\in \mathbb{W}^\ast$ and biases $b \in \mathbb{B}$, is for all weights $(\vem(W);\stack(B)) \in \mathbb{R}^{N_W+N_B} $, except possibly for a zero set in $\mathbb{R}^{N_W+N_B}$ with respect to the Lebesgue measure, of class $(\mathcal{C}1)^k (\mathcal{X},\mathbb{R})$ or $(\mathcal{C}2)^k (\mathcal{X},\mathbb{R})$ and hence a Morse function. The same statement holds if $\mathbb{W}^\ast$ is replaced by $\mathbb{W}$. 
		\end{theorem}
		
		\begin{proof}
			Theorem~\ref{th:nn_morse} is proven as Theorem~\ref{th:nn_morse_app} in Appendix~\ref{app:regularity}.
		\end{proof}
		
		\begin{remark}\label{rem:nn_morse}
			In Appendix~\ref{app:regularity}, we explain after the proof of Theorem~\ref{th:nn_morse_app} in Remark~\ref{rem:nn_morsetheorem}, that the proof of Theorem~\ref{th:nn_morse_app} can also apply in certain cases to MLPs with a bottleneck. In the analysis of bottleneck architectures in Theorem~\ref{th:nn_bottleneck}, we prove in part~\ref{th:nn_bottleneck_c} the necessary conditions to show the same statement as in Theorem~\ref{th:nn_morse}/Theorem~\ref{th:nn_morse_app}  for specific bottleneck architectures.
		\end{remark}
		
		We end this section with an example of a one-layer augmented neural network illustrating the assertions of Theorems~\ref{th:nn_criticalpoint_fulllebesgue} and~\ref{th:nn_morse}.

		\begin{example} \label{ex:nn_augmented}
			\normalfont Consider the one-layer neural network architecture $\Phi \in \Xi^\infty_\textup{A}(\mathbb{R},\mathbb{R})$ with $L = 1$ defined by the weights and biases
			\begin{equation*}
				W_1 = \begin{pmatrix}
					w_1 \\ w_2
				\end{pmatrix} \in \mathbb{R}^{2}, \qquad b_1 = \begin{pmatrix}
					b_{11} \\ b_{12}
				\end{pmatrix} \in \mathbb{R}^2, \qquad \widetilde{W}_1 = (\widetilde{w}_1,\widetilde{w}_2) \in \mathbb{R}^{1\times 2}, \qquad \tilde{b}_1 \in \mathbb{R},
			\end{equation*}
			and soft-plus activation functions $[\sigma_1]_1(x) = [\sigma_1]_2(x) = \ln(1+\exp(x))$. $\Phi$ is visualized in Figure~\ref{fig:nn_augmented}. By Theorem~\ref{th:nn_morse}, the network $\Phi$ is for all $v = (w_1,w_2,\widetilde{w}_1,\widetilde{w}_2,b_{11},b_{12},\tilde{b}_1)^\top \in \mathbb{R}^7$, except possibly for a zero set with respect to the Lebesgue measure in $\mathbb{R}^7$ a Morse function, which we verify in the following. By Lemma~\ref{lem:nn_gradient} it holds for the gradient
			\begin{align*}
				\nabla \Phi(x) = (w_1,w_2)
				\begin{pmatrix}
					[\sigma_1]_1'([a_1]_1) & 0 \\
					0 & [\sigma_1]_2'([a_1]_2)
				\end{pmatrix}
				\begin{pmatrix}
					\widetilde{w}_1 \\ \widetilde{w}_2
				\end{pmatrix} = \frac{w_1 \widetilde{w}_1}{1+\exp(-a_1)}
				+\frac{w_2 \widetilde{w}_2}{1+\exp(-a_2)}
			\end{align*}
			and the Hessian is given by
			\begin{equation*}
				H_\Phi(x) = \frac{w_1^2 \widetilde{w}_1\exp(-a_1)}{(1+\exp(-a_1))^2}
				+\frac{w_2^2 \widetilde{w}_2\exp(-a_2)}{(1+\exp(-a_2))^2},
			\end{equation*}
			where $a_1 =w_1x+b_{11}$ and $a_2 = w_2x+b_{12}$.
			The set of weights
			\begin{equation*}
				\mathcal{W}_0 \coloneqq \left\{v\in \mathbb{R}^7: w_1 = 0 \vee w_2 = 0 \vee \widetilde{w}_1 = 0 \vee \widetilde{w}_2 = 0 \vee w_1 = w_2\right\}
			\end{equation*}
			describes a lower-dimensional set and hence is a zero set with respect to the Lebesgue measure in $\mathbb{R}^7$. From now on, we assume $v \in \mathbb{R}^7\setminus \mathcal{W}_0$. 
			If $\Phi$ has no critical point it is of class $(\mathcal{C}1)^\infty(\mathbb{R},\mathbb{R})$. Otherwise, let $x^\ast \in \mathbb{R}$ be a critical point, then the Hessian matrix evaluated at $x^\ast$ is given by
			\begin{equation*}
				H_\Phi(x^\ast) = \frac{\widetilde{w}_1}{\widetilde{w}_2} \frac{w_1^2}{(1+\exp(-a_1^\ast))^2}\left(\widetilde{w}_2\exp(-a_1^\ast)+\widetilde{w}_1\exp(-a_2^\ast)\right),
			\end{equation*}
			where we used the algebraic constraint $\nabla \Phi(x^\ast) = 0$ and defined $a_1^\ast =w_1x^\ast+b_{11}$ and $a_2^\ast = w_2x^\ast+b_{12}$. As $v \in \mathbb{R}^7 \setminus \mathcal{W}_0$ it holds
			\begin{equation*}
				H_\Phi(x^\ast) = 0  \qquad \Leftrightarrow \qquad g(v,x^\ast) \coloneqq \widetilde{w}_2\exp(-a_1^\ast)+\widetilde{w}_1\exp(-a_2^\ast) = 0.
			\end{equation*}
			If $\widetilde{w}_1\widetilde{w}_2>0$, $g(v,x^\ast)$ is always non-zero, such that every critical point is non-degenerate. Otherwise $\widetilde{w}_1\widetilde{w}_2<0$ and
			the equation $g(v,x^\ast) = 0$ is for fixed $v$ uniquely solvable for $x^\ast$:
			\begin{equation} \label{eq:ex_augmented1}
				g(v,x^\ast) = 0 \qquad \Leftrightarrow \qquad x^\ast  = h(v) =  \frac{\ln(-\widetilde{w}_2/\widetilde{w}_1)-b_{11}+b_{12}}{w_1-w_2}.
			\end{equation}
			This implies that for fixed $v \in \mathbb{R}^7\setminus \mathcal{W}_0$, there can exist at most one degenerate critical point. Multiple critical points may exist, but~\eqref{eq:ex_augmented1} implies that there cannot be multiple degenerate critical points. In the appendix, we prove in Lemma \ref{lem:app_regular} that in the case $\widetilde{w}_1\widetilde{w}_2<0$, the set of weights, such that $\Phi$ has one degenerate critical point, has Lebesgue measure zero in $\R^7$. Hereby, we use Theorem \ref{th:nn_criticalpoint_fulllebesgue} that non-degenerate critical points perturb to non-degenerate critical points under small perturbations of the weights. Hence, for all weights and biases, except possibly for a Lebesgue measure zero set in $\R^7$, $\Phi$ is a Morse function.
		\end{example}

		\begin{figure}
			\centering
			\includegraphics[width=0.35\textwidth]{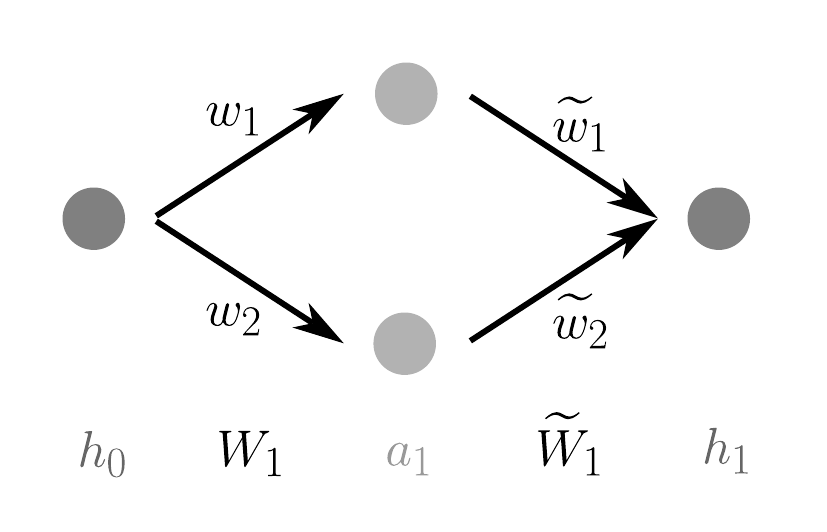}
			\caption{One-layer augmented neural network architecture $\Phi \in \Xi^\infty_\textup{A}(\mathbb{R},\mathbb{R})$ of Example~\ref{ex:nn_augmented}.}
			\label{fig:nn_augmented}
		\end{figure}
		
		\subsection{Analysis of Bottleneck Architectures} \label{sec:nn_bottleneck}
		
		In the last section, we proved Theorem~\ref{th:nn_morse} and showed the regularity of the critical points of augmented neural networks. As mentioned in Remark~\ref{rem:nn_morse}, the theorem can also be applied to neural networks with a bottleneck if they have a specific structure. In the following, we study different types of neural networks with a bottleneck, and show that they can be of class $(\mathcal{C}1)^k(\mathcal{X},\mathbb{R})$, $(\mathcal{C}2)^k(\mathcal{X},\mathbb{R})$, or $(\mathcal{C}3)^k(\mathcal{X},\mathbb{R})$. The main difference between the types of bottlenecks defined in the following theorem is whether the neural network has an augmented part and whether a certain matrix product is non-zero or not.
		
		\begin{theorem}\label{th:nn_bottleneck}
			Let $\Phi \in \Xi_\textup{B}^k(\mathcal{X},\mathbb{R})$, $k \geq 2$, $\mathcal{X}\subset \mathbb{R}^n$ be a MLP with a bottleneck with weight matrices $W =  (W_1, \widetilde{W}_1, \ldots, W_L, \widetilde{W}_L) = (V_1,\ldots,V_{2L}) \in \mathbb{W}$. Assume that $\Phi$ has its first bottleneck in layer $g_{j^\ast} \in \mathbb{R}^{d_{j^\ast}}$. Define for $l \in \{1,\ldots,2L-1\}$ the matrix product
			\begin{equation*}
				Z_l(x) \coloneqq \widetilde{W}_L \Psi_L(a_L)W_L\ldots V_{l+1}\in \mathbb{R}^{1\times d_l},
			\end{equation*}
			where $V_{l+1} = W_{l/2+1}$ if $l$ is even and $V_{l+1} = \widetilde{W}_{(l+1)/2}$ if $l$ is odd. $Z_l(x)$ is the product of all matrices occurring in the gradient defined in Lemma~\ref{lem:nn_gradient} until $V_{l+1}$.
			\begin{enumerate}[label=(\alph*), font=\normalfont]
				\item \label{th:nn_bottleneck_a}Let the first bottleneck of $\Phi$ be non-augmented, i.e., $n = d_0 \geq \ldots > d_{j^\ast}$ and let $Z_{j^\ast}(x)\neq 0 $ for all $x \in \mathcal{X}$. If $W \in \mathbb{W}^\ast$, $\Phi$ is of class $(\mathcal{C}1)^k(\mathcal{X},\mathbb{R})$. Consequently, for all sets of weights  $(\vem(W);\stack(B)) \in \mathbb{R}^{N_W+N_B} $, $W \in \mathbb{W}$, $b \in \mathbb{B}$, except possibly for a zero set in $\mathbb{R}^{N_W+N_B}$ with respect to the Lebesgue measure, $\Phi$ is of class $(\mathcal{C}1)^k (\mathcal{X},\mathbb{R})$.
				\item \label{th:nn_bottleneck_b}Let the first bottleneck of $\Phi$ be non-augmented, i.e., $n = d_0 \geq \ldots > d_{j^\ast}$ and denote the set of points $x \in \mathcal{X}$ such that $Z_{j^\ast}(x) = 0 $ by $\mathcal{X}^\ast \neq \varnothing$. Then every critical point $x^\ast \in \mathcal{X}^\ast$ is degenerate, such that $\Phi$ is of class $(\mathcal{C}3)^k (\mathcal{X},\mathbb{R})$.
				\item \label{th:nn_bottleneck_c}Let the first bottleneck of $\Phi$ be augmented, i.e., there exists an index $i^\ast<j^\ast $ with $n = d_0 \leq \ldots < d_{i^\ast}$ and let $Z_{i^\ast}(x)\neq 0$ for all $x \in \mathcal{X}$. Then $\Phi$ is for all weights $(\vem(W);\stack(B)) \in \mathbb{R}^{N_W+N_B} $, $W \in \mathbb{W}^\ast$, $b \in \mathbb{B}$, except possibly for a zero set in $\mathbb{R}^{N_W+N_B}$ with respect to the Lebesgue measure, of class $(\mathcal{C}1)^k (\mathcal{X},\mathbb{R})$ or $(\mathcal{C}2)^k (\mathcal{X},\mathbb{R})$. The same statement holds if $\mathbb{W}^\ast$ is replaced by $\mathbb{W}$.
				\item \label{th:nn_bottleneck_d} Let the first bottleneck of $\Phi$ be augmented, i.e., there exists an index $i^\ast<j^\ast $ with   $n = d_0 \leq \ldots < d_{i^\ast}$ and denote the set of points $x \in \mathcal{X}$ such that $Z_{i^\ast}(x) = 0 $ by $\mathcal{X}^\ast \neq \varnothing$. Then $x^\ast\in\mathcal{X}^\ast$ can be either degenerate or non-degenerate, so $\Phi$ is  of class $(\mathcal{C}2)^k (\mathcal{X},\mathbb{R})$ or $(\mathcal{C}3)^k (\mathcal{X},\mathbb{R})$.
			\end{enumerate}
		\end{theorem}
		
		\begin{proof}
			Let $\Phi \in \Xi^k_{\textup{B}}(\mathcal{X},\mathbb{R})$, with weight matrices $W = (W_1, \widetilde{W}_1, \ldots, W_L, \widetilde{W}_L) = (V_1,\ldots,V_{2L}) \in \mathbb{W}$. By definition, there exist layers $g_i^\ast$, $g_j^\ast$ and $g_l^\ast$ with $0 \leq i^\ast < j^\ast < l^\ast \leq 2L$, such that $d_i^\ast>d_j^\ast$ and $d_j^\ast <d_l^\ast$. The network gradient is by Lemma~\ref{lem:nn_gradient} given by
			\begin{align*}
				\nabla \Phi(x)^\top &= \widetilde{W}_L\Psi_L(a_{L})W_{L} \ldots V_{l +1}V_{l} \ldots \widetilde{W}_1\Psi_1(a_{1})W_1 \\
				&=V_{2L}\Psi_L(a_{L})V_{2L-1} \ldots V_{l +1}V_{l}\ldots V_2\Psi_1(a_{1})V_1
			\end{align*}
			with $l \in \{1,\ldots,2L-1\}$, where $V_{l} = \widetilde{W}_{l/2} \in \mathbb{R}^{d_{l} \times d_{l-1} }$, $V_{l+1} = W_{l/2+1} \in \mathbb{R}^{d_{l+1} \times d_{l} }$ if $l$ is even and $V_{l} =  \Psi_{(l+1)/2}(a_{(l+1)/2}) \in \mathbb{R}^{d_{l} \times d_{l} }$, $V_{l+1} = \widetilde{W}_{(l+1)/2} \in \mathbb{R}^{d_{l+1}\times d_l}$ if $l$ is odd. We denote by $Z_l(x) \in \mathbb{R}^{1\times d_l}$ the matrix product $\widetilde{W}_L \ldots V_{l+1}$ and by $Y_l(x)\in \mathbb{R}^{d_l \times n}$ the matrix product $V_l \ldots W_1$, such that
			\begin{equation} \label{eq:nn_matrixproduct}
				\nabla \Phi(x)^\top = Z_l(x)Y_l(x) \in \mathbb{R}^{1 \times n}.
			\end{equation}
			
			For case~\ref{th:nn_bottleneck_a}, we assume that $n = d_0 \geq \ldots > d_{j^\ast}$ and that $Z_{j^\ast}(x)\neq 0 $ for all $x \in \mathcal{X}$. If $W \in \mathbb{W}^\ast$, the assumption $n = d_0 \geq \ldots > d_{j^\ast}$ implies that the dimensions of the matrices in the product $Y_{j^\ast}(x)$ are monotonically increasing, such that by Lemma~\ref{lem:nn_productfullrank}  $Y_{j^\ast}(x)$ has for every $x \in \mathcal{X}$ full rank $d_{j^\ast}$. As $Z_{j^\ast}(x) \neq 0$ for all $x \in \mathcal{X}$ also the matrix $Z_{j^\ast}(x)$ has always full rank 1, such that another application of Lemma~\ref{lem:nn_productfullrank} implies that $\nabla \Phi(x)$ has for every $x \in \mathcal{X}$ full rank. Consequently no critical point exists and $\Phi$ is of class $(\mathcal{C}1)^k(\mathcal{X},\mathbb{R})$ if $W \in \mathbb{W}^\ast$. The same argumentation as in Theorem~\ref{th:nn_nocriticalpoints} for non-augmented neural networks shows that for all sets of weights  $(\vem(W);\stack(B)) \in \mathbb{R}^{N_W+N_B} $, except possibly for a zero set in $\mathbb{R}^{N_W+N_B}$ with respect to the Lebesgue measure, $\Phi$ is of class $(\mathcal{C}1)^k (\mathcal{X},\mathbb{R})$.
			
			For case~\ref{th:nn_bottleneck_b}, we assume that $n = d_0 \geq \ldots > d_{j^\ast}$ and that $Z_{j^\ast}(x^\ast) = 0 $ for all $x^\ast \in \mathcal{X}^\ast \neq \varnothing$. The structure of the gradient~\eqref{eq:nn_matrixproduct} implies that every $x^\ast\in \mathcal{X}^\ast$ is a critical point of $\Phi$. The $i$-th row  $[H_\Phi(x)^\top]_i^\top$ of the Hessian $H_\Phi(x)\in \mathbb{R}^{n \times n}$ evaluated at a critical point $x^\ast \in \mathcal{X}^\ast$ is given by
			\begin{equation*}
				[H_\Phi(x^\ast)^\top]_i^\top = \frac{\partial }{\partial x_i}\left[Z_{j^\ast}(x)Y_{j^\ast}(x)\right] \bigg\vert_{x = x^\ast} =   \frac{\partial Z_{j^\ast}}{\partial x_i}(x^\ast) Y_{j^\ast}(x^\ast) + Z_{j^\ast}(x^\ast)\frac{\partial Y_{j^\ast}}{\partial x_i}(x^\ast) = \frac{\partial Z_{j^\ast}}{\partial x_i}(x^\ast) Y_{j^\ast}(x^\ast),
			\end{equation*}
			as $Z_{j^\ast}(x^\ast) = 0$ for $x^\ast \in \mathcal{X}^\ast$. 
			The $n$ vectors $\frac{\partial Z_{j^\ast}}{\partial x_i}(x^\ast) \in \mathbb{R}^{1 \times d_{j^\ast}}$, $i \in \{1,\ldots,n\}$ are linearly dependent as $d_{j^\ast}<n$. Hence, there exist $\alpha_1, \ldots, \alpha_n$ not all equal to zero such that
			\begin{equation*}
				\sum_{i = 1}^n \alpha_i \frac{\partial Z_{j^\ast}}{\partial x_i}(x^\ast) = 0 \in \mathbb{R}^{1\times d_{j^\ast}}.
			\end{equation*}
			As matrix multiplication is distributive, it follows
			\begin{equation*}
				\sum_{i = 1}^n  \alpha_i	[H_\Phi(x^\ast)^\top]_i^\top  = 	\sum_{i = 1}^n  \alpha_i \left[\frac{\partial Z_{j^\ast}}{\partial x_i}(x^\ast) Y_{j^\ast}(x^\ast)\right] = \left[\sum_{i = 1}^n  \alpha_i \frac{\partial Z_{j^\ast}}{\partial x_i}(x^\ast)\right]  Y_{j^\ast}(x^\ast) = 0\in \mathbb{R}^n,
			\end{equation*}
			which implies that the rows of the Hessian $H_\Phi(x^\ast)$ are linearly dependent. Consequently, $H_\Phi(x^\ast)$ does not have full rank and at least one eigenvalue is zero, such that the critical point $x^\ast$ is degenerate. As $x^\ast$ was arbitrary, it follows that every critical point of $\Phi$ is degenerate and $\Phi$ is of class $(\mathcal{C}3)^k(\mathcal{X},\mathbb{R})$.
			
			For case~\ref{th:nn_bottleneck_c} we assume that there exists an index $i^\ast<j^\ast $ with $n = d_0 \leq \ldots < d_{i^\ast}$ and that $Z_{i^\ast}(x)\neq 0$ for all $x \in \mathcal{X}$. We aim to show the statement of Theorem~\ref{th:nn_morse}, which is proven as Theorem~\ref{th:nn_morse_app} in Appendix~\ref{app:regularity}, for the considered bottleneck architecture, as proposed in Remark~\ref{rem:nn_morse}. To that purpose, we compare the matrices $\overline{Y}$ and $\overline{Z}$ of the proof of Theorem~\ref{th:nn_morse_app} with the matrices $Y_{i^\ast}(x)$ and $Z_{i^\ast}(x)$ defined in this proof by identifying the index $l^\ast$ of Theorem~\ref{th:nn_morse_app} with the index $i^\ast$ considered here. If $i^\ast$ is even, equation~\eqref{eq:modified1a} implies that $Y_{i^\ast}(x) = \overline{Y}^\top$ and $Z_{i^\ast}(x) = \overline{Z}^\top W_{i^\ast/2+1}$, and if $i^\ast$ is odd,~\eqref{eq:modified2a} shows that $Y_{i^\ast}(x) = \overline{Y}^\top$ and $Z_{i^\ast}(x) = \overline{Z}^\top \widetilde{W}_{(i^\ast+1)/2}$. If $W \in \mathbb{W}^\ast$, the assumption $n = d_0 \leq \ldots < d_{j^\ast}$ implies that the dimensions of the matrices in the product $Y_{i^\ast}(x)$ are monotonically decreasing, such that by Lemma~\ref{lem:nn_productfullrank}  $Y_{i^\ast}(x)$ has for every $x \in \mathcal{X}$ full rank $n$, and hence also $\overline{Y}$ has full rank $n$. The assumption $Z_{i^\ast}(x)\neq 0$ for all $x \in \mathcal{X}$ implies that it is necessary that $\overline{Z}\neq 0$ for all $x \in \mathcal{X}$, so $\overline{Z}$ has full rank 1. By Remark~\ref{rem:nn_morsetheorem}, the proof of Theorem~\ref{th:nn_morse_app} works analogously for the considered neural network $\Phi$ with a bottleneck, such that $\Phi$ is for all weights $(\vem(W);\stack(B)) \in \mathbb{R}^{N_W+N_B} $, $W \in \mathbb{W}^\ast$, $b \in \mathbb{B}$, except possibly for a zero set in $\mathbb{R}^{N_W+N_B}$ with respect to the Lebesgue measure, of class $(\mathcal{C}1)^k (\mathcal{X},\mathbb{R})$ or $(\mathcal{C}2)^k (\mathcal{X},\mathbb{R})$. As $\vem( \mathbb{W}_0) \times \stack(\mathbb{B})$ is a zero set with respect to the Lebesgue measure in $\mathbb{R}^{N_W+N_B}$ and the union of two sets with measure zero has again measure zero, the same statement holds if $\mathbb{W}^\ast$ is replaced by $\mathbb{W}$.
			
			For case~\ref{th:nn_bottleneck_d} we assume that there exists an index $i^\ast<j^\ast $ with $n = d_0 \leq \ldots < d_{i^\ast}$ and that $Z_{i^\ast}(x^\ast)= 0$ for all $x^\ast \in \mathcal{X}^\ast \neq \varnothing$. The structure of the gradient~\eqref{eq:nn_matrixproduct} implies that every $x^\ast\in \mathcal{X}^\ast$ is a critical point of $\Phi$. As in case~\ref{th:nn_bottleneck_b}, the $i$-th row  $[H_\Phi(x)^\top]_i^\top$ of the Hessian $H_\Phi(x)\in \mathbb{R}^{n \times n}$ evaluated at a critical point $x^\ast \in \mathcal{X}^\ast$ is given by
			\begin{equation*}
				[H_\Phi(x^\ast)^\top]_i^\top = \frac{\partial }{\partial x_i}\left[Z_{j^\ast}(x)Y_{j^\ast}(x)\right] \bigg\vert_{x = x^\ast} =   \frac{\partial Z_{j^\ast}}{\partial x_i}(x^\ast) Y_{j^\ast}(x^\ast) + Z_{j^\ast}(x^\ast)\frac{\partial Y_{j^\ast}}{\partial x_i}(x^\ast) = \frac{\partial Z_{j^\ast}}{\partial x_i}(x^\ast) Y_{j^\ast}(x^\ast),
			\end{equation*}	
			as $Z_{j^\ast}(x^\ast) = 0$ for $x^\ast \in \mathcal{X}^\ast$.
			As $d_{j^\ast}>n$, the $n$ vectors $\frac{\partial Z_{j^\ast}}{\partial x_i}(x^\ast) \in \mathbb{R}^{1 \times d_{j^\ast}}$, $i \in \{1,\ldots,n\}$ can be linearly dependent, but they can also be linearly independent. This implies, in analogy to case~\ref{th:nn_bottleneck_b}, that the Hessian matrix can have full rank, but does not need to. In Example~\ref{ex:nn_bottleneck} we show that both cases are possible. As $\Phi$ always has at least one critical point, it follows that $\Phi$ is of class $(\mathcal{C}2)^k (\mathcal{X},\mathbb{R})$ or $(\mathcal{C}3)^k (\mathcal{X},\mathbb{R})$.
		\end{proof}
		
		We end this section with a few examples of neural networks with bottleneck architectures, which show that all cases mentioned in Theorem~\ref{th:nn_bottleneck} exist. The considered architectures are visualized in Figure~\ref{fig:nn_bottleneck}.

		\begin{example} \label{ex:nn_bottleneck}
			\normalfont We present for each of the cases~\ref{th:nn_bottleneck_a}-\ref{th:nn_bottleneck_d} of Theorem~\ref{th:nn_bottleneck} a two-layer neural network architecture $\Phi \in \Xi_{\textup{B},\mathbb{W}^\ast}^k(\mathcal{X},\mathbb{R})$, $k \geq 2$ with full rank weight matrices $W \in \mathbb{W}^\ast$ and nonlinear, strictly monotonically increasing activation functions. 
			\begin{enumerate}[label=(\alph*), font=\normalfont]
				\item Let $\mathcal{X} = \mathbb{R}^2$, $L = 2$ and $\Phi$ be defined by the weight matrices
				\begin{equation*}
					W_1 = (1,1), \quad \widetilde{W}_1 = \begin{pmatrix}
						1 \\ 1
					\end{pmatrix}, \quad W_2 = (1,1), \quad \widetilde{W}_2=1
				\end{equation*}
				and arbitrary biases $b\in \mathbb{B}$, see Figure~\ref{fig:nn_bottleneck}\subref{fig:bottleneck_a}. Then $\Phi$ has a non-augmented bottleneck at layer $g_1 = a_1$, so $j^\ast = 1$. It holds
				\begin{equation*}
					Z_1(x) = 1 \cdot \Psi_2(a_2) \cdot (1,1) \cdot \begin{pmatrix}
						1 \\ 1
					\end{pmatrix} = 2 \sigma_2'(a_2) \neq 0 \qquad \forall \, x \in \mathcal{X}.
				\end{equation*}
				By Theorem~\ref{th:nn_bottleneck}\ref{th:nn_bottleneck_a}, $\Phi$ is off class $(\mathcal{C}1)^k(\mathcal{X},\mathbb{R})$. As
				\begin{equation*}
					\nabla \Phi(x)^\top = 1 \cdot \Psi_2(a_2) \cdot (1,1) \cdot \begin{pmatrix}
						1 \\ 1
					\end{pmatrix} \cdot \Psi_1(a_1) \cdot (1,1) = (
					2\sigma_2'(a_2)\sigma_1'(a_1), 2\sigma_2'(a_2)\sigma_1'(a_1)) \neq (0,0)
				\end{equation*}
				for all $x \in \mathcal{X}$, 	we verified that $\Phi$ cannot have any critical points. 
				\item Let $\mathcal{X} = \mathbb{R}^2$, $L = 2$ and $\Phi$ be defined by the weight matrices
				\begin{equation*}
					W_1 = (1,1), \quad \widetilde{W}_1 = \begin{pmatrix}
						1 \\ 0
					\end{pmatrix}, \quad W_2 = (0,1), \quad \widetilde{W}_2=1
				\end{equation*}
				and arbitrary biases $b\in \mathbb{B}$, see Figure~\ref{fig:nn_bottleneck}\subref{fig:bottleneck_b}. Then $\Phi$ has a non-augmented bottleneck at layer $g_1 = a_1$, so $j^\ast = 1$. It holds
				\begin{equation*}
					Z_1(x) = 1 \cdot \Psi_2(a_2) \cdot (0,1) \cdot \begin{pmatrix}
						1 \\ 0
					\end{pmatrix} = 0 \qquad \forall \, x \in \mathcal{X}.
				\end{equation*}
				By Theorem~\ref{th:nn_bottleneck}\ref{th:nn_bottleneck_b}, $\Phi$ is off class $(\mathcal{C}3)^k(\mathcal{X},\mathbb{R})$.  As
				\begin{equation*}
					\nabla \Phi(x)^\top = 1 \cdot \Psi_2(a_2) \cdot (0,1) \cdot \begin{pmatrix}
						1 \\ 0
					\end{pmatrix} \cdot \Psi_1(a_1) \cdot (1,1) = (
					2\sigma_2'(a_2)\sigma_1'(a_1), 2\sigma_2'(a_2)\sigma_1'(a_1))  = (0,0)
				\end{equation*}
				for all $x \in \mathcal{X}$, the gradient $\nabla \Phi$ is the constant zero function, hence the Hessian matrix $H_\Phi$ is also zero everywhere. Hence, we verified that $\Phi$ is of class $(\mathcal{C}3)^k(\mathcal{X},\mathbb{R})$, as every critical point is degenerate.

				\begin{figure}
					\centering
					\begin{subfigure}{0.01\textwidth}
						\textcolor{white}{.}
					\end{subfigure}
					\begin{subfigure}{0.46\textwidth}
						\centering
						\includegraphics[width=\textwidth]{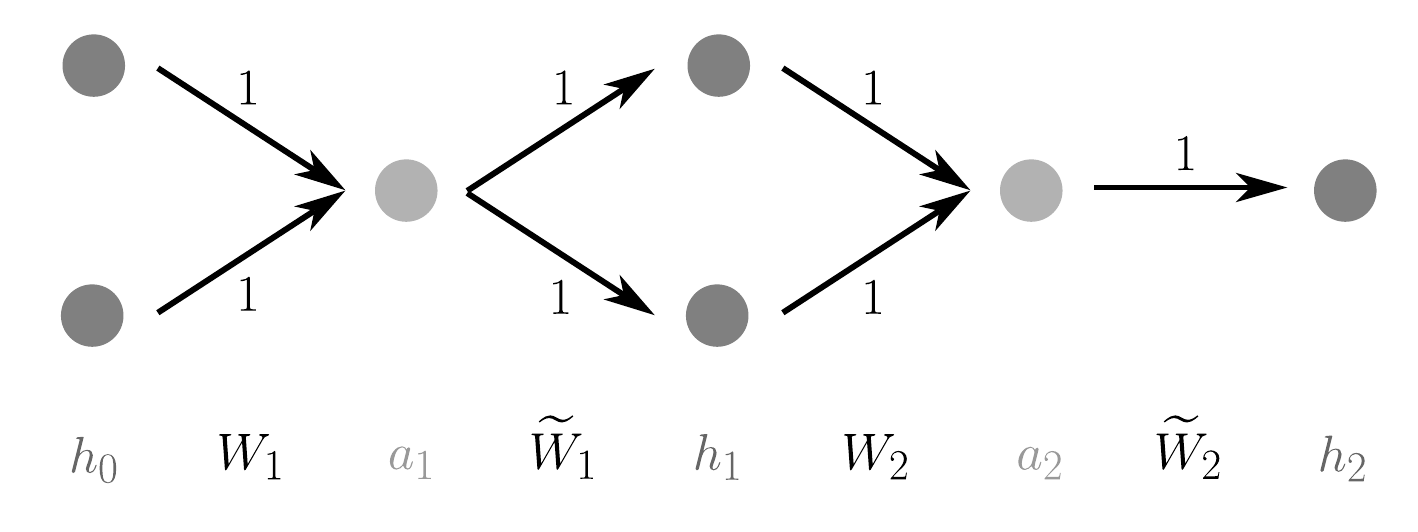}
						\caption{Example of a neural network $\Phi \in \Xi_{\textup{B},\mathbb{W}^\ast}^k(\mathbb{R}^2,\mathbb{R})$ with a non-augmented bottleneck at layer $a_1$, which is of class $(\mathcal{C}1)^k(\mathcal{X},\mathbb{R})$.}
						\label{fig:bottleneck_a}
					\end{subfigure}
					\begin{subfigure}{0.02\textwidth}
						\textcolor{white}{.}
					\end{subfigure}
					\begin{subfigure}{0.46\textwidth}
						\centering
						\includegraphics[width=\textwidth]{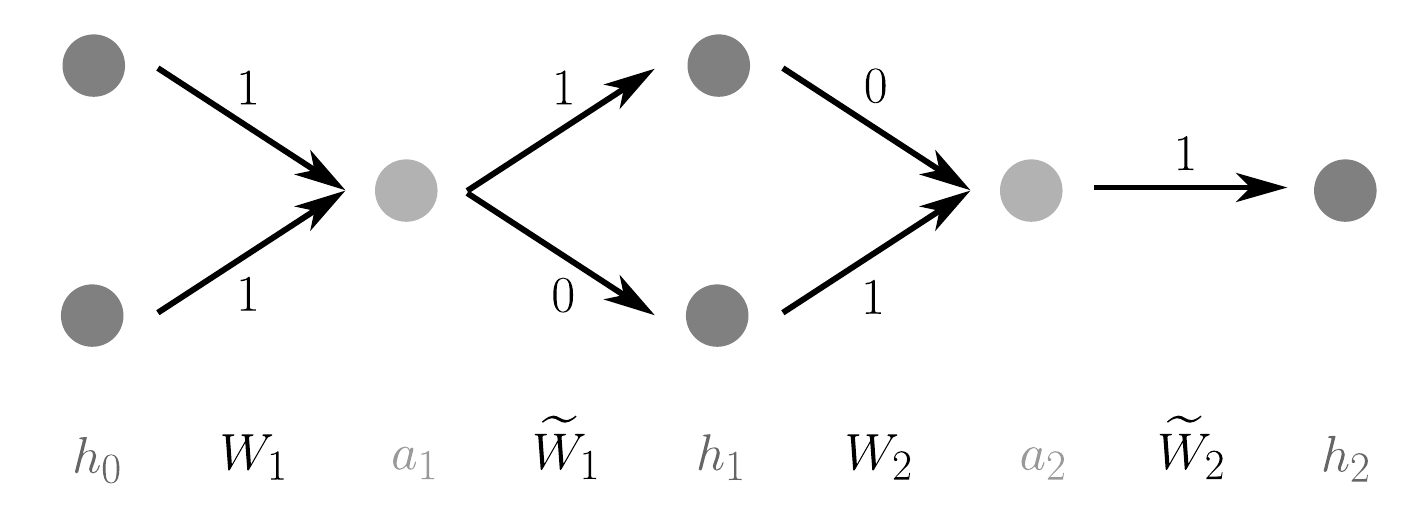}
						\caption{Example of a neural network  $\Phi \in \Xi_{\textup{B},\mathbb{W}^\ast}^k(\mathbb{R}^2,\mathbb{R})$ with a non-augmented bottleneck at layer $a_1$, which is of class $(\mathcal{C}3)^k(\mathcal{X},\mathbb{R})$.}
						\label{fig:bottleneck_b}
					\end{subfigure}
					\begin{subfigure}{0.01\textwidth}
						\textcolor{white}{.}
					\end{subfigure}
					\begin{subfigure}{0.01\textwidth}
						\textcolor{white}{.}
					\end{subfigure}
					\begin{subfigure}{0.46\textwidth}
						\vspace{6mm}
						\centering
						\includegraphics[width=\textwidth]{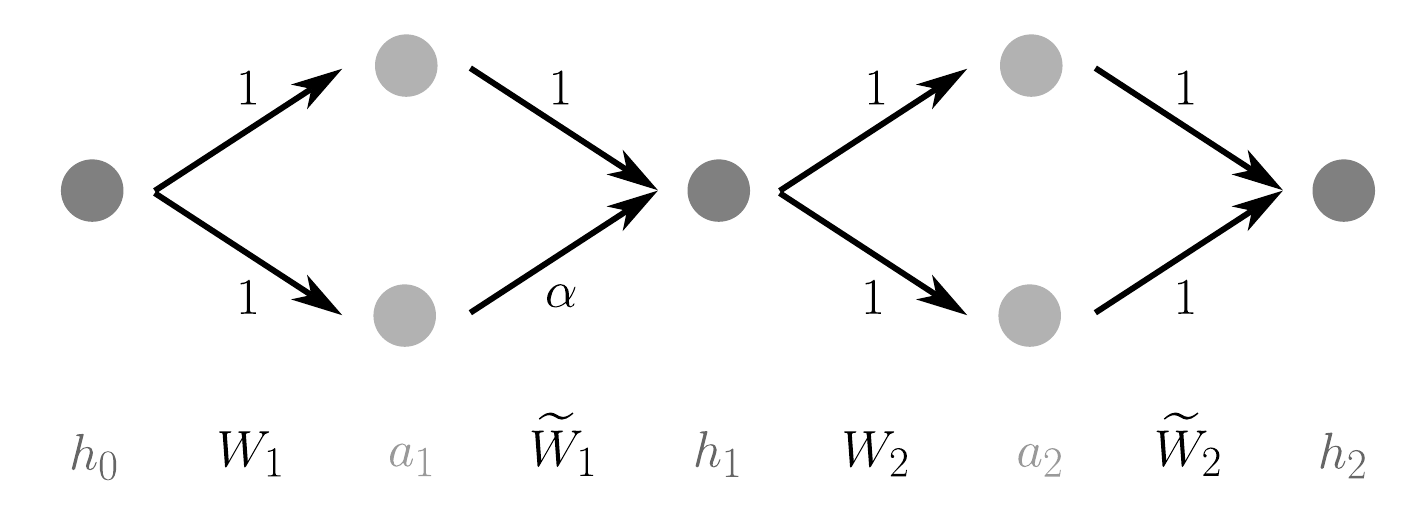}
						\caption{Example of a neural network  $\Phi \in \Xi_{\textup{B},\mathbb{W}^\ast}^k(\mathbb{R},\mathbb{R})$ with an augmented bottleneck  at layer $h_1$, which is depending on $\alpha$ of class $(\mathcal{C}1)^k(\mathcal{X},\mathbb{R})$ or $(\mathcal{C}2)^k(\mathcal{X},\mathbb{R})$.}
						\label{fig:bottleneck_c}
					\end{subfigure}
					\begin{subfigure}{0.02\textwidth}
						\textcolor{white}{.}
					\end{subfigure}
					\begin{subfigure}{0.46\textwidth}
						\vspace{6mm}
						\centering
						\includegraphics[width=\textwidth]{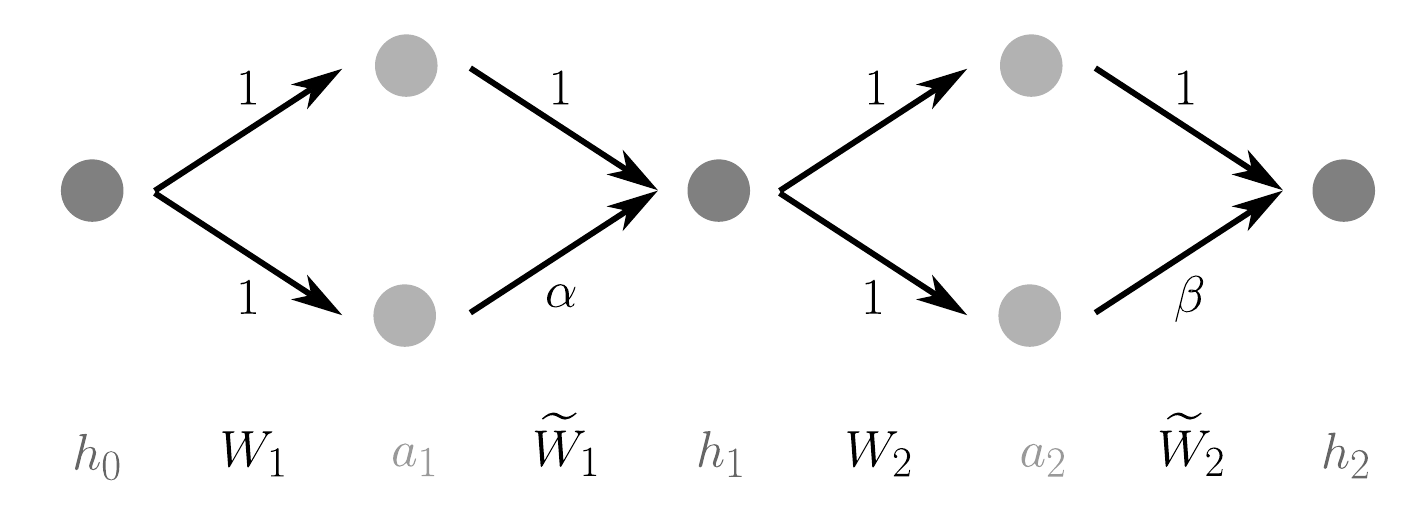}
						\caption{Example of a neural network $\Phi \in \Xi_{\textup{B},\mathbb{W}^\ast}^k(\mathbb{R},\mathbb{R})$ with an augmented bottleneck  at layer $h_1$, which is depending on $\alpha$ and $\beta$ of class $(\mathcal{C}1)^k(\mathcal{X},\mathbb{R})$, $(\mathcal{C}2)^k(\mathcal{X},\mathbb{R})$ or $(\mathcal{C}3)^k(\mathcal{X},\mathbb{R})$}
						\label{fig:bottleneck_d}
					\end{subfigure}
					\begin{subfigure}{0.01\textwidth}
						\textcolor{white}{.}
					\end{subfigure}
					\caption{Two-layer neural network architectures $\Phi \in \Xi_{\textup{B},\mathbb{W}^\ast}^k(\mathcal{X},\mathbb{R})$ with nonlinear, strictly monotonically increasing activation functions, which are analyzed in Example~\ref{ex:nn_bottleneck}.}
					\label{fig:nn_bottleneck}
				\end{figure}

				\item \label{ex:nn_bottleneck_c} Let $\mathcal{X} = \mathbb{R}$, $L = 2$ and $\Phi$ be defined by the weight matrices
				\begin{equation*}
					W_1 = \begin{pmatrix}
						1 \\ 1
					\end{pmatrix}, \quad \widetilde{W}_1 = (1,\alpha), \quad W_2 = \begin{pmatrix}
						1 \\ 1
					\end{pmatrix}, \quad \widetilde{W}_2=(1,1)
				\end{equation*}
				with $\alpha \in \mathbb{R}$ specified in the following and arbitrary biases $b\in \mathbb{B}$, see Figure~\ref{fig:nn_bottleneck}\subref{fig:bottleneck_c}. Then $\Phi$ has an augmented bottleneck at layer $g_2 = h_1$, so $i^\ast = 1$ and $j^\ast = 2$. It holds
				\begin{align*}
					Z_1(x) &= (1,1) \cdot \Psi_2(a_2) \cdot \begin{pmatrix}
						1 \\ 1
					\end{pmatrix} \cdot (1,\alpha) \\
					&= \big([\sigma_2]_1'([a_2]_1) + [\sigma_2]_2'([a_2]_2), \alpha([\sigma_2]_1'([a_2]_1) + [\sigma_2]_2'([a_2]_2))\big) \neq 0 \qquad \forall \, x \in \mathcal{X},
				\end{align*}
				as the sum of two strictly monotonically increasing functions is again strictly monotonically increasing. By Theorem~\ref{th:nn_bottleneck}\ref{th:nn_bottleneck_c}, $\Phi$ is for all sets of weights, except possibly for a zero set with respect to the Lebesgue measure, of class $(\mathcal{C}1)^k(\mathcal{X},\mathbb{R})$ or $(\mathcal{C}2)^k(\mathcal{X},\mathbb{R})$. We show that both cases are possible. It holds
				\begin{align*}
					\nabla \Phi(x)^\top &= (1,1) \cdot \Psi_2(a_2) \cdot \begin{pmatrix}
						1 \\ 1
					\end{pmatrix} \cdot (1,\alpha)\cdot \Psi_1(a_1) \cdot \begin{pmatrix}
						1 \\ 1
					\end{pmatrix} \\
					&= \big([\sigma_2]_1'([a_2]_1) + [\sigma_2]_2'([a_2]_2)\big)\cdot \big([\sigma_1]_1'([a_1]_1) + \alpha [\sigma_1]_2'([a_1]_2)\big).
				\end{align*}
				We define the set
				\begin{equation*}
					\mathcal{S} = \left\{s \in \mathbb{R}: s = -\frac{[\sigma_1]_1'([a_1]_1)}{[\sigma_1]_2'([a_1]_2)} \textup{ for some } x \in \mathbb{R}\right\} \subset \mathbb{R}_{<0},
				\end{equation*}
				which is non-empty and by construction of the network independent of the choice of $\alpha$. Also $\mathbb{R}\setminus\mathcal{S}$ is non-empty as  $\mathcal{S}\subset\mathbb{R}_{<0}$. If $\alpha \in \mathbb{R}\setminus \mathcal{S}$, then $\nabla \Phi(x) \neq 0$ for all $x \in \mathcal{X}$, such that $\Phi$ is of class $(\mathcal{C}1)^k(\mathcal{X},\mathbb{R})$. If the activation functions $[\sigma_1]_1$ and $[\sigma_1]_2$ are non-linear, then $[\sigma_1]_1'$ and $[\sigma_1]_2'$ are non-constant, such that $\mathcal{S}$ has non-zero Lebesgue measure. Hence, for all $\alpha \in \mathcal{S}$, except possibly for a zero set with respect to the Lebesgue measure,  Theorem~\ref{th:nn_bottleneck}\ref{th:nn_bottleneck_c} implies that $\Phi$ is of class $(\mathcal{C}2)^k(\mathcal{X},\mathbb{R})$ and has only non-degenerate critical points. 
				\item Let $\mathcal{X} = \mathbb{R}$, $L = 2$ and $\Phi$ be defined by the weight matrices
				\begin{equation*}
					W_1 = \begin{pmatrix}
						1 \\ 1
					\end{pmatrix}, \quad \widetilde{W}_1 = (1,\alpha), \quad W_2 = \begin{pmatrix}
						1 \\ 1
					\end{pmatrix}, \quad \widetilde{W}_2=(1,\beta)
				\end{equation*}
				with $\alpha,\beta \in \mathbb{R}$ specified in the following and arbitrary biases $b\in \mathbb{B}$, see Figure~\ref{fig:nn_bottleneck}\subref{fig:bottleneck_d}. It holds
				\begin{align*}
					Z_1(x) &= (1,\beta) \cdot \Psi_2(a_2) \cdot \begin{pmatrix}
						1 \\ 1
					\end{pmatrix} \cdot (1,\alpha) \\
					&= \big([\sigma_2]_1'([a_2]_1) + \beta [\sigma_2]_2'([a_2]_2), \alpha([\sigma_2]_1'([a_2]_1) + \beta [\sigma_2]_2'([a_2]_2))\big).
				\end{align*}
				We define the set 
				\begin{equation*}
					\mathcal{S}_\alpha = \left\{s \in \mathbb{R}: s = -\frac{[\sigma_2]_1'([a_2]_1)}{[\sigma_2]_2'([a_2]_2)} \textup{ for some } x \in \mathbb{R}\right\},
				\end{equation*}
				which is non-empty and by construction of the network dependent on the choice of $\alpha$, but independent of the choice of $\beta$. If $\beta \notin \mathcal{S}_\alpha$, the analysis of the neural network is the same as in case~\ref{ex:nn_bottleneck_c} and $\Phi$ can be of class $(\mathcal{C}1)^k(\mathcal{X},\mathbb{R})$ or $(\mathcal{C}2)^k(\mathcal{X},\mathbb{R})$.  In the following we choose $\beta \in \mathcal{S}_\alpha$, such that for every choice of $\alpha,\beta$ there exist $x^\ast \in \mathbb{R}$ such that $Z_1(x^\ast) = 0$ and the assumptions of Theorem~\ref{th:nn_bottleneck}\ref{th:nn_bottleneck_d} are fulfilled. Consequently $\Phi$ is off class $(\mathcal{C}2)^k(\mathcal{X},\mathbb{R})$ or $(\mathcal{C}3)^k(\mathcal{X},\mathbb{R})$. We show that both cases are possible. It holds
				\begin{align*}
					\nabla \Phi(x)^\top &= (1,\beta) \cdot \Psi_2(a_2) \cdot \begin{pmatrix}
						1 \\ 1
					\end{pmatrix} \cdot (1,\alpha)\cdot \Psi_1(a_1) \cdot \begin{pmatrix}
						1 \\ 1
					\end{pmatrix} \\
					&= \big([\sigma_2]_1'([a_2]_1) +\beta [\sigma_2]_2'([a_2]_2)\big)\cdot \big([\sigma_1]_1'([a_1]_1) + \alpha [\sigma_1]_2'([a_1]_2)\big).
				\end{align*}
				As $\beta \in \mathcal{S}_\alpha$, for every choice of $\alpha,\beta$ there exist at least one critical point $x^\ast \in \mathbb{R}$. The Hessian matrix $H_\Phi(x)$ is by the product rule given by
				\begin{align*}
					H_\Phi(x) & = \left([\sigma_2]_1''([a_2]_1)\frac{\partial [a_2]_1}{\partial x} +\beta [\sigma_2]_2''([a_2]_2)\frac{\partial [a_2]_2}{\partial x}\right)\cdot \left([\sigma_1]_1'([a_1]_1) + \alpha [\sigma_1]_2'([a_1]_2)\right) \\
					& \; \;\; + \left([\sigma_2]_1'([a_2]_1) +\beta [\sigma_2]_2'([a_2]_2)\right)\cdot \left([\sigma_1]_1''([a_1]_1)\frac{\partial [a_1]_1}{\partial x} + \alpha [\sigma_1]_2''([a_1]_2)\frac{\partial [a_1]_2}{\partial x}\right).
				\end{align*}
				Evaluated at a critical point $x^\ast$, it holds 
				\begin{equation*}
					H_\Phi(x^\ast) = \left([\sigma_2]_1''([a_2]_1)\frac{\partial [a_2]_1}{\partial x} +\beta [\sigma_2]_2''([a_2]_2)\frac{\partial [a_2]_2}{\partial x}\right)\cdot \left([\sigma_1]_1'([a_1]_1) + \alpha [\sigma_1]_2'([a_1]_2)\right),
				\end{equation*}
				as $Z_1(x^\ast) = 0$.
				By choosing $\alpha \in \mathcal{S}$, which is defined in part~\ref{ex:nn_bottleneck_c}, we can guarantee that the critical point $x^\ast$ can be degenerate, so it is possible that $\Phi$ is of class $(\mathcal{C}3)^k(\mathcal{X},\mathbb{R})$. To show that $\Phi$ can also be of class $(\mathcal{C}2)^k(\mathcal{X},\mathbb{R})$, we choose some $\alpha \in \mathbb{R}\setminus \mathcal{S}$, such that $[\sigma_1]_1'([a_1]_1) + \alpha [\sigma_1]_2'([a_1]_2) \neq 0$ for all $x \in \mathbb{R}$. $\alpha \in \mathbb{R}\setminus \mathcal{S}$ exists by part~\ref{ex:nn_bottleneck_c}. 	If we choose for example soft-plus activation functions $[\sigma_2]_1(x) = [\sigma_2]_2(x) = \ln(1+\exp(x))$, then
				\begin{equation*}
					\frac{\partial [a_2]_1}{\partial x} =	\frac{\partial [a_2]_2}{\partial x} = [\sigma_1]_1'([a_1]_1) + \alpha [\sigma_1]_2'([a_1]_2),
				\end{equation*}
				such that
				\begin{align*}
					H_\Phi(x^\ast) & = 
					\left(\frac{\exp(-[a_2]_1)}{(1+\exp(-[a_2]_1))^2}+\beta \frac{\exp(-[a_2]_2)}{(1+\exp(-[a_2]_2))^2}\right)\cdot \left([\sigma_1]_1'([a_1]_1) + \alpha [\sigma_1]_2'([a_1]_2)\right)^2 \\
					& = \left(\frac{\exp(-[a_2]_1)}{(1+\exp(-[a_2]_1))^2} -  \frac{1+\exp(-[a_2]_2)}{1+\exp(-[a_2]_1)}\frac{\exp(-[a_2]_2)}{(1+\exp(-[a_2]_2))^2}\right)\\
					& \hspace{4mm}\cdot \left([\sigma_1]_1'([a_1]_1) + \alpha [\sigma_1]_2'([a_1]_2)\right)^2 \\
					& = \frac{1}{1+\exp(-[a_2]_1)} \left(\frac{1}{1+\exp([a_2]_1)} -  \frac{1}{1+\exp([a_2]_2)}\right)\cdot \left([\sigma_1]_1'([a_1]_1) + \alpha [\sigma_1]_2'([a_1]_2)\right)^2 ,
				\end{align*}
				which is non-zero as long as $[a_2]_1 \neq[a_2]_2$ as $\frac{1}{1+\exp(x)}$ is a monotonically decreasing function. As $a_2 = W_1 h_1 + b_2$, variation of the bias $b_2 \in \mathbb{R}^2$ guarantees that for all $b_2 \in \mathbb{R}^2$, except possibly for a set of measure with respect to the Lebesgue measure in $\mathbb{R}^2$, it holds  $[a_2]_1 \neq[a_2]_2$. Hence, we showed the existence of a neural network architecture $\Phi$ which fulfills the assumptions of Theorem~\ref{th:nn_bottleneck}\ref{th:nn_bottleneck_d} and which is of class $(\mathcal{C}2)^k(\mathcal{X},\mathbb{R})$.
			\end{enumerate}
		\end{example}

		\section{Neural ODEs}
		\label{sec:neuralODEs}
		
		Each neural ODE architecture is based on the solution $h: \mathcal{I}_a \rightarrow \mathbb{R}^m$ of an initial value problem 
		\begin{equation} \label{eq:IVP} 
			\frac{\dd h}{\dd t} = f(t,h(t)), \qquad h(0) = a \in \mathcal{A} \subset \mathbb{R}^m, \tag{IVP}
		\end{equation}
		with $f :\Omega \rightarrow \mathbb{R}^m$, $\Omega = \Omega_t \times \Omega_h \subset \R \times \R^m$ open, where $\mathcal{I}_a \subset \Omega_t$ open denotes the maximal time interval of existence for the solution with $h(0) = a \in \mathcal{A}$, where $ \varnothing \neq \mathcal{A}\subset \Omega_h \subset \mathbb{R}^m$ is a non-zero set of possible initial conditions. We denote a solution with initial condition $h(0) = a $ by $h_a(t)$ to take into account the dependence on the initial data. The considered neural ODE architectures depend on the time-$T$ map $h_a(T)$ of~\eqref{eq:IVP}, hence we have to assume that for every $a \in \mathcal{A}$, the maximal time interval of existence fulfills $[0,T] \subset \mathcal{I}_a \subset \Omega_t$. 
		The vector field $f: \Omega \rightarrow \mathbb{R}^m$ is a continuous function, which can depend on parameters. As the results established in Sections~\ref{sec:node_architecture} to~\ref{sec:node_regularity} do not depend on the choice of the vector field, we consider no specific parameterization of $f$. In Section~\ref{sec:NODE_universal}, we discuss the relationship between general and parameterized vector fields of neural ODEs. First, we state a basic result from ODE theory regarding the regularity of the solution map of the initial value problem~\eqref{eq:IVP}.
		
		\begin{lemma}[\cite{Hale1980}]\label{lem:ivp_regularity}
			Let $f \in C^{0,k}(\Omega,\mathbb{R}^m)$, $\Omega = \Omega_t \times \Omega_h \subset \R \times \R^m$ open, $k \geq 1$ and assume $[0,T]\subset \mathcal{I}_a \subset \Omega_t$ for all $a \in \mathcal{A}\subset \R^m$. Then the solution $h_a:\mathcal{I}_a \rightarrow \mathbb{R}^m$ of the initial value problem~\eqref{eq:IVP} is unique and it holds for the time-$T$ map that $h_a(T) \in C^k(\mathcal{A},\mathbb{R}^m)$.
		\end{lemma}
		
		
		We need to assume the regularity of the vector field to be able to differentiate the network and characterize whether the neural ODE is a Morse function. The analysis is restricted to neural ODE architectures with scalar output, for more general networks, the results apply to the scalar components of the output. The architectures we consider have an affine linear layer before and after the time-$T$ map of the initial value problem, resulting in the neural ODE architecture
		\begin{equation} \label{eq:NODE}
			\Phi: \mathcal{X} \rightarrow \mathbb{R}, \qquad x \mapsto \widetilde{\lambda}( h_{\lambda(x)}(T)) = \widetilde{W} \cdot h_{Wx+b}(T) + \tilde{b} = \widetilde{W} \cdot h_{a}(T) + \tilde{b},
		\end{equation} 
		with $\mathcal{X}\subset \mathbb{R}^n$ open, where the first affine linear layer $\lambda: \mathbb{R}^n \rightarrow \mathbb{R}^m$ is represented by a matrix $W \in \mathbb{R}^{m \times n}$ and a vector $b \in \mathbb{R}^m$ and the second affine linear layer $\widetilde{\lambda}: \mathbb{R}^m\rightarrow \mathbb{R}$ is represented by a row vector $\widetilde{W} \in \mathbb{R}^{1 \times m}$ and a scalar $\tilde{b} \in \mathbb{R}$. The intermediate layer, also called pre-activated state, which defines the initial condition of~\eqref{eq:IVP} is calculated by $a = \lambda(x) = Wx+b$. We define the set of initial conditions of~\eqref{eq:IVP} to be $\mathcal{A} \coloneqq \lambda(\mathcal{X})$. Compared to MLPs, the neural ODE architecture~\eqref{eq:NODE} has one inner affine linear layer corresponding to the weights $W_1$ and $b_1$, then an initial value problem evolves in the dimension $m$ of the intermediate layer $a \coloneqq Wx+b$ and then a second outer affine linear layer corresponding to the weights $\widetilde{W}_L$ and $\tilde{b}_L$ is applied to the time-$T$ map of~\eqref{eq:IVP}. Under the assumptions specified in the following, we can define the set of all well-defined scalar neural ODEs with architecture~\eqref{eq:NODE} studied in this work.
		
		\begin{definition}[Neural ODE] \label{def:node}
			For $k \geq 1$, the set of all scalar neural ODE architectures $\Phi: \mathcal{X} \rightarrow \mathbb{R}$, $\mathcal{X}\subset \mathbb{R}^n$ open, defined by~\eqref{eq:NODE} with vector field $f \in C^{0,k}(\Omega,\mathbb{R}^m)$, $\Omega = \Omega_t \times \Omega_h \subset \R \times \R^m$ open, where $\mathcal{A} = \lambda(\mathcal{X})\subset \Omega_h$ and $[0,T]\subset \mathcal{I}_a \subset \Omega_t$ for all $a \in \mathcal{A}$, is denoted by $\textup{NODE}^k(\mathcal{X},\mathbb{R})$.
		\end{definition}

		The regularity of the neural ODEs of Definition~\ref{def:node} follows directly from the regularity of the underlying initial value problem.
		
		\begin{lemma}\label{lem:node_regularity}
			Let $\Phi \in \textup{NODE}^k(\mathcal{X},\mathbb{R})$ with underlying vector field $f \in C^{0,k}(\Omega,\mathbb{R}^m)$, $\Omega = \Omega_t \times \Omega_h \subset \R \times \R^m$ open and $k \geq 1$. Then it holds $\Phi\in C^k(\mathcal{X},\mathbb{R})$ and hence  $\textup{NODE}^k(\mathcal{X},\mathbb{R}) \subset C^k(\mathcal{X},\mathbb{R}) $.
		\end{lemma}
		
		\begin{proof}
			The result follows directly from Lemma~\ref{lem:ivp_regularity} and the smoothness of the affine linear layers $\lambda$ and $\tilde \lambda$.
		\end{proof}

		\begin{remark} \label{rem:uniqueness}
			To have a well-defined neural ODE $\Phi \in \textup{NODE}^k(\mathcal{X},\mathbb{R})$, $\mathcal{X}\subset \mathbb{R}^n$ open, the solution of~\eqref{eq:IVP} needs to be unique. Continuous and unique solutions of~\eqref{eq:IVP} can also be guaranteed under weaker conditions than in Lemma~\ref{lem:ivp_regularity}: as stated in \cite{Hale1980}, sufficient conditions are for example local Lipschitz continuity of $f$ with respect to the second variable and continuity with respect to the first variable, or the Carathéodory conditions, which allow for discontinuities in the time variable. For presentation purposes, we assume in this work that $f \in C^{0,k}(\Omega,\mathbb{R}^m)$, for the analysis of neural ODEs under weaker conditions, see for example~\cite{Kuehn2023}.
		\end{remark}

		\subsection{Special Architectures}
		\label{sec:node_architecture}
		
		In the following, we subdivide the set of all neural ODE architectures $\text{NODE}^k(\mathcal{X},\mathbb{R})$ introduced in Definition~\ref{def:node} in three different classes: non-augmented neural ODEs, augmented neural ODEs, and degenerate neural ODEs. 
		
		\subsubsection*{Non-Augmented}
		
		We call a neural ODE $\Phi \in \text{NODE}^k(\mathcal{X},\mathbb{R}) $ non-augmented if $n \geq m$, i.e., the dimension of the ODE is smaller than the dimension of the input data. We require that the weight matrices of both affine linear layers have full rank, i.e., $\text{rank}(W) = m$ and $\text{rank}(\widetilde{W})=1$. The subset of non-augmented neural ODEs is denoted by
		\begin{equation*}
			\textup{NODE}^k_\textup{N}(\mathcal{X},\mathbb{R}) \coloneqq \left\{\Phi \in \textup{NODE}^k(\mathcal{X},\mathbb{R}) : \Phi \textup{ is non-augmented}\right\}.
		\end{equation*}
		
		\subsubsection*{Augmented}
		
		We call a neural ODE $\Phi \in \text{NODE}^k(\mathcal{X},\mathbb{R}) $ augmented if $n<m$, i.e., the dimension of the ODE is larger than the dimension of the input data. We require that the weight matrices of both affine linear layers have full rank, i.e., $\text{rank}(W)=n$ and  $\text{rank}(\widetilde{W}) = 1$. The subset of augmented neural ODEs is denoted by
		\begin{equation*}
			\textup{NODE}^k_\textup{A}(\mathcal{X},\mathbb{R}) \coloneqq \left\{\Phi \in \textup{NODE}^k(\mathcal{X},\mathbb{R}) : \Phi \textup{ is augmented}\right\}.
		\end{equation*}
		
		\subsubsection*{Degenerate}	
		
		We say that a neural ODE $\Phi \in \text{NODE}^k(\mathcal{X},\mathbb{R})$ is degenerate if at least one of the weight matrices of the two affine linear layers has not full rank, i.e., $\text{rank}(W)<\min\{m,n\}$ or  $\text{rank}(\widetilde{W}) = 0$. The subset of  degenerate neural ODEs is denoted by
		\begin{equation*}
			\textup{NODE}^k_{\textup{D}}(\mathcal{X},\mathbb{R}) \coloneqq \left\{\Phi \in \textup{NODE}^k(\mathcal{X},\mathbb{R}) : \Phi \textup{ is degenerate}\right\}.
		\end{equation*}

		\begin{figure}[t!]
			\centering
			\begin{subfigure}{0.9\textwidth}
				\centering
				\includegraphics[width=0.69\textwidth]{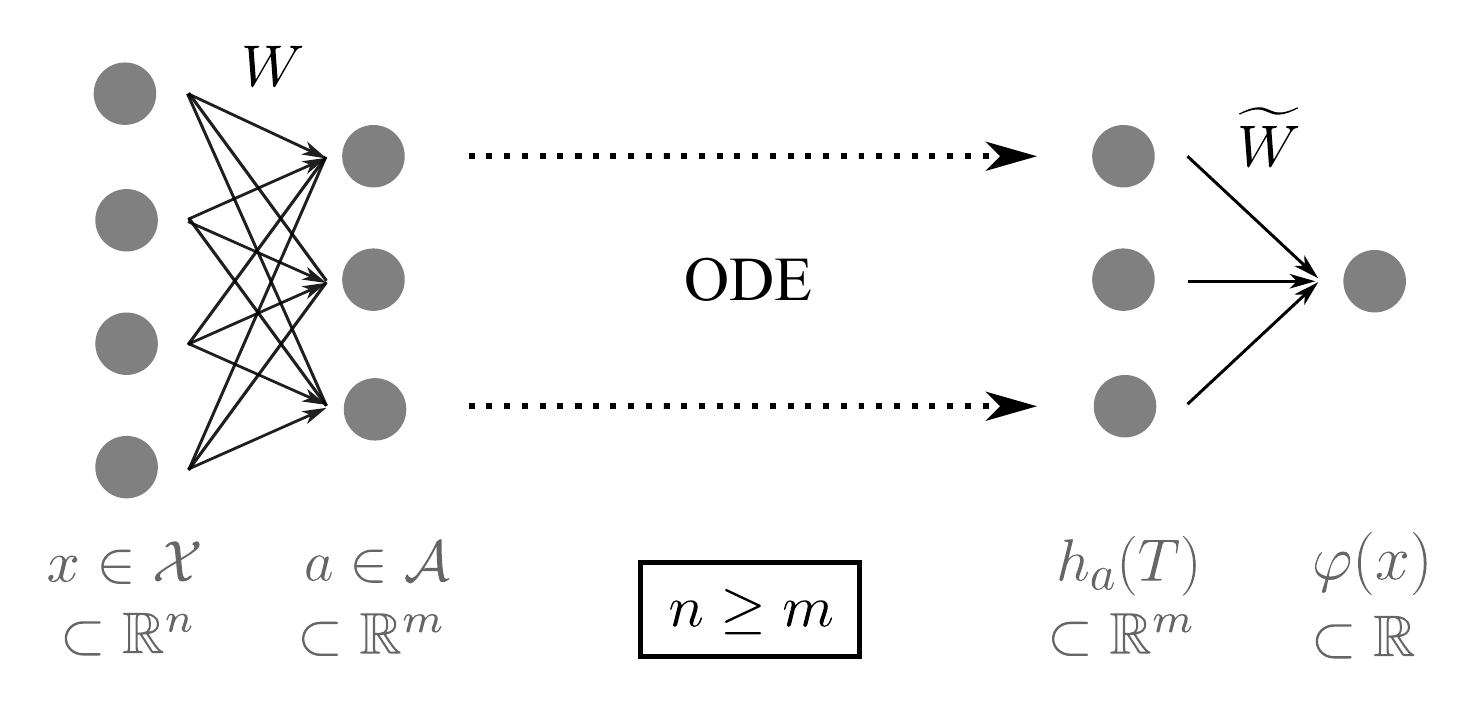}
				\caption{Structure of a non-augmented neural ODE $\Phi \in \text{NODE}_\textup{N}^k(\mathcal{X},\mathbb{R})$, $W$ and $\widetilde{W}$ have full rank. }
				\label{fig:node_arch_nonaugmented}
			\end{subfigure}
			\begin{subfigure}{0.9\textwidth}
				\vspace{5mm}
				\centering
				\includegraphics[width=0.69\textwidth]{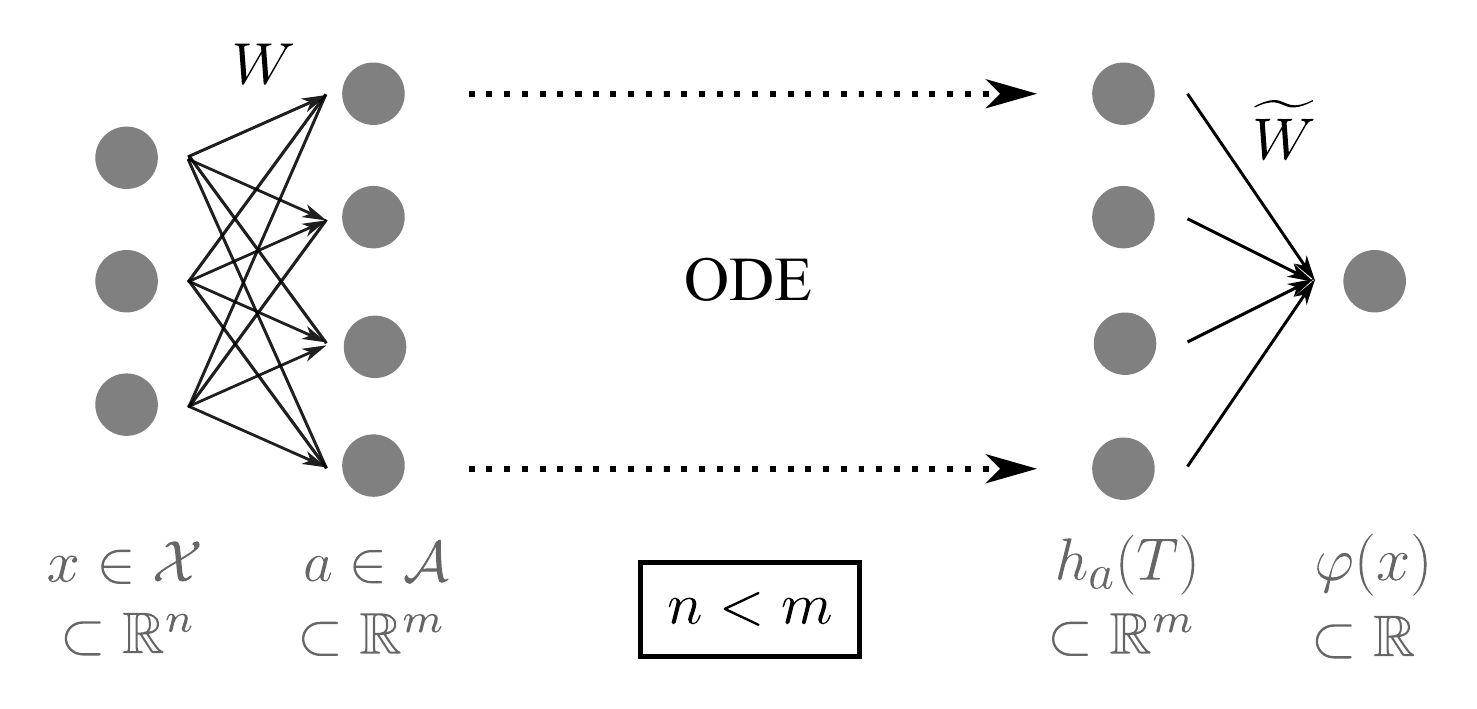}
				\caption{Structure of an augmented neural ODE $\Phi \in \text{NODE}_\textup{A}^k(\mathcal{X},\mathbb{R})$, $W$ and $\widetilde{W}$ have full rank.}
				\label{fig:node_arch_augmented}
			\end{subfigure}
			\begin{subfigure}{0.9\textwidth}
				\vspace{5mm}
				\centering
				\includegraphics[width=0.69\textwidth]{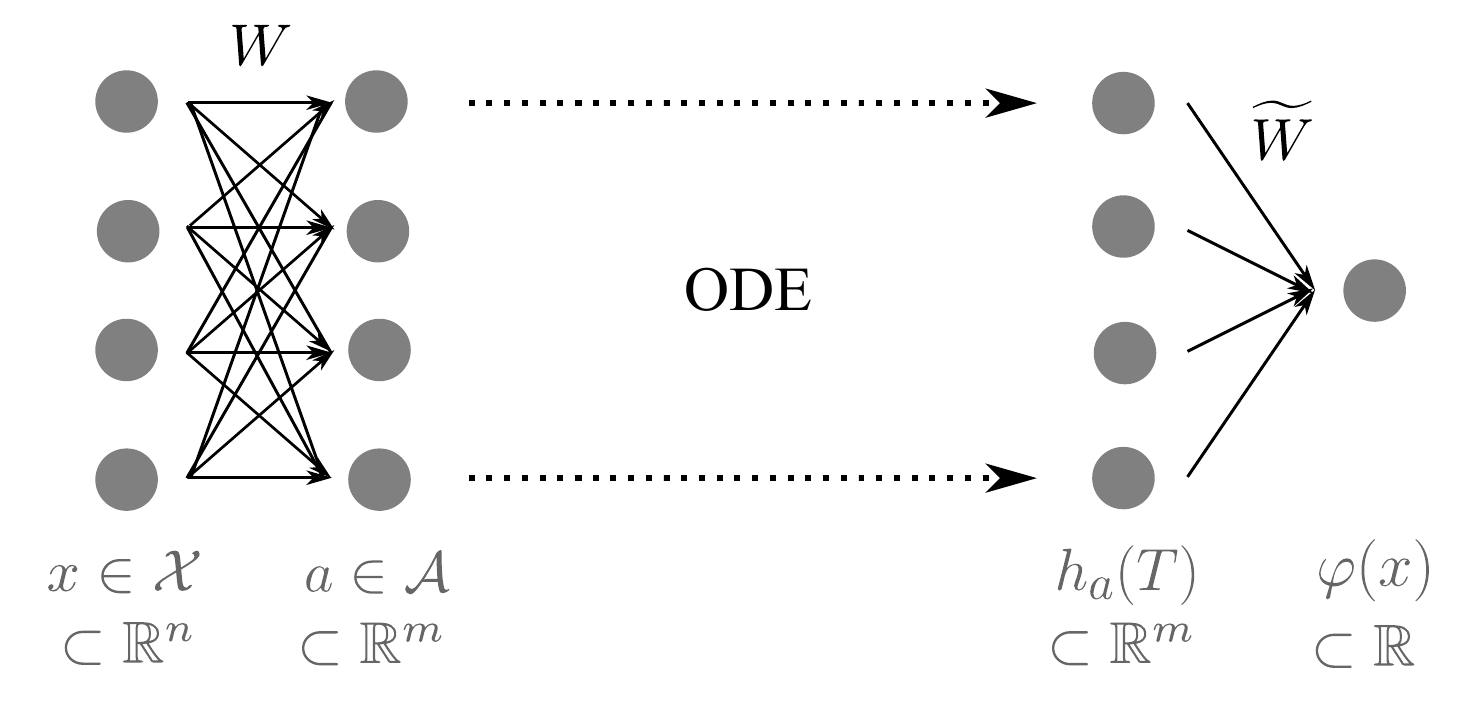}
				\caption{Structure of a degenerate neural ODE $\Phi \in \text{NODE}^k_\textup{D}(\mathcal{X},\mathbb{R})$, where at least one of the weight matrices $W$ or $\widetilde{W}$ has not full rank. There is no relationship between the dimensions $m$ and $n$.}
				\label{fig:node_arch_bottleneck}
			\end{subfigure}
			\caption{The three different types of architectures non-augmented, augmented, and degenerate, for neural ODEs $\Phi \in \text{NODE}^k(\mathcal{X},\mathbb{R})$.}
			\label{fig:node_architectures}
		\end{figure}
		
		The three different types of neural ODEs are visualized in Figure~\ref{fig:node_architectures}. In the following, we show that these three types of architectures build a disjoint subdivision of all neural ODE architectures.

		\begin{proposition} \label{prop:node_disjointpartition}
			The subdivision of neural ODEs in non-augmented neural ODEs of the form $\textup{NODE}^k_\textup{N}(\mathcal{X},\mathbb{R})$, augmented neural ODEs of the form $\textup{NODE}^k_\textup{A}(\mathcal{X},\mathbb{R})$ and degenerate neural ODEs $\textup{NODE}^k_{\textup{D}}(\mathcal{X},\mathbb{R})$, is a complete partition in three disjoint sub-classes of neural ODEs, i.e., 
			\begin{equation*}
				\textup{NODE}^k(\mathcal{X},\mathbb{R}) = \textup{NODE}^k_\textup{N}(\mathcal{X},\mathbb{R})\; \dot{\cup} \;\, \textup{NODE}^k_\textup{A}(\mathcal{X},\mathbb{R}) \; \dot{\cup} \;\, \textup{NODE}^k_{\textup{D}}(\mathcal{X},\mathbb{R})
				.		
			\end{equation*}
		\end{proposition}
		
		\begin{proof}
			All scalar neural ODEs of the form $\textup{NODE}^k(\mathcal{X},\mathbb{R})$ can be split up in neural ODEs, where the weight matrices $W$ and $\widetilde{W}$ have both full rank or at least one weight matrix has not full rank and hence the neural ODE is of class $\textup{NODE}^k_{\textup{D}}(\mathcal{X},\mathbb{R})$. The neural ODEs with full rank matrices are subdivided into the classes $\textup{NODE}^k_\textup{N}(\mathcal{X},\mathbb{R})$ and $ \textup{NODE}^k_\textup{A}(\mathcal{X},\mathbb{R})$, depending on whether $n \geq m$ or if $n<m$, such that we obtain a complete partition in three disjoint sub-classes of neural ODEs.
		\end{proof}
		
		\begin{remark}
			The comparison of non-augmented and augmented neural ODEs with MLPs with full rank matrices is straightforward. Degenerate neural ODEs resemble neural networks with a bottleneck in the following sense: if $\textup{rank}(W) = r<\min\{m,n\}$, then $\lambda(x)$ has $r$ linearly independent components, such that the effective dimension of the input is reduced from $n$ to $r$. As the ODE lives in $\mathbb{R}^m$, the dimension is increased from $r$ to~$m$, inducing a hidden bottleneck as $n>r$ and $r<m$. In the case that $\textup{rank}(\widetilde{W}) = 0$, the output of the neural ODE is constant and has no non-trivial dynamics.
		\end{remark}
		
		For the upcoming analysis, we define in analogy to MLPs a parameter space for the weight matrices and biases of scalar neural ODEs. For neural ODEs, the vector field $f$, which corresponds to all weights, biases, and activation functions of the hidden layer of an MLP, can be freely chosen. As our results do not depend on the choice of the vector field~$f$, we define in the following only a parameter space for the weights and biases $W$, $\widetilde{W}$, $b$ and $\tilde{b}$ building the affine linear layers $\lambda$ and $\tilde{\lambda}$.
		
		\begin{definition}[Neural ODE Parameter Space]
			Let $\Phi\in \textup{NODE}^k(\mathcal{X},\mathbb{R})$ be a scalar neural ODE with weight matrices $W\in \mathbb{R}^{m \times n}$, $\widetilde{W}\in\mathbb{R}^{1 \times m}$ and biases $b \in \mathbb{R}^m$, $\tilde{b}\in \mathbb{R}$. The weight space $\mathbb{V}$ of all possible weights and biases is defined as
			\begin{equation*}
				\mathbb{V}\coloneqq \mathbb{R}^{m \times n} \times \mathbb{R}^{1 \times m} \times \mathbb{R}^m \times \mathbb{R}.
			\end{equation*}
			The subset of $\mathbb{V}$, such that both weight matrices $W$ and $\widetilde{W}$ have full rank, is defined by
			\begin{equation*}
				\mathbb{V}^\ast \coloneqq \left\{(W,\widetilde{W},b,\tilde{b}) \in \mathbb{V}: \textup{rank}(W) = \min\{m,n\}, \textup{rank}(\widetilde{W}) = 1\right\}.
			\end{equation*}
			For $V \in \mathbb{V}$ it holds $\vem(V) \in \mathbb{R}^{N_V}$, where $N_V \coloneqq mn+2m+1$ and $\ve$ is the operator of Definition~\ref{def:nn_weightspace}\ref{def:nn_weightspace_b}.
		\end{definition}
		
		By definition, it holds that the weight matrices of non-augmented and augmented neural ODEs lie in $\mathbb{V}^\ast$ and the weight matrices of degenerate neural ODEs lie in $\mathbb{V}_0 \coloneqq \mathbb{V}\setminus \mathbb{V}^\ast$. As for MLPs, it also holds for neural ODEs that the set of weights and biases $\vem(\mathbb{V}_0)$, where at least one weight matrix has not full rank, is a set of Lebesgue measure zero in $\mathbb{R}^{N_V}$. 	
		
		\begin{lemma}
			The set $\vem(\mathbb{V}_0)$ is a zero set with respect to the Lebesgue measure in $\mathbb{R}^{N_V}$.
		\end{lemma}
		
		\begin{proof}
			Let $(W,\widetilde{W},b,\tilde{b}) \in \mathbb{V}$. By Lemma~\ref{lem:nn_weightspace}, the set
			\begin{equation*}
				Z_0 \coloneqq \left\{\vem(W,\widetilde{W}) \in \mathbb{R}^{mn+m}: \textup{rank}(W)<\min\{m,n\} \text{ or } \textup{rank}(\widetilde{W})=0 \right\}
			\end{equation*}
			is a zero set with respect to the Lebesgue measure in $\mathbb{R}^{mn+m}$. It follows that also the set
			\begin{equation*}
				\vem(\mathbb{V}_0) = Z_0 \times \mathbb{R}^{m}\times \mathbb{R} \subset \mathbb{R}^{mn+m} \times \mathbb{R}^{m}\times \mathbb{R}  = \mathbb{R}^{N_V}
			\end{equation*}
			is a zero set with respect to the Lebesgue measure in $\mathbb{R}^{N_V}$.
		\end{proof}

		\subsection{Existence of Critical Points}
		\label{sec:node_criticalpoints}
		
		In this section, we study the existence of critical points dependent on the special architecture of the scalar neural ODE. As for MLPs, we first calculate the network gradient.
		
		\begin{lemma}[Neural ODE Network Gradient]\label{lem:node_gradient}
			Let $\Phi \in \textup{NODE}^k(\mathcal{X},\mathbb{R})$, $k\geq 1$, $\mathcal{X}\subset \mathbb{R}^n$ open, be a scalar neural ODE with weight matrices $(W,\widetilde{W},b,\tilde{b})\in \mathbb{V}$. Then
			\begin{equation*}
				\nabla \Phi(x) = W^\top \left(\frac{\partial h_a(T)}{\partial a}\right)^\top \widetilde{W}^\top \in \mathbb{R}^n,
			\end{equation*}
			where $a = Wx+b$ and $\frac{\partial h_a(T)}{\partial a} \in \mathbb{R}^{m \times m}$ is the Jacobian matrix of the time-$T$ map $h_a(T)$ with respect to the initial value $a$ of the ODE, i.e.,  $\left[\frac{\partial h_a(T)}{\partial a}\right]_{ij} = \frac{\partial [h_a(T)]_i}{\partial a_j}$.
		\end{lemma}
		
		\begin{proof}
			By Lemma~\ref{lem:node_regularity} it holds $\Phi\in C^1(\mathcal{X},\mathbb{R})$, such that the multi-dimensional chain rule applied to 
			\begin{equation*}
				\Phi(x)\coloneqq \widetilde{\lambda}( h_{\lambda(x)}(T)) = \widetilde{W} \cdot h_{Wx+b}(T) + \tilde{b} = \widetilde{W} \cdot h_{a}(T) + \tilde{b}
			\end{equation*}
			yields
			\begin{equation*}
				\frac{\dd \Phi}{\dd x} = \widetilde{W} \frac{\partial h_a(T)}{\partial a} W \in \mathbb{R}^{1 \times n}
			\end{equation*}
			where $a = Wx+b$. The result follows by taking the transpose.
		\end{proof}
		
		The network gradient is a product of the weight matrices $W$, $\widetilde{W}$ and the Jacobian matrix $\frac{\partial h_a(T)}{\partial a}$. In the following, we show that independently of the choice of the vector field $f$ of~\eqref{eq:IVP} and the initial condition $a \in \mathcal{A}$, the Jacobian $\frac{\partial h_a(T)}{\partial a}$ has always full rank. The proof of the following proposition needs some results from ODE theory, stated for completeness in Appendix~\ref{sec:app_ODE}.
		
		\begin{proposition} \label{prop:node_jacobian}
			Let $\Phi \in \textup{NODE}^k(\mathcal{X},\mathbb{R})$, $k\geq 1$, $\mathcal{X}\subset \mathbb{R}^n$ open. Then the Jacobian $\frac{\partial h_a(T)}{\partial a}$ of the time-$T$ map $h_a(T)$ with respect to the initial data $a \in \mathcal{A} = \lambda(\mathcal{X})$ of the underlying initial value problem~\eqref{eq:IVP} has always full rank $m$.
		\end{proposition}
		
		\begin{proof}
			As $k\geq 1$, the initial value problem~\eqref{eq:IVP} with initial condition $a \in \mathcal{A} = \lambda(\mathcal{X})$  has by Lemma~\ref{lem:ivp_regularity} a unique solution $h_a: \mathcal{I}_a \rightarrow \mathbb{R}^m$ with $[0,T]\subset \mathcal{I}_a \subset \mathbb{R}$. Furthermore, given the continuously differentiable solution $h_a: \mathcal{I}_a  \rightarrow \mathbb{R}^m$ of~\eqref{eq:IVP}, the matrix function $A: \mathcal{I}_a  \rightarrow \mathbb{R}^{m \times m}$, $A(t) \coloneqq \frac{\partial f(t,h_a(t))}{\partial h_y}$, is continuous. By Lemma~\ref{lem:node_variational}, the Jacobian matrix $\frac{\partial h_a(t)}{\partial a}\in\mathbb{R}^{m \times m}$ satisfies the linear homogeneous ODE
			\begin{equation*}
				\frac{\dd Y}{\dd t} = A(t) Y, \qquad Y(0) = \text{Id}_m,
			\end{equation*}
			where $\text{Id}_m\in\mathbb{R}^{m \times m}$ is the identity matrix. By Lemma~\ref{lem:node_linearODE}\ref{lem:node_linearODE_a}, each column of $Y$ defines a unique solution $[Y]_i: \mathcal{I}_a \rightarrow\mathbb{R}^m$ with initial condition $[Y]_i(0) = \mathrm{e}_i^{(m)}$, where  $\mathrm{e}_i^{(m)}$ is the $i$-th unit vector in~$\mathbb{R}^m$. Hence, $Y:\mathcal{I}_a \rightarrow \mathbb{R}^{m \times m}$ is a unique matrix solution defined for all $t \in [0,T]\subset \mathcal{I}_a $ with initial condition $Y(0) = \text{Id}_m$, such that $\det(Y(0)) = 1$. As the matrix $A(t) \in \mathbb{R}^{m \times m}$ has only real entries, all eigenvalues of $A(t)$ are real or they come in complex conjugate pairs, such that the trace of $A(t)$ is a real number for all $t\in \mathcal{I}_a $. Consequently, $$\exp\left(\int_{t_0}^t \textup{tr}(A(r))\; \dd r\right)$$ is a real-valued, positive number, such that it follows by Lemma~\ref{lem:node_linearODE}\ref{lem:node_linearODE_b} that $\det (Y(T))\neq 0$. Consequently $Y(T) = \frac{\partial h_a(T)}{\partial a}$ has full rank $m$, independently of the initial data $a \in \mathcal{A}$.
		\end{proof}
		
		The characterization of the Jacobian matrix $\frac{\partial h_a(T)}{\partial a}$ in Proposition~\ref{prop:node_jacobian} is sufficient to prove that for non-augmented scalar ODEs, no critical points can exist.
		
		\begin{theorem} \label{th:node_nocriticalpoints}
			Any non-augmented scalar neural ODE $\Phi \in \textup{NODE}_\textup{N}^k(\mathcal{X},\mathbb{R})$, $\mathcal{X}\subset \R^n$ open, $k \geq 1$ has no critical point, i.e., $\nabla\Phi(x) \neq 0$ for all $x \in \mathcal{X}$. Consequently, it holds for $k\geq 1$ that 
			\begin{equation*}
				\textup{NODE}_\textup{N}^k(\mathcal{X},\mathbb{R}) \subset (\mathcal{C}1)^k(\mathcal{X},\mathbb{R}),
			\end{equation*}
			such that every non-augmented scalar neural ODE is a Morse function.
		\end{theorem}
		
		\begin{proof}
			Denote the weight matrices of $\Phi$ by $(W,\widetilde{W},b,\tilde{b})\in \mathbb{V}$. By Lemma~\ref{lem:node_gradient} it holds
			\begin{equation*}
				\nabla \Phi(x) = W^\top \left(\frac{\partial h_a(T)}{\partial a}\right)^\top \widetilde{W}^\top \in \mathbb{R}^n.
			\end{equation*}
			As the neural ODE is non-augmented, $W^\top \in\mathbb{R}^{n \times m}$ has rank $m \leq n$ and $\widetilde{W}^\top\in\mathbb{R}^{m \times 1}$ has rank~1. By Proposition~\ref{prop:node_jacobian}, the Jacobian $\left(\frac{\partial h_a(T)}{\partial a}\right)^\top \in \mathbb{R}^{m \times m}$ has independently of $a \in \mathcal{A}$ and hence also independently of $x \in \mathcal{X}$ full rank $m\geq 1$. By Lemma~\ref{lem:nn_productfullrank}, the product $\nabla \Phi(x)$ of the three given full rank matrices with monotonically decreasing widths has also full rank $1$, which implies that $\nabla \Phi(x) \neq 0$ for all $x \in \mathcal{X}$ and the result follows.
		\end{proof}
		
		\begin{remark}
			Theorem~\ref{th:node_nocriticalpoints} generalizes the results obtained in \cite{Kuehn2023}, where it is shown that a non-augmented neural ODE with $n = m$ and $W = \textup{Id}_n$ cannot be of class $(\mathcal{C}2)^k(\mathcal{X},\mathbb{R})$.
		\end{remark}
		
		The following theorem shows that the proof of Theorem~\ref{th:node_nocriticalpoints} does not apply to augmented and degenerate neural ODEs. It guarantees the existence of neural ODE architectures $\Phi \in \textup{NODE}^k_\textup{A}(\mathcal{X},\mathbb{R})$ and $\Phi \in \textup{NODE}^k_\textup{D}(\mathcal{X},\mathbb{R})$, which do have critical points.
		
		\begin{theorem} \label{th:node_criticalpoint}
			Let $\Phi \in \textup{NODE}^k(\mathcal{X},\mathbb{R})$, $k\geq 1$, be a scalar augmented or degenerate neural ODE.
			\begin{enumerate}[label=(\alph*), font=\normalfont]
				\item \label{th:node_criticalpoint_a} Given $W \in \mathbb{R}^{n \times m}$ with $n<m$ and $\textup{rank}(W) = n$ and a point $x \in \mathcal{X}$, then the weight matrix $\widetilde{W}\in\mathbb{R}^{1 \times m}$ can be chosen in such a way, that $\Phi \in \textup{NODE}^k_\textup{A}(\mathcal{X},\mathbb{R})$ has a critical point at $x$.
				\item \label{th:node_criticalpoint_b} Given $W \in \mathbb{R}^{n \times m}$ with $\textup{rank}(W) < \min\{m,n\}$ and a point $x \in \mathcal{X}$, then the weight matrix $\widetilde{W}\in\mathbb{R}^{1 \times m}$ can be chosen in such a way, that $\Phi \in \textup{NODE}^k_\textup{D}(\mathcal{X},\mathbb{R})$ has a critical point at $x$.
			\end{enumerate}
		\end{theorem}
		
		\begin{proof}
			By Lemma~\ref{lem:node_gradient} it holds
			\begin{equation*}
				\nabla \Phi(x) = W^\top \left(\frac{\partial h_a(T)}{\partial a}\right)^\top \widetilde{W}^\top \in \mathbb{R}^n,
			\end{equation*}
			where $a = Wx+b$ for the given $x \in\mathcal{X}$. The Jacobian $\frac{\partial h_a(T)}{\partial a}$ has by Lemma~\ref{prop:node_jacobian} full rank $m$. 
			
			In case~\ref{th:node_criticalpoint_a}, the matrix product $W^\top \left(\frac{\partial h_a(T)}{\partial a}\right)^\top \in \mathbb{R}^{n \times m}$ has by Lemma~\ref{lem:nn_productfullrank} full rank $n<m$, which implies that the $m$ columns of the matrix product are linearly dependent. Hence, there exists a non-zero choice of the matrix $\widetilde{W} \in \mathbb{R}^{1\times m}$, such that $\nabla \Phi(x) = 0$ for the given $x \in \mathcal{X}$.
			
			In case~\ref{th:node_criticalpoint_b}, the matrix $W$ has non-full rank, i.e., $\text{rank}(W) = r < \min\{m,n\}$. As the Jacobian matrix $\frac{\partial h_a(T)}{\partial a}$, has full rank, it follows that the matrix product  $W^\top \left(\frac{\partial h_a(T)}{\partial a}\right)^\top \in \mathbb{R}^{n \times m}$ has also rank~$r$. Since $r < \min\{m,n\}$, the $m$ columns of the matrix product are linearly dependent. Hence, there exists a non-zero choice of the matrix $\widetilde{W} \in \mathbb{R}^{1\times m}$, such that $\nabla \Phi(x) = 0$ for the given $x \in \mathcal{X}$.
		\end{proof}

		\subsection{Regularity of Critical Points}
		\label{sec:node_regularity}
		
		For a further characterization of augmented and degenerate neural ODEs, we study the regularity of critical points. First, we show in analogy to augmented MLPs, that augmented neural ODEs are for all weights and biases, except possibly for a zero set with respect to the Lebesgue measure, Morse functions.
		
		\begin{theorem}\label{th:node_augmented}
			Any augmented scalar neural ODE $\Phi \in \textup{NODE}^k_\textup{A}(\mathcal{X},\mathbb{R})$, $k\geq 2$ with weights and biases $(W,\widetilde{W},b,\tilde{b})\in \mathbb{V}^\ast$, is for all $\vem(W,\widetilde{W},b,\tilde{b}) \in \mathbb{R}^{N_V},$ except possibly for a zero set in $\mathbb{R}^{N_V}$ with respect to the Lebesgue measure, of class $(\mathcal{C}1)^k(\mathcal{X},\mathbb{R})$ or $(\mathcal{C}2)^k(\mathcal{X},\mathbb{R})$ and hence a Morse function.
		\end{theorem}
		
		\begin{proof}
			Let $\Phi \in \textup{NODE}^k_\textup{A}(\mathcal{X},\mathbb{R})$ be an augmented scalar neural ODE with full rank weight matrices $W\in \mathbb{R}^{m \times n}$, $\widetilde{W}\in \mathbb{R}^{1 \times m}$ and biases $b\in\mathbb{R}^n$, $\tilde{b}\in \mathbb{R}$. By Theorem~\ref{th:node_criticalpoint}\ref{th:node_criticalpoint_a}, it is possible that the neural ODE $\Phi$ has a critical point. In the case that $\Phi$ has no critical point, $\Phi$ is of class $(\mathcal{C}1)^k(\mathcal{X},\mathbb{R})$ and hence a Morse function. In the following we show, that if $\Phi$ has a critical point, then it is for all weights, where $\Phi$ has critical points, except possibly for a zero set with respect to the Lebesgue measure in $\mathbb{R}^{N_V}$ of class $(\mathcal{C}2)^k(\mathcal{X},\mathbb{R})$. By Lemma~\ref{lem:node_gradient}, the gradient of $\Phi$ is given by
			\begin{equation*}
				\nabla \Phi(x) = W^\top \left(\frac{\partial h_a(T)}{\partial a}\right)^\top \widetilde{W}^\top \in \mathbb{R}^n,
			\end{equation*}
			where $a = Wx+b$ and $\frac{\partial h_a(T)}{\partial a} \in \mathbb{R}^{m \times m}$ is the Jacobian matrix of the time-$T$ map $h_a(T)$ with respect to the initial value $a$. For a weight vector $v \in \mathbb{V}^\ast$, we define $\widehat{\Phi} \in C^k(\mathcal{X}\times \mathbb{V}^\ast)$, $k \geq 2$, to be the neural ODE $\Phi$ with an explicit dependence on the weight vector $v$. As a composition of $k$ times continuously differentiable functions, $\widehat{\Phi}$ is not only in $x$ but also in $v$ $k$ times continuously differentiable. We aim to apply Lemma~\ref{lem:nn_morsefunction} to show that $\Phi$ is, for all weights $v\in \mathbb{V}^\ast$, except possibly for a set of measure zero with respect to the Lebesgue measure in $\mathbb{R}^{N_V}$, a Morse function. To that purpose, we need to show that the matrix $\frac{\partial^2 \widehat{\Phi}}{\partial v \partial x}\in \mathbb{R}^{n \times p}$ has for every $(x,v)\in \mathcal{X}\times \mathbb{V}^\ast$ full rank. The second partial derivatives with respect to $x$ and the components of $\widetilde{W}$ are given by
			\begin{equation*}
				\frac{\partial^2 \widehat{\Phi}}{\partial [\widetilde{W}]_i\partial x} = \frac{\partial }{\partial [\widetilde{W}]_i} \left[ W^\top \left(\frac{\partial h_a(T)}{\partial a}\right)^\top \widetilde{W}^\top \right] = W^\top \left(\frac{\partial h_a(T)}{\partial a}\right)^\top \mathrm{e}_i^{(m)},
			\end{equation*}
			where $\mathrm{e}_i^{(m)}$ denotes the $i$-th unit vector in $\mathbb{R}^m$. Consequently, it follows that 
			\begin{equation*}
				\frac{\partial^2 \widehat{\Phi}}{\partial \widetilde{W}\partial x} =  W^\top \left(\frac{\partial h_a(T)}{\partial a}\right)^\top \in \mathbb{R}^{n \times m},
			\end{equation*}
			has full rank $n$, since by assumption the matrix $W$ has full rank $n$ as $m >n$ and by Proposition~\ref{prop:node_jacobian} the Jacobian $\frac{\partial h_a(T)}{\partial a}$ has full rank $m$. Hence, the weight matrix  $\frac{\partial^2 \widehat{\Phi}}{\partial v \partial x}\in \mathbb{R}^{n \times p}$ has for every $(x,v)\in \mathcal{X}\times \mathbb{V}^\ast$ full rank $n<p$, because the submatrix $\frac{\partial^2 \widehat{\Phi}}{\partial \widetilde{W}\partial x}$ has rank $n$. Lemma~\ref{lem:nn_morsefunction} implies that $\Phi$ is, for all weights $v\in\mathbb{V}^\ast$, except possibly for a zero set with respect to the Lebesgue measure in $\mathbb{R}^{N_V}$, a Morse function. 
		\end{proof}
		
		We illustrate the assertion of the last theorem in a one-dimensional example. 
		
		\begin{example} \label{ex:node_augmented}
			\normalfont Consider the scalar augmented neural ODE $\Phi \in \textup{NODE}^\infty_\textup{A}(\mathbb{R},\mathbb{R})$ with $n=1$, $m = 2$ defined by the weights and biases
			\begin{equation*}
				W = \begin{pmatrix}
					w_1 \\ w_2
				\end{pmatrix} \in \mathbb{R}^{2}, \qquad b = \begin{pmatrix}
					b_{1} \\ b_{2}
				\end{pmatrix} \in \mathbb{R}^2, \qquad \widetilde{W} = (\widetilde{w}_1,\widetilde{w}_2) \in \mathbb{R}^{1\times 2}, \qquad \tilde{b} \in \mathbb{R},
			\end{equation*}
			which is visualized in Figure~\ref{fig:node_augmented}. As an underlying initial value problem, we consider
			\begin{equation*}
				\frac{\dd }{\dd t}\begin{pmatrix}
					h_1 \\ h_2
				\end{pmatrix} = \begin{pmatrix}
					0 \\ \exp(h_1)
				\end{pmatrix}, \qquad \begin{pmatrix}
					h_1 (0)\\ h_2(0)
				\end{pmatrix} = Wx+b,
			\end{equation*}
			with solution
			\begin{equation*}
				\begin{pmatrix}
					h_1 (t)\\ h_2(t)
				\end{pmatrix} = \begin{pmatrix}
					w_1x+b_1 \\ w_2x+b_2 + \exp(w_1x+b_1) t
				\end{pmatrix},
			\end{equation*}
			such that it follows
			\begin{align*}
				\Phi(x) &= \tilde{w}_1 (w_1x+b_1) + \tilde{w}_2(w_2x+b_2 + \exp(w_1x+b_1)T), \\
				\nabla \Phi(x) &= \tilde{w}_1 w_1 + \tilde{w}_2w_2 + \tilde{w}_2 w_1 \exp(w_1x+b_1)T, \\
				H_\Phi(x) &= \tilde{w}_2 w_1^2  \exp(w_1x+b_1)T.
			\end{align*}
			Consequently, $\Phi$ can only be of class $(C3)^\infty(\mathbb{R},\mathbb{R})$ if 
			\begin{equation*}
				\vem(W,\widetilde{W},b,\tilde{b})\in \left\{(w_1,w_2,\tilde{w}_1,\tilde{w}_2,b_1,b_2,\tilde{b}) \in \mathbb{R}^7: \tilde{w}_2 = 0 \vee w_1 = 0 \right\},
			\end{equation*}
			which is a zero set with respect to the Lebesgue measure in $\mathbb{R}^7$. Hence  for all weights $\vem(W,\widetilde{W},b,\tilde{b})\in \mathbb{R}^7$, except possibly for a zero set in $\mathbb{R}^{7}$ with respect to the Lebesgue measure, $\Phi$ is of class $(\mathcal{C}1)^k(\mathcal{X},\mathbb{R})$ or $(\mathcal{C}2)^k(\mathcal{X},\mathbb{R})$.
		\end{example}
		
		\begin{figure}
			\centering
			\includegraphics[width=0.45\textwidth]{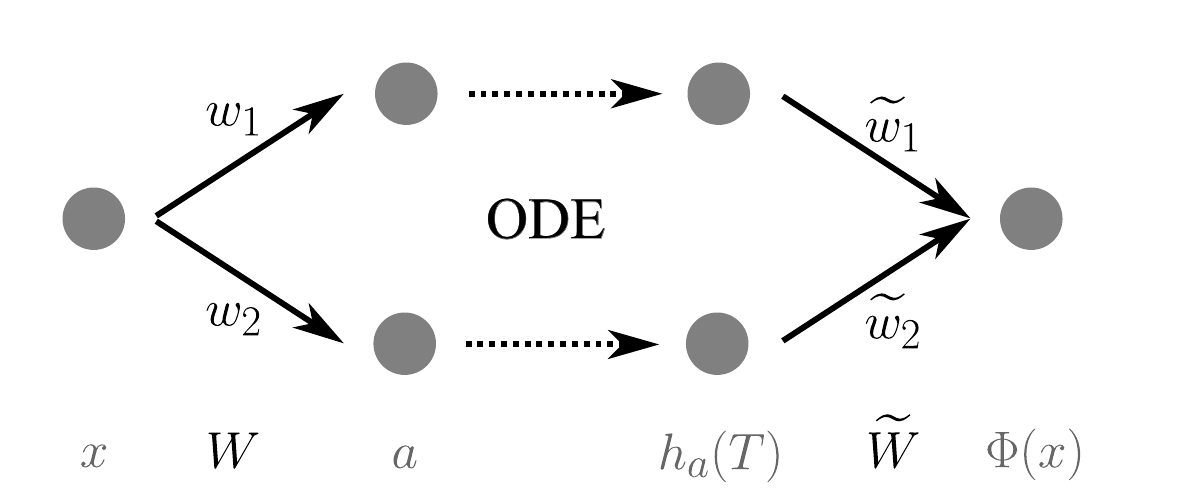}
			\caption{Scalar augmented neural ODE $\Phi \in \textup{NODE}^\infty_\textup{A}(\mathbb{R},\mathbb{R})$ of Example~\ref{ex:node_augmented}.}
			\label{fig:node_augmented}
		\end{figure}
		
		For degenerate neural ODEs, we cannot guarantee that the considered matrix $\frac{\partial^2 \widehat{\Phi}}{\partial v \partial x}$ in the proof of Theorem~\ref{th:node_augmented} has always full rank. In the following, we can directly show that any critical point of a degenerate neural ODE has to be degenerate. To that purpose, we derive a formula for the Hessian matrix $H_\Phi(x)$.
		
		\begin{lemma}\label{lem:node_hessian}
			Let $\Phi \in \textup{NODE}^k(\mathcal{X},\mathbb{R})$, $k\geq 2$, $\mathcal{X}\subset \mathbb{R}^n$ be a scalar neural ODE with weight matrices $(W,\widetilde{W},b,\tilde{b})\in \mathbb{V}$. Then
			\begin{equation*}
				H_\Phi(x) = W^\top H_\eta(s) W \in \mathbb{R}^{n \times n},
			\end{equation*}
			where $a = \lambda(x) = Wx+b$ and $\eta \in C^k(\mathcal{A},\mathbb{R})$, $\eta(s) \coloneqq \widetilde{W} h_a(T)$, $\mathcal{A} = \lambda(\mathcal{X})$.
		\end{lemma}
		
		\begin{proof}
			We introduce $\eta: \mathcal{A} \rightarrow \mathbb{R}$, $\eta(s) \coloneqq\widetilde{W} h_a(T) $, which is by Lemma~\ref{lem:ivp_regularity} of class $C^k(\mathcal{A},\mathbb{R})$ with $\mathcal{A} = \lambda(\mathcal{X})$. It follows 
			\begin{equation*}
				\nabla \eta(x) =  \frac{\partial h_a(T)}{\partial a}^\top \widetilde{W}^\top \in \mathbb{R}^m,
			\end{equation*}
			such that Lemma~\ref{lem:node_gradient} implies $\nabla \Phi(x) = W^\top\nabla\eta(x)$.  Hence, the multi-dimensional chain rule implies
			\begin{align*}
				H_\Phi(x) = \nabla (\nabla \Phi(x)^\top) = \nabla (\nabla \eta(x)^\top W) = W^\top H_\eta(s)W\in \mathbb{R}^{n \times n},
			\end{align*}
			where $a = Wx+b$.
		\end{proof}
		
		Without any analysis of the Hessian matrix $H_\eta$, we can directly conclude that every critical point of a degenerate neural ODE is degenerate, as at least one of the weight matrices $W$ or $\widetilde{W}$ is singular.
		
		\begin{theorem}\label{th:node_bottleneck}
			Any scalar degenerate neural ODE $\Phi \in \textup{NODE}^k_{\textup{D}}(\mathcal{X},\mathbb{R})$ has only degenerate critical points, i.e., $\Phi$ is of class $(\mathcal{C}1)^k(\mathcal{X},\mathbb{R})$ or $(\mathcal{C}3)^k(\mathcal{X},\mathbb{R})$.
		\end{theorem}
		
		\begin{proof}
			As the scalar neural ODE $\Phi$ is degenerate, at least one of the weight matrices $W \in \mathbb{R}^{m \times n}$ and $\widetilde{W} \in \mathbb{R}^{1 \times m}$ has not full rank. 
			If $\widetilde{W} = 0$, then $\Phi(x) = \tilde{b}$ is constant, such that every point $x \in \mathcal{X}$ is a degenerate critical point and hence $\Phi \in (\mathcal{C}3)^k(\mathcal{X},\mathbb{R})$. If $\textup{rank}(W) < \min\{m,n\}$, then it holds
			\begin{equation*}
				\textup{rank}(H_\Phi(x)) \leq \min\{\textup{rank}(W^\top),\textup{rank}(H_\eta(s)),\textup{rank}(W)\}  < \min\{m,n\},
			\end{equation*}
			which implies that the Hessian matrix $H_\Phi(x)$ is singular. Consequently if $\Phi$ has a critical point, then $\Phi \in (\mathcal{C}3)^k(\mathcal{X},\mathbb{R})$ and if $\Phi$ has no critical point, then $\Phi \in (\mathcal{C}1)^k(\mathcal{X},\mathbb{R})$.
		\end{proof}
		
		In Theorem~\ref{th:nn_bottleneck} we showed, that MLPs with a bottleneck can be of all classes $(\mathcal{C}1)^k(\mathcal{X},\mathbb{R})$, $(\mathcal{C}2)^k(\mathcal{X},\mathbb{R})$ and $(\mathcal{C}3)^k(\mathcal{X},\mathbb{R})$. The following remark explains why, nevertheless, the results established for neural networks and neural ODEs are comparable.
		
		\begin{remark}\label{rem:node_bottleneck}
			It is for both scalar neural ODEs and MLPs the case that if the last layer $\widetilde{W}\in\mathbb{R}^{1 \times m}$ or $\widetilde{W}_L\in\mathbb{R}^{1 \times m_L}$ has not full rank, the output is constant and hence the network is of class $(C3)^k(\mathcal{X},\mathbb{R})$. We can compare the case that the weight matrix $W\in\mathbb{R}^{m \times n}$ of the neural ODE has not full rank with the case, that the first weight matrix $W_1 \in \mathbb{R}^{m_1 \times n}$ of the MLP has not full rank. If $W_1$ is singular, still all four types of bottlenecks in Theorem~\ref{th:nn_bottleneck} are possible, but Theorem~\ref {th:nn_classes_coordinatechange}  implies that the neural network cannot be of class $(\mathcal{C}2)^k(\mathcal{X},\mathbb{R})$. This is the case, as a linear change of coordinates to obtain an equivalent neural network architecture is by the proof of Theorem~\ref{th:nn_fullrank_equivalent} needed if and only if the first weight matrix $W_1$ does not have full rank. The difference in the results of neural ODEs and MLPs is induced by the fact that for neural ODEs, the bottleneck obstruction cannot occur in the initial value problem, which corresponds to the hidden layer of the neural network.
		\end{remark}
		
		For the more interesting case, that the first weight matrix $W\in  \mathbb{R}^{m \times n}$ of the scalar neural ODE has not full rank, we provide an example which illustrates that the neural ODE can be both of class $(\mathcal{C}1)^k(\mathcal{X},\mathbb{R})$ or $(\mathcal{C}3)^k(\mathcal{X},\mathbb{R})$.
		
		\begin{figure}[b!]
			\centering
			\includegraphics[width=0.45\textwidth]{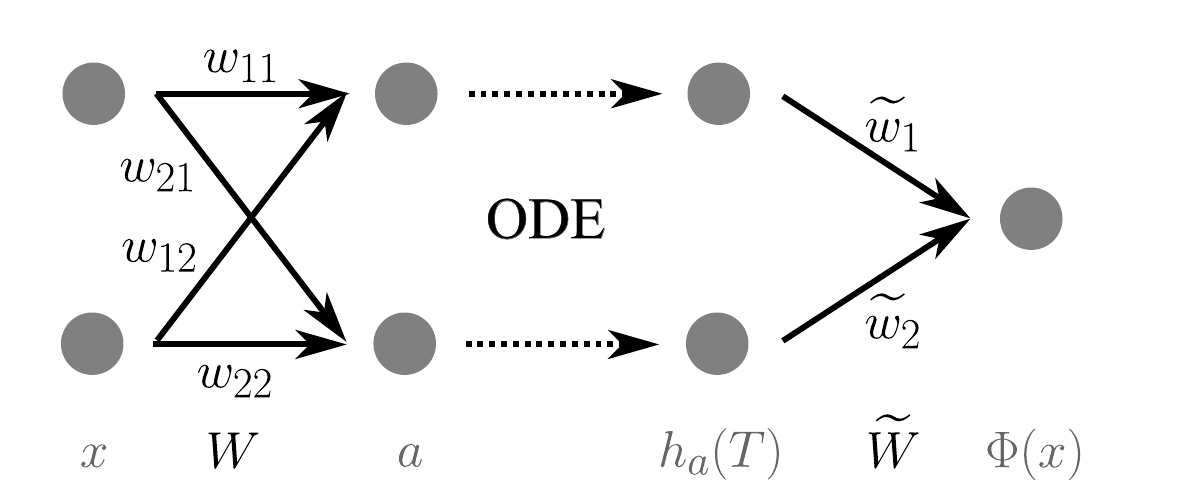}
			\caption{Scalar degenerate neural ODE $\Phi \in  \textup{NODE}^k_{\textup{D}}(\mathbb{R},\mathbb{R})$ of Example~\ref{ex:node_bottleneck}. }
			\label{fig:node_bottleneck}
		\end{figure}

		\begin{example}\label{ex:node_bottleneck}
			\normalfont Consider the scalar degenerate  neural ODE  $\Phi \in \textup{NODE}^k_{\textup{D}}(\mathbb{R},\mathbb{R})$, $k \geq 2$ with $n=2$, $m = 2$ defined by the weights and biases
			\begin{equation*}
				W = \begin{pmatrix}
					w_{11} & w_{12} \\ w_{21} & w_{22}
				\end{pmatrix} \in \mathbb{R}^{2 \times 2}, \qquad b = \begin{pmatrix}
					b_{1} \\ b_{2}
				\end{pmatrix} \in \mathbb{R}^2, \qquad \widetilde{W} = (\widetilde{w}_1,\widetilde{w}_2) \in \mathbb{R}^{1\times 2}, \qquad \tilde{b} \in \mathbb{R},
			\end{equation*}
			where we assume that the neural ODE is degenerate with $\textup{rank}(W) = 1$. The neural ODE $\Phi$ is visualized in Figure~\ref{fig:node_bottleneck}. As $\textup{rank}(W) = 1$, there exists a constant $\alpha$, such that 
			\begin{equation*}
				\begin{pmatrix}
					w_{21} & w_{22}
				\end{pmatrix} =\alpha \cdot \begin{pmatrix}
					w_{11} & w_{12}
				\end{pmatrix} \qquad \Rightarrow \qquad a = \begin{pmatrix}
					a_1 \\ a_2
				\end{pmatrix} = Wx + b = \begin{pmatrix}
					q + b_1\\ \alpha q+b_2
				\end{pmatrix}
			\end{equation*}
			with $q \coloneqq w_{11}x_1 + w_{12}x_2$. Consequently, it holds
			\begin{equation*}
				\Phi(x) = \tilde{w}_1 [h_a(T)]_1 + \tilde{w}_2 [h_a(T)]_2
			\end{equation*}
			and it follows for the gradient
			\begin{equation*}
				\nabla \Phi(x) 
				= \begin{pmatrix}
					w_{11}\left(\frac{\partial \Phi}{\partial a_1} + \alpha\frac{\partial \Phi}{\partial a_2} \right) \\
					w_{12}\left(\frac{\partial \Phi}{\partial a_1} + \alpha\frac{\partial \Phi}{\partial a_2} \right)
				\end{pmatrix}.
			\end{equation*}
			As the columns of $\nabla \Phi$ are linearly dependent, also the columns of the Hessian $H_\Phi$ are linearly dependent, which implies that every possible critical point of $\Phi$ is degenerate. Dependent on the vector field and $\alpha$, the neural ODE $\Phi$ has or does not have critical points. As an example, we consider the underlying initial value problem defined by
			\begin{equation*}
				\frac{\dd }{\dd t}\begin{pmatrix}
					h_1 \\ h_2
				\end{pmatrix} = \begin{pmatrix}
					h_1 \\ h_2
				\end{pmatrix}, \qquad \begin{pmatrix}
					h_1 (0)\\ h_2(0)
				\end{pmatrix} =\begin{pmatrix}
					a_1 \\ a_2
				\end{pmatrix}\qquad 	\Rightarrow \qquad \begin{pmatrix}
					h_1(T) \\ h_2(T)
				\end{pmatrix} = \begin{pmatrix}
					a_1 e^T \\ a_2 e^T
				\end{pmatrix},
			\end{equation*}
			such that 
			\begin{equation*}
				\nabla \Phi(x) 
				= \begin{pmatrix}
					w_{11}\left(\tilde{w}_1+ \alpha \tilde{w}_2 \right)e^T \\
					w_{12}\left(\tilde{w}_1+ \alpha \tilde{w}_2 \right)e^T
				\end{pmatrix}.
			\end{equation*}
			Hence, for $\alpha = -\frac{\tilde{w}_1}{\tilde{w}_2}$ or for $W = 0$, every point $x \in \mathcal{X}$ is a degenerate critical point and $\Phi$ is of class $(\mathcal{C}3)^\infty(\mathcal{X},\mathbb{R})$, otherwise $\Phi$ has no critical point and is of class $(\mathcal{C}1)^\infty(\mathcal{X},\mathbb{R})$.
		\end{example}

		\subsection{General and Parameterized Neural ODEs} \label{sec:NODE_universal}
		
		In the last section, we have seen that it is a generic property of an augmented neural ODE $\Phi \in \textup{NODE}^k_\textup{A}(\mathcal{X},\mathbb{R})$, $\mathcal{X}\subset \R^m$ open, to be a Morse function. If $k \geq n+1 \geq 2$, Morse functions are by Theorem~\ref{th:morsefunctionsdense} a dense subset in the Banach space of $k$ times differentiable functions, such that it is also in the space $C^k(\mathcal{X},\R)$ a generic property to be a Morse function. For augmented neural ODEs with a general non-autonomous vector field, we can not only show that the property of being a Morse function resembles the function space $C^k(\mathcal{X},\R)$, but it even holds that any function  $\Psi \in C^k(\mathcal{X},\mathbb{R})$, $k \geq 1$, can be represented exactly by an augmented neural ODE. Similar statements have already been proven in \cite{Kuehn2023, Zhang2020a} for augmented neural ODE architectures with one linear layer. In our setting, the following theorem shows the universal embedding property of augmented neural ODEs with respect to the space $C^k(\mathcal{X},\R)$, $k\geq 1$.
		
		\begin{theorem}[Universal Embedding of Augmented Neural ODEs]\label{th:node_universalembedding}
			Any map $\Psi \in C^k(\mathcal{X},\mathbb{R})$, $k \geq 1$, $\mathcal{X}\subset\mathbb{R}^n$ can be exactly represented by an augmented scalar neural ODE $\Phi \in \textup{NODE}^k_\textup{A}(\mathcal{X},\mathbb{R})$ with arbitrary $m>n$, i.e., $\Psi(x) = \Phi(x)$ for all $x \in \mathcal{X}$. 
		\end{theorem}
		
		\begin{proof}
			In analogy to the proof in \cite{Kuehn2023}, we fix $T>0$, choose $m>n$ and define the augmented neural ODE $\Phi$ by the weights
			\begin{equation*}
				W = \begin{pmatrix}
					\textup{Id}_n \\ 0
				\end{pmatrix} \in \mathbb{R}^{m\times n}, \quad b = 0\in\mathbb{R}^m,\quad\widetilde{W} = (0,\ldots,0,1)  \in \mathbb{R}^{1 \times m}, \quad \tilde{b} = 0 \in \mathbb{R}
			\end{equation*}
			and the ODE
			\begin{equation*}
				\frac{\dd }{\dd t} \begin{pmatrix}
					h_1 \\ \vdots \\ h_{m-1} \\ h_m
				\end{pmatrix}
				= \begin{pmatrix}
					0 \\ \vdots \\ 0 \\
					\frac{1}{T} \cdot \Psi(h_1,\ldots.h_n)
				\end{pmatrix} \eqqcolon f(t,h_1,\ldots,h_n),
			\end{equation*} 
			such that the vector field $f$ is of class $C^{0,k}(\mathbb{R} \times \mathcal{X},\mathbb{R})$. $\Phi$ is augmented, as $m>n$ and the weights matrices $W$ and $\widetilde{W}$ have full rank. By Lemma~\ref{lem:node_regularity} it holds $\Phi\in \textup{NODE}^k_\textup{A}(\mathcal{X},\mathbb{R})$ and it follows
			\begin{equation*}
				\Phi(x) = \widetilde{W} \cdot h_{Wx+b}(T) +\tilde{b}= \widetilde{W} \cdot h_{(x,0)^\top}(T) = \widetilde{W}\begin{pmatrix}
					x \\ 0 \\ \Psi(x)
				\end{pmatrix} = \Psi(x)
			\end{equation*}
			for all $x \in \mathcal{X}$.
		\end{proof}
		
		To prove Theorem~\ref{th:node_universalembedding}, we used the fact that we can freely choose the vector field of the underlying initial value problem. Often in practice, the vector field of the neural ODE is itself a neural network, or, more generally, a parameterized vector field. In the following, we show how to use the universal embedding property of general augmented neural ODEs to prove a universal approximation result for parameterized augmented neural ODEs. To that purpose, it is necessary that the parameterized vector field is expressive enough, such that the vector field can be chosen sufficiently close in the supremum norm to the vector field constructed in the proof of Theorem~\ref{th:node_universalembedding}. In the following, we define a parameterized neural ODE depending on a parameter function $\theta: \R \rightarrow \R^p$.

			\begin{definition}[Parameterized Neural ODE] \label{def:node_parameter}
				For $k \geq 1$, the set of all scalar parameterized neural ODE architectures $\Phi_\theta: \mathcal{X} \rightarrow \mathbb{R}$, $\mathcal{X}\subset \mathbb{R}^n$ open, defined by~\eqref{eq:NODE} based on the underlying initial value problem 
				\begin{equation*}  
					\frac{\dd h}{\dd t} = f_\theta(t,h(t),\theta(t)), \qquad h(0) = a \in \mathcal{A} \subset \mathbb{R}^m,
				\end{equation*}
				with vector field $f_\theta: \Omega_t \times \Omega_h \times \R^p \rightarrow \R^m$, $\Omega_t \times \Omega_h \subset \R \times \R^m$ open, parameter function $\theta: \R \rightarrow \R^p$, $\mathcal{A} = \lambda(\mathcal{X})\subset \Omega_h$, is denoted by $\textup{NODE}^k_\theta(\mathcal{X},\mathbb{R}) \subset C^k(\mathcal{X},\mathbb{R})$ if the initial value problem~\eqref{eq:IVP} with vector field
				\begin{equation} \label{eq:IVPpar}
					f: \Omega_t \times \Omega_h \rightarrow \R^m, \quad f(t,h(t)) \coloneqq f_\theta(t,h(t),\theta(t)), \qquad f\in C^{0,k}(\Omega_t\times \Omega_h, \R^m),\tag{$\text{IVP}_\text{par}$}
				\end{equation}
				defines a well-defined neural ODE, i.e., $[0,T] \subset \mathcal{I}_a\subset \Omega_t$ for all $a \in \mathcal{A}$. 
				
				For a fixed vector field $f_\theta: \Omega_t \times \Omega_h \times \R^p \rightarrow \R^m$, the set of all parameter functions $\theta:\mathbb{R}\rightarrow \mathbb{R}^{p}$, such that the corresponding neural ODE is an element of $\textup{NODE}_{\theta}^k(\mathcal{X},\mathbb{R})$, $k \geq 1$, is denoted by $\Theta^k(\mathbb{R},\mathbb{R}^p)$.
			\end{definition}
			
			\begin{remark}
				The regularity of a parameterized neural ODE is defined via the regularity of a general neural ODE. In order to be $k$ times continuously differentiable, both the vector field $f_\theta$ and the parameter function $\theta: \R \rightarrow \R^p$ need to be sufficiently regular. Using the theory of Carathéodory ODEs, the regularity of $f$ can be shown, for example, for functions $\theta$, which are only piece-wise continuous \cite{Hale1980}. Hence, the parameter $k$ of the space $\Theta^k(\mathbb{R},\mathbb{R}^p)$ does not describe the regularity of the parameter functions included in that space, but the regularity of the corresponding neural ODE. This is of particular interest if the vector field of a parameterized neural ODE is chosen itself as a neural network, where the weights can change throughout the network, leading to discontinuities in the parameter function. 
			\end{remark}
			
			In the following, we transfer the (non-)universal embedding property of a general neural ODE architecture  $\textup{NODE}^k(\mathcal{X},\mathbb{R})$ of Definition~\ref{def:node} to a (non-)universal approximation property of a parameterized neural ODE architecture $\textup{NODE}^k_\theta(\mathcal{X},\mathbb{R})$ of Definition~\ref{def:node_parameter}. In that way, we show that our results, proven for the exact input-output map, can be transferred to approximate results for parameterized neural ODEs, which are mainly used in practice. On the one hand side we prove in part~\ref{th:ODEgerenalparameterized_a} that if the general neural ODE architecture  $\textup{NODE}^k(\mathcal{X},\mathbb{R})$ does not have the universal approximation (embedding) property, then independently of the parameterization, every neural ODE architecture $\textup{NODE}_{\theta}^k(\mathcal{X},\mathbb{R})$ also does not have the universal approximation (embedding) property. On the other side,  part~\ref{th:ODEgerenalparameterized_b} shows that if the parameter function $\theta$ is chosen in such a way, that the vector field $f_\theta$ approximates every general vector field $f$ arbitrary well, then the parameterized architecture has the universal approximation property if the general neural ODE architecture has the universal embedding or approximation property.
			
			\begin{theorem}[Relationship General and Parameterized Neural ODE]\label{th:node_approx}
				Let $\mathcal{X}\subset \mathbb{R}^n$ open and consider the general neural ODE architecture $\textup{NODE}^k(\mathcal{X},\mathbb{R})$, $k\geq 1$, based on~\eqref{eq:IVP} and the parameterized neural ODE architecture $\textup{NODE}^i_{\theta}(\mathcal{X},\mathbb{R})$, $i\geq 1$ based on~\eqref{eq:IVPpar}.
				\begin{enumerate}[label=(\alph*), font=\normalfont]
					\item \label{th:ODEgerenalparameterized_a} If the neural ODE architecture $\textup{NODE}^k(\mathcal{X},\mathbb{R})$, $k \geq 1$,  does not have the universal approximation (embedding) property with respect to the space $C^j(\mathcal{X}, \mathbb{R})$, $j\geq 0$, then also the architecture $\textup{NODE}^{k}_{\theta}(\mathcal{X},\mathbb{R})$ does not have the universal approximation (embedding) property with respect to the space $C^j(\mathcal{X}, \mathbb{R})$.
					\item \label{th:ODEgerenalparameterized_b} Let the following three assumptions hold:
					\begin{itemize}
						\item The neural ODE architecture $\textup{NODE}^k(\mathcal{X},\mathbb{R})$, $k\geq 1$, has the universal approximation or embedding property with respect to the space $C^j(\mathcal{X}, \mathbb{R})$, $j \geq 0$.
						\item The vector field $f_\theta: \Omega_t \times \Omega_h \times \R^p \rightarrow \R^m$ and a space of parameters $\Theta^i(\mathbb{R},\mathbb{R}^{p})$, $i \geq 1$, are fixed.
						\item For any vector field $f \in C^{0,k}(\Omega_t \times \Omega_y,\mathbb{R}^m)$ corresponding to the architecture $\textup{NODE}^k(\mathcal{X},\mathbb{R})$, and every  $\delta >0$, there exists $\theta \in \Theta^i(\mathbb{R},\mathbb{R}^{p})$, such that 
						\begin{equation*}
							\norm{f(t,h(t))-f_\theta(t,h(t),\theta(t))}_\infty < \delta \quad \textup{ for all } (t,h(t))\in [0,T] \times \Omega_h \subset \Omega_t \times \Omega_h.
						\end{equation*}
					\end{itemize}
					Then also the neural ODE architecture $\textup{NODE}^i_{\theta}(\mathcal{X},\mathbb{R})$ has the universal approximation property with respect to the space $C^j(\mathcal{X}, \mathbb{R})$. 
				\end{enumerate}
			\end{theorem}
			
			\begin{proof}
				The proof is similar to \cite[Theorem 3.4]{Kuehn2025}, where an analogous statement is proven for neural delay differential equations. For completeness, we state a simplified proof for the neural ODE case here.
				
				Part~\ref{th:ODEgerenalparameterized_a} is proven by contraposition: let $j\geq 0$, $k\geq 0$ and assume that the parameterized architecture $\textup{NODE}^{k}_{\theta}(\mathcal{X},\mathbb{R})$ has the universal approximation property with respect to the space $C^j(\mathcal{X},\mathbb{R})$. Hence, for any map $\Psi\in C^j(\mathcal{X}, \mathbb{R})$, there exists a choice of the parameter function $\theta\in \Theta^k(\mathbb{R},\mathbb{R}^{p})$, such that the corresponing neural ODE $\Phi_\theta \in\textup{NODE}^k_{\theta}(\mathcal{X},\mathbb{R})$ satisfies
				\begin{equation}\label{eq:approximation}
					\norm{\Phi_\theta(x)-\Psi(x)}_\infty < \varepsilon \quad \textup{ for all } x \in \mathcal{X}. 
				\end{equation}
				The neural ODE $\Phi_\theta \in\textup{NODE}^k_{\theta}(\mathcal{X},\mathbb{R})$ is based on a parameterized vector field  $f_\theta: \Omega_t \times \Omega_h \times \R^p \rightarrow \R^m$ with fixed parameter function $\theta\in \Theta^k(\mathbb{R},\mathbb{R}^{p})$. By Definition~\ref{def:node_parameter}, the vector field 
				\begin{equation*}
					f: \Omega_t \times \Omega_h \rightarrow \mathbb{R}^m, \quad f(t,h(t)) \coloneqq f_\theta(t,h(t),\theta(t))
				\end{equation*}
				leads to a well defined neural ODE $\Phi \in \textup{NODE}^k(\mathcal{X},\mathbb{R})$. By construction $\Phi(x) = \Phi_\theta(x)$ for all $x \in \mathcal{X}$, such that also the class of general neural ODE architectures $\textup{NODE}^k(\mathcal{X},\mathbb{R})$ has the universal approximation property with respect to the space $C^j(\mathcal{X},\mathbb{R})$, as  $\Psi\in C^j(\mathcal{X}, \mathbb{R})$ was arbitrary. The statement follows, as this is a contradiction to the assumption that the class  $\textup{NODE}^k(\mathcal{X},\mathbb{R})$ does not have the universal approximation property. The same argumentation holds verbatim for the universal embedding property if the definition of approximation in~\eqref{eq:approximation} is replaced by the definition of embedding, i.e., $\Phi_\theta(x) = \Psi(x)$ for all $x \in \mathcal{X}$.
				
				To prove part~\ref{th:ODEgerenalparameterized_b}, consider a function $\Psi \in C^j(\mathcal{X},\mathbb{R})$ for a fixed $j \geq 0$. As the class $\textup{NODE}^k(\mathcal{X},\mathbb{R})$ has the universal approximation or embedding property for a fixed $k\geq 0$, there exists for every given $\eps>0$ a neural ODE $\Phi \in \textup{NODE}^k(\mathcal{X},\mathbb{R})$ with
				\begin{equation*}
					\norm{\Phi(x)-\Psi(x)}_\infty  < \frac{\varepsilon}{2} \quad \textup{ for all } x \in \mathcal{X}.
				\end{equation*}
				If it holds for the weight matrix of the affine linear layer $\tilde{\lambda}$ that $\widetilde{W} = 0 \in \R^{1 \times m}$, then set $\delta = 1$, otherwise, let $\delta = \frac{\eps}{2 \normk{\widetilde{W}}_\infty T} >0$. By assumption, there exists for the vector field $f: \Omega_t \times \Omega_h \rightarrow \mathbb{R}^m$ of the neural ODE~$\Phi$ a parameter function $\theta\in\Theta^i(\mathbb{R},\mathbb{R}^{p})$, $i \geq 0$, such that
				\begin{equation*}
					\norm{f(t,h(t))-f_\theta(t,h(t),\theta(t))}_\infty < \delta \quad \textup{ for all } (t,h(t))\in [0,T] \times \Omega_h \subset \Omega_t \times \Omega_h.
				\end{equation*}
				It follows for the neural ODE $\Phi_\theta\in\textup{NODE}^i_{\tau,\theta}(\mathcal{X},\mathbb{R})$, based on the fixed vector field $f_\theta$ and the parameter function $\theta\in\Theta^i(\mathbb{R},\mathbb{R}^{p})$ chosen before that
				\begin{align*}
					\sup_{x \in \mathcal{X}} \norm{\Phi_\theta(x)-\Psi(x)}_\infty &\leq  \sup_{x \in \mathcal{X}} \norm{\Phi(x)-\Psi(x)}_\infty + \sup_{x \in \mathcal{X}} \norm{\Phi_\theta(x)-\Phi(x)}_\infty\\
					&< \frac{\eps}{2} +  \sup_{x \in \mathcal{X}} \norm{\tilde\lambda(h_{\theta,\lambda(x)}(T)) - \tilde\lambda(h_{\lambda(x)}(T))}_\infty,
				\end{align*}
				where we denote by $h_{\lambda(x)}$ the solution corresponding to $\Phi$ and by $h_{\theta,\lambda(x)}$ the solution corresponding to $\Phi_\theta$. Now we can estimate
				\begin{align*}
					&\sup_{x \in \mathcal{X}} \norm{\tilde\lambda(h_{\theta,\lambda(x)}(T)) - \tilde\lambda(h_{\lambda(x)}(T))}_\infty \\
					=\, &\sup_{x \in \mathcal{X}} \; \norm{\widetilde{W}\left( Wx+b + \int_0^T f(t,h(t))  \; \dd t \right) + \tilde{b} - \widetilde{W}\left( Wx+b + \int_0^T f_\theta(t,h(t),\theta(t)) \; \dd t \right) - \tilde{b} \, }_\infty \\
					\leq \,&\sup_{x \in \mathcal{X}} \; \norm{\widetilde{W}}_\infty \norm{\int_0^T f(t,h(t)) - f_\theta(t,h(t),\theta(t))  \; \dd t \,  }_\infty \\
					\leq \,&\sup_{x \in \mathcal{X}} \; \norm{\widetilde{W}}_\infty \int_0^T \norm{ f(t,h(t)) - f_\theta(t,h(t),\theta(t)) }_\infty \dd t < \norm{\widetilde{W}}_\infty T \delta = \frac{\eps}{2}.
				\end{align*}
				Consequently it holds
				\begin{equation*}
					\sup_{x \in \mathcal{X}} \norm{\Phi_\theta(x)-\Psi(x)}_\infty  < \frac{\eps}{2} + \frac{\eps}{2} = \eps,
				\end{equation*}
				and as $\eps>0$ and $\Psi\in C^j(\mathcal{X},\mathbb{R})$  were arbitrary, also the parameterized architecture $\textup{NODE}^i_{\theta}(\mathcal{X},\mathbb{R})$ has the universal approximation property with respect to the space $C^j(\mathcal{X},\mathbb{R})$.
			\end{proof}
			
			Combined with our results proven in Sections~\ref{sec:node_criticalpoints} and~\ref{sec:node_regularity}, Theorem~\ref{th:node_approx} implies the following: general non-augmented neural ODEs $\textup{NODE}_\text{N}^k(\mathcal{X},\R)$ can by Lemma~\ref{lem:node_gradient} only be of class $(\mathcal{C}1)^k(\mathcal{X},\R)$, and general degenerate neural ODEs $\textup{NODE}_\text{D}^k(\mathcal{X},\R)$ can by Theorem~\ref{th:node_bottleneck} only be of class $(\mathcal{C}1)^k(\mathcal{X},\R)$ or $(\mathcal{C}3)^k(\mathcal{X},\R)$, hence non-augmented and degenerate neural ODEs cannot have the universal embedding property with respect to the space $C^k(\mathcal{X},\R)$. Part~\ref{th:ODEgerenalparameterized_a} now shows, that also parameterized non-augmented and degenerate neural ODEs $\textup{NODE}_{\theta,\text{N}}^k(\mathcal{X},\R)$ and $\textup{NODE}_{\theta,\text{D}}^k(\mathcal{X},\R)$ cannot have the universal embedding property with respect to the space $C^k(\mathcal{X},\R)$. Combined with Corollary~\ref{cor:uniapprox} we infer, that parameterized non-augmented neural ODEs $\textup{NODE}_{\theta,\text{N}}^k(\mathcal{X},\R)$ also cannot have the universal approximation property with respect to the space  $C^k(\mathcal{X},\R)$. 
			
			In the case of augmented neural ODEs, we have show in Theorem~\ref{th:node_universalembedding}, that the architecture $\textup{NODE}_\text{A}^k(\mathcal{X},\R)$ has the universal embedding property with respect to the space $C^k(\mathcal{X},\R)$. Part~\ref{th:ODEgerenalparameterized_b} of the last theorem now implies, that if the parameterized vector field of the architecture $\textup{NODE}_{\theta,\text{A}}^k(\mathcal{X},\R)$ fulfills the given assumptions, then $\textup{NODE}_{\theta,\text{A}}^k(\mathcal{X},\R)$ has the universal approximation property.

		\vspace{3mm}
		
		\section{Conclusion and Outlook}
		
		In this work, we studied the geometric structure of the input-output map of MLPs and neural ODEs. We used the definition of a Morse function, which is a scalar function, where every critical point is non-degenerate. It is interesting to study the property of being a Morse function, as Morse functions are dense in the space of sufficiently smooth scalar functions and are fundamental building blocks within many areas of mathematics. We fully classified the existence and regularity of critical points depending on the specific architecture of the MLP or the neural ODE.
		
		The architectures studied are subdivided into non-augmented,  augmented, and  bottleneck/degenerate architectures. For MLPs, the subdivision is defined by the width of the largest layer and by checking if there exists a bottleneck in the underlying graph. If at least one weight matrix of a neural network $\Phi$ has not full rank, there exists up to a linear change of coordinates a normal form $\overline{\Phi}$, which is an MLP with only full rank matrices and which is equivalent to $\Phi$, i.e., it has the same input-output map. Our proof of the existence of a normal form is constructive, i.e., we gave an explicit algorithm to calculate an equivalent normal form~$\overline{\Phi}$. Furthermore, we proved that the case where all weight matrices have full rank is generic, i.e., the set of weights, where at least one weight matrix does not have full rank, has Lebesgue measure zero in the weight space. Hence, we focused in the following analysis on neural networks in normal form.
		
		For neural ODEs, we defined an analogous subdivision in non-augmented, augmented, and degenerate  architectures. As neural ODEs can be interpreted as an infinite depth limit of MLPs, the vector field $f$ of the neural ODE represents all weights, biases, and activation functions of the hidden layers. We studied the most general case, where $f$ can be any non-autonomous vector field. In analogy to MLPs, the dimension of the phase space defines whether an architecture is augmented or non-augmented. We say that a neural ODE is degenerate if at least one of the weight matrices of the affine linear layers before and after the initial value problem does not have full rank. 
		
		Our main contribution of this work is to classify in each case the existence and regularity of critical points. In both non-augmented cases, no critical points can exist, as the network gradient has always full rank, which we proved via rank arguments and linear variational equations. If the dimension of at least one layer or the dimension of the phase space is augmented, we showed that it is possible that critical points exist in both cases. Using differential geometry and Morse theory, we showed that, except for a zero set with respect to the Lebesgue measure in the weight space, the input-output map is a Morse function. For MLPs with a bottleneck, we showed that critical points can be degenerate or non-degenerate. Depending on the dimension of the first bottleneck and the weight matrices, we derived conditions specifying the regularity of the critical points depending on the type of bottleneck. In contrast to that, we proved for degenerate neural ODEs that every critical point is degenerate. The different classification for MLPs with a bottleneck and degenerate neural ODEs relies on the fact that for neural ODEs, the dimension of the vector field is constant and no bottleneck can occur in the hidden layer. We explain why, nevertheless, the results obtained for degenerate neural ODEs are comparable to the results for MLPs, where the first weight matrix does not have full rank. Overall, our results show that the geometric structure of the input-output map is comparable for MLPs and for neural ODEs.
		
		Besides the classification of critical points, we also studied the implications of our results on the universal embedding and approximation property of MLPs and neural ODEs. For neural ODEs, we showed that they cannot have the universal embedding property in the non-augmented and the degenerate case. In contrast, augmented neural ODEs have the universal embedding property. Hereby, it was used that  $f$ can be any non-autonomous vector field. In the case of an expressive parameterized vector field, augmented parameterized neural ODEs have the universal approximation property. In the non-augmented case, MLPs cannot have the universal embedding property. For augmented MLPs, there already exist various universal approximation theorems in the literature. For the bottleneck case, we distinguished between cases where no universal embedding is possible and cases where the input-output dynamics resembles augmented neural networks, such that it is possible that the neural network has the universal approximation property. 
		
		The focus on the input-output dynamics of MLPs and neural ODEs explains why certain architectures perform better than others. Our analysis showed how the dimensions of the layers and the phase space, the rank of the weight matrices, and obstructions such as bottlenecks influence whether the input-output map is a Morse function. The established results are used as a starting point for a perturbation analysis to derive (non-)universal approximation results. In future work, it would be interesting to study the geometric implications if an MLP or a neural ODE is a Morse function. Morse functions are used across many areas, such as topology, differential geometry, and data analysis, to study the underlying manifold on which the Morse function is defined. Transferred to our setting, it is of interest to study the manifold of input data or initial conditions with an augmented neural network defined on it, which is generically a Morse function. Beyond the evident connection to data geometry and its applications, one could even aim to consider, from a very theoretical viewpoint, whether neural network representations of Morse functions could be helpful to solve open conjectures in Morse, Morse-Bott, and Floer homology. In this context, it will also become important to study not only the existence of critical points but to obtain more precise quantitative global estimates regarding the number of critical points generated by a neural network.

		\vspace{5mm}
		
		\noindent\textbf{Acknowledgments.} CK and SVK acknowledge funding from the Deutsche Forschungsgemeinschaft (DFG) via the SPP2298 `Theoretical Foundations of Deep Learning'. CK would like to thank the VolkswagenStiftung for support via a Lichtenberg Professorship. SVK would like to thank the Munich Data Science Institute (MDSI) for partial support via a Linde doctoral fellowship. CK and SVK also appreciate the valuable feedback and suggestions of an anonymous referee, which led to significant improvements in the article.\\

		\noindent\textbf{Data Availability.} No datasets were generated or analyzed during the current study. \\
	
		\noindent\textbf{Conflict of Interest.} The authors have no competing interests.
		
		
		
		
		\vspace{3mm}
		
		\addcontentsline{toc}{section}{References}
		\bibliographystyle{abbrvurl}
		\bibliography{0literature}

@Book{Aggarwal2018,
  author    = {Charu C. Aggarwal},
  publisher = {Springer},
  title     = {Neural Networks and Deep Learning},
  year      = {2018},
  edition   = {1},
  doi       = {10.1007/978-3-319-94463-0},
}

@Book{Forster2017,
  author    = {Otto Forster},
  publisher = {Springer Spektrum},
  title     = {Analysis 2, Differentialrechnung im $\mathbb{R}^n$, gewöhnliche Differentialgleichungen},
  year      = {2017},
  edition   = {11},
  series    = {Grundkurs Mathematik},
  doi       = {10.1007/978-3-658-19411-6},
}

@Book{Hirsch1976,
  author    = {Morris W. Hirsch},
  publisher = {Springer New York},
  title     = {Differential Topology},
  year      = {1976},
  series    = {Graduate Texts in Mathematics},
  volume    = {33},
  doi       = {10.1007/978-1-4684-9449-5},
}

@Article{Kratsios2021,
  author    = {Anastasis Kratsios},
  journal   = {Annals of Mathematics and Artificial Intelligence},
  title     = {The Universal Approximation Property - Characterization, Construction, Representation, and Existence},
  year      = {2021},
  number    = {5-6},
  pages     = {435-469},
  volume    = {89},
  doi       = {10.1007/s10472-020-09723-1},
  publisher = {Springer Science and Business Media {LLC}},
}

@Book{Morse1934,
  author    = {Marston Morse},
  publisher = {American Mathematical Society},
  title     = {The Calculus of Variations in the Large},
  year      = {1934},
  series    = {Colloquium Publications},
  volume    = {18},
}

@Article{Rosenblatt1957,
  author  = {Frank Rosenblatt},
  journal = {Cornell Aeronautical Laboratory, INC., Buffalo, New York},
  title   = {The Perceptron - A Perceiving and Recognizing Automaton},
  year    = {1957},
  number  = {460-1},
  volume  = {85},
}

@Article{Zhang2020a,
  author    = {Zhang, Han and Gao, Xi and Unterman, Jacob and Arodz, Tom},
  journal   = {Proceedings of the 37th International Conference on Machine Learning},
  title     = {Approximation Capabilities of Neural {ODEs} and Invertible Residual Networks},
  year      = {2020},
  pages     = {11086-11095},
  volume    = {119},
  copyright = {arXiv.org perpetual, non-exclusive license},
  doi       = {10.48550/ARXIV.1907.12998},
  publisher = {arXiv},
}

@Book{Guckenheimer2002,
  author    = {John Guckenheimer and Philip Holmes},
  publisher = {Springer New York},
  title     = {Nonlinear Oscillations, Dynamical Systems, and Bifurcations of Vector Fields},
  year      = {2002},
  edition   = {7},
  series    = {Applied Mathematical Sciences},
  volume    = {42},
  doi       = {10.1007/978-1-4612-1140-2},
}

@Article{Weinan2017,
  author    = {E Weinan},
  journal   = {Commun. Math. Stat},
  title     = {A Proposal on Machine Learning via Dynamical Systems},
  year      = {2017},
  pages     = {1-11},
  volume    = {5},
  doi       = {10.1007/s40304-017-0103-z},
  publisher = {Springer Science and Business Media {LLC}},
}

@Article{Chen2018,
  author    = {Chen, Ricky T. Q. and Rubanova, Yulia and Bettencourt, Jesse and Duvenaud, David},
  journal   = {NeurIPS},
  title     = {Neural Ordinary Differential Equations},
  year      = {2018},
  copyright = {arXiv.org perpetual, non-exclusive license},
  doi       = {10.48550/ARXIV.1806.07366},
  publisher = {arXiv},
}

@Book{Chicone2006,
  author    = {Carmen Chicone},
  publisher = {Springer New York},
  title     = {Ordinary Differential Equations with Applications},
  year      = {2006},
  edition   = {2},
  series    = {Texts in Applied Mathematics},
  volume    = {34},
  doi       = {10.1007/0-387-35794-7},
}

@Article{He2016,
  author    = {Kaiming He and Xiangyu Zhang and Shaoqing Ren and Jian Sun},
  journal   = {IEEE Conference on Computer Vision and Pattern Recognition},
  title     = {Deep Residual Learning for Image Recognition},
  year      = {2016},
  pages     = {770-778},
  booktitle = {2016 {IEEE} Conference on Computer Vision and Pattern Recognition ({CVPR})},
  doi       = {10.1109/cvpr.2016.90},
  publisher = {{IEEE}},
}

@Article{Hornik1989,
  author    = {Kurt Hornik and Maxwell Stinchcombe and Halbert White},
  journal   = {Neural Networks},
  title     = {Multilayer feedforward networks are universal approximators},
  year      = {1989},
  number    = {5},
  pages     = {359-366},
  volume    = {2},
  doi       = {10.1016/0893-6080(89)90020-8},
  publisher = {Elsevier {BV}},
}

@PhdThesis{Kidger2022,
  author    = {Patrick Kidger},
  school    = {Mathematical Institute, University of Oxford},
  title     = {On Neural Differential Equations},
  year      = {2022},
  copyright = {arXiv.org perpetual, non-exclusive license},
  doi       = {10.48550/ARXIV.2202.02435},
  publisher = {arXiv},
}

@Article{Lin2018,
  author    = {Lin, Hongzhou and Jegelka, Stefanie},
  journal   = {Advances in Neural Information Processing Systems},
  title     = {{ResNet} with one-neuron hidden layers is a Universal Approximator},
  year      = {2018},
  pages     = {6169-6178},
  volume    = {31},
  copyright = {arXiv.org perpetual, non-exclusive license},
  doi       = {10.48550/ARXIV.1806.10909},
  publisher = {arXiv},
}

@Article{Pinkus1999,
  author    = {Allan Pinkus},
  journal   = {Acta Numerica},
  title     = {Approximation theory of the {MLP} model in neural networks},
  year      = {1999},
  pages     = {143-195},
  volume    = {8},
  doi       = {10.1017/s0962492900002919},
  publisher = {Cambridge University Press ({CUP})},
}

@InBook{Schaefer2006,
  author    = {Anton Maximilian Schäfer and Hans Georg Zimmermann},
  pages     = {632-640},
  publisher = {Springer Berlin Heidelberg},
  title     = {Recurrent Neural Networks Are Universal Approximators},
  year      = {2006},
  booktitle = {Artificial Neural Networks {\textendash} {ICANN} 2006},
  doi       = {10.1007/11840817_66},
}

@Article{Kuehn2023,
  author    = {Kuehn, Christian and Kuntz, Sara-Viola},
  journal   = {Preprint},
  title     = {Embedding Capabilities of Neural {ODEs}},
  year      = {2023},
  copyright = {Creative Commons Attribution 4.0 International},
  doi       = {10.48550/ARXIV.2308.01213},
  publisher = {arXiv},
}

@Book{Hale1980,
  author    = {Jack K. Hale},
  publisher = {Krieger Publishing Company},
  title     = {Ordinary Differential Equations},
  year      = {1980},
  edition   = {2},
}

@Article{Kurochkin2021,
  author    = {Kurochkin, S. V.},
  journal   = {Computational Mathematics and Mathematical Physics},
  title     = {Neural Network with Smooth Activation Functions and without Bottlenecks Is Almost Surely a Morse Function},
  year      = {2021},
  issn      = {1555-6662},
  number    = {7},
  pages     = {1162-1168},
  volume    = {61},
  doi       = {10.1134/s0965542521070101},
  publisher = {Pleiades Publishing Ltd},
}

@Book{Nicolaescu2011,
  author    = {Nicolaescu, Liviu},
  publisher = {Springer New York},
  title     = {An Invitation to Morse Theory},
  year      = {2011},
  isbn      = {9781461411055},
  doi       = {10.1007/978-1-4614-1105-5},
  issn      = {2191-6675},
  journal   = {Universitext},
}

@Book{Prasolov2006,
  author    = {Prasolov, V.},
  publisher = {American Mathematical Society},
  title     = {Elements of Combinatorial and Differential Topology},
  year      = {2006},
  isbn      = {9781470411534},
  doi       = {10.1090/gsm/074},
  issn      = {1065-7339},
  journal   = {Graduate Studies in Mathematics},
}

@Article{Thome2016,
  author    = {Thome, Néstor},
  journal   = {Aequationes mathematicae},
  title     = {Inequalities and equalities for $l=2$ (Sylvester), $l=3$ (Frobenius), and $l>3$ matrices},
  year      = {2016},
  issn      = {1420-8903},
  number    = {5},
  pages     = {951--960},
  volume    = {90},
  doi       = {10.1007/s00010-016-0412-4},
  publisher = {Springer Science and Business Media LLC},
}

@Book{Magnus2019,
  author    = {Jan R. Magnus and Heinz Neudecker},
  publisher = {Wiley},
  title     = {Matrix Differential Calculus with Applications in Statistics and Econometrics},
  year      = {2019},
  edition   = {3},
  isbn      = {9781119541219},
  doi       = {10.1002/9781119541219},
  issn      = {1940-6347},
  journal   = {Wiley Series in Probability and Statistics},
}

@Book{Shafarevich2013,
  author    = {Shafarevich, Igor R. and Remizov, Alexey O.},
  publisher = {Springer Berlin Heidelberg},
  title     = {Linear Algebra and Geometry},
  year      = {2013},
  isbn      = {9783642309946},
  doi       = {10.1007/978-3-642-30994-6},
}

@Article{Esteve2020,
  author    = {Esteve, Carlos and Geshkovski, Borjan and Pighin, Dario and Zuazua, Enrique},
  title     = {Large-time asymptotics in deep learning},
  year      = {2020},
  copyright = {arXiv.org perpetual, non-exclusive license},
  doi       = {10.48550/ARXIV.2008.02491},
  publisher = {arXiv},
}

@Article{Cook2022,
  author    = {Cook, Blake J. and Peterson, Andre D. H. and Woldman, Wessel and Terry, John R.},
  journal   = {Mathematical Neuroscience and Applications},
  title     = {Neural Field Models: A mathematical overview and unifying framework},
  year      = {2022},
  issn      = {2801-0159},
  volume    = {Volume 2},
  doi       = {10.46298/mna.7284},
  publisher = {Centre pour la Communication Scientifique Directe (CCSD)},
}

@Book{Goodfellow2016,
  author    = {Ian Goodfellow and Yoshua Bengio and Aaron Courville},
  publisher = {MIT Press},
  title     = {Deep Learning},
  year      = {2016},
  note      = {\url{http://www.deeplearningbook.org}},
}

@Article{Johnson2018,
  author    = {Johnson, Jesse},
  title     = {Deep, Skinny Neural Networks are not Universal Approximators},
  year      = {2018},
  copyright = {arXiv.org perpetual, non-exclusive license},
  doi       = {10.48550/arXiv.1810.00393},
  publisher = {arXiv},
}

@Article{Kuehn2025,
  author    = {Kuehn, Christian and Kuntz, Sara-Viola},
  journal   = {Preprint},
  title     = {The Influence of the Memory Capacity of Neural {DDEs} on the Universal Approximation Property},
  year      = {2025},
  copyright = {Creative Commons Attribution 4.0 International},
  doi       = {10.48550/arXiv.2505.07244},
  publisher = {arXiv},
}

		\vspace{2mm}
		
		\begin{appendices}
			
			\section{Neural Network Normal Form} \label{app:normalform}
			
			In this appendix, we state the two necessary lemmata to prove Theorem~\ref{th:nn_fullrank_equivalent} about the MLP normal form. The first lemma treats the case that an inner weight matrix $W_{l+1} \in \R^{m_{l+1}\times n_l}$ has for some $l \in \{1,\ldots,L-1\}$ not full rank and explicitly constructs an equivalent MLP with smaller architecture.
			
			\begin{lemma} \label{lem:nn_notfullrank1_app}
				Let $\Phi \in \Xi^k_{ \mathbb{W}_0}(\mathcal{X},\mathbb{R})$ with a non-zero, non-full rank matrix $W_{l+1} \in \mathbb{R}^{m_{l+1} \times n_{l}}$ for some $l \in \{1,\ldots,L-1\}$, i.e., $0<\textup{rank}(W_{l+1}) < \min\{m_{l+1},n_{l}\}$. Then there exists an equivalent MLP $\overline{\Phi} \in \Xi^k(\mathcal{X},\mathbb{R})$ with smaller architecture. Especially only the number of nodes in layer $h_{l}$ is reduced by one, i.e., $\overline{n}_{l} = n_{l}-1\geq 1$, and $W_{l+1}$, $\widetilde{W}_{l}$ and $\tilde{b}_{l}$ are replaced by new matrices $\overline{W}_{l+1} \in \mathbb{R}^{m_{l+1} \times \overline{n}_{l}}$ and $\overline{\widetilde{W}}_{l} \in \mathbb{R}^{\overline{n}_{l} \times m_{l}}$ with $ \textup{rank}(\overline{W}_{l+1}) = \textup{rank}(W_{l+1})$, $ \textup{rank}(\overline{\widetilde{W}}_{l}) \leq \textup{rank}(\widetilde{W}_{l})$, $\overline{W}_{l+1} \overline{\widetilde{W}}_{l} = W_{l+1} \widetilde{W}_{l}$ and a new bias $\overline{\tilde{b}}_{l} \in \mathbb{R}^{\overline{n}_{l}}$.
			\end{lemma}
			
			\begin{proof}
				If $0<\textup{rank}(W_{l+1}) < \min\{m_{l+1},n_{l}\}$ for some $l\in \{1,\ldots,L-1\}$, it holds $m_{l+1},n_{l} \geq 2$. The columns of $W_{l+1} \in \mathbb{R}^{m_{l+1} \times n_{l}}$ are linearly dependent, i.e., for the columns $[W_{l+1}]_j \in \mathbb{R}^{m_{l+1}}$, $j \in \{1,\ldots,n_{l}\}$ there exist scalars $\tilde{\alpha}_j$, $j \in \{1,\ldots,n_{l}\}$ not all equal to zero, such that
				\begin{equation*}
					\sum_{j = 1}^{n_{l}} \tilde{\alpha}_j [W_{l+1}]_j = 0 \in \mathbb{R}^{m_{l+1}}.
				\end{equation*}
				Since $\tilde{\alpha}_i \neq 0$ for some $i \in \{1,\ldots,n_{l}\}$, it holds 
				\begin{equation*} 
					[W_{l+1}]_i = \sum_{j = 1, j \neq i}^{n_{l}} \alpha_j [W_{l+1}]_j
				\end{equation*}
				with $\alpha_j \coloneqq - \tilde{\alpha}_j/\tilde{\alpha}_i$. If $[W_{l+1}]_i = 0$ is a zero column, no information of the node $[h_{l}]_i$ is transferred to the next layer, which indicates that the neural network $\Phi$ is equivalent to a smaller neural network, where the node $[h_{l}]_i$ and all connected weights are removed. In the following, we prove this statement for the general case that $[W_{l+1}]_i \in \mathbb{R}^{m_{l+1}}$ is a linear combination of the other columns. We define a new neural network $\overline{\Phi} \in \Xi^k(\mathcal{X},\mathbb{R})$, which has the same structure, weights, biases and activation functions as~$\Phi$, only the number of nodes in layer $h_{l}$ is reduced by one, i.e., $\overline{n}_{l} = n_{l}-1 \geq 1$, and $W_{l+1}$, $\widetilde{W}_{l}$ and $\tilde{b}_{l}$ are replaced by new matrices $\overline{W}_{l+1} \in \mathbb{R}^{m_{l+1} \times \overline{n}_{l}}$ and $\overline{\widetilde{W}}_{l} \in \mathbb{R}^{\overline{n}_{l} \times m_{l}}$ and a new bias $\overline{\tilde{b}}_{l} \in \mathbb{R}^{\overline{n}_{l}}$. We define the matrix
				\begin{equation*}
					\overline{W}_{l+1} = \begin{pmatrix}
						[W_{l+1}]_{1} &\cdots& [W_{l+1}]_{i-1} & [W_{l+1}]_{i+1} \cdots &[W_{l+1}]_{ n_{l}} 
					\end{pmatrix}
				\end{equation*}
				which has the same rank as $W_{l+1}$, as $W_{l+1}$ arises from $\overline{W}_{l+1}$ by adding the new column $[W_{l+1}]_i$, which is a linear combination of the other columns and hence does not increase the maximal number of linearly independent columns. Furthermore we define		
				\begin{equation*}
					\overline{\widetilde{W}}_{l} = \begin{pmatrix}
						[\widetilde{W}_{l}]_{1,1} + \alpha_1 [\widetilde{W}_{l}]_{i,1}   &\cdots &[\widetilde{W}_{l}]_{1,m_{l}}  + \alpha_1 [\widetilde{W}_{l}]_{i,m_{l}} \\
						\vdots   && \vdots \\
						[\widetilde{W}_{l}]_{i-1,1} + \alpha_{i-1} [\widetilde{W}_{l}]_{i,1} &\cdots &[\widetilde{W}_{l}]_{i-1,m_{l}}  + \alpha_{i-1} [\widetilde{W}_{l}]_{i,m_{l}}\\
						[\widetilde{W}_{l}]_{i+1,1} + \alpha_{i+1} [\widetilde{W}_{l}]_{i,1}  &\cdots &[\widetilde{W}_{l}]_{i+1,m_{l}}  + \alpha_{i+1} [\widetilde{W}_{l}]_{i,m_{l}} \\
						\vdots  && \vdots \\
						[\widetilde{W}_{l}]_{n_{l},1}+ \alpha_{n_{l}} [\widetilde{W}_{l}]_{i,1}  &\cdots &[\widetilde{W}_{l}]_{n_{l},m_{l}}  + \alpha_{n_{l}} [\widetilde{W}_{l}]_{i,m_{l}}
					\end{pmatrix} 
				\end{equation*}
				which has rank smaller or equal to $\textup{rank}(\widetilde{W}_{l})$. To see this, we first apply linear row operations to the matrix $\widetilde{W}_{l}$ by adding to the $j$-th row, $j \in \{1,\ldots,i-1,i+1,\ldots,n_{l}\}$, the $i$-th row multiplied with $\alpha_j$ each. The resulting matrix has the same rank as $\widetilde{W}_{l}$. As we obtain now the matrix 	$\overline{\widetilde{W}}_{l}$ by removing the $i$-th row, it holds $ \textup{rank}(\overline{\widetilde{W}}_{l}) \leq \textup{rank}(\widetilde{W}_{l})$. Finally, the new bias is given by
				\begin{equation*}
					\overline{\tilde{b}}_{l} = \begin{pmatrix}
						[\tilde{b}_{l}]_1 + \alpha_1 [\tilde{b}_{l}]_i \\
						\vdots \\
						[\tilde{b}_{l}]_{i-1} + \alpha_{i-1} [\tilde{b}_{l}]_i \\
						[\tilde{b}_{l}]_{i+1} + \alpha_{i+1} [\tilde{b}_{l}]_i \\
						\vdots \\
						[\tilde{b}_{l}]_{n_{l}} + \alpha_{n_{l}} [\tilde{b}_{l}]_i
					\end{pmatrix}.
				\end{equation*}
				As the neural network architectures $\Phi$ and $\overline{\Phi}$ agree in all weights, biases and activation functions until $\sigma_{l}$, it holds
				\begin{equation*}
					\overline{\sigma}_{l}(\overline{a}_{l}) = \overline{\sigma}_{l}(\overline{W}_{l}\overline{h}_{l-1}+\overline{b}_{l}) = \sigma_{l}(W_{l}h_{l-1}+b_{l}) = \sigma_{l}(a_{l}), \qquad \textup{for all } x \in \mathcal{X}.  
				\end{equation*}
				In the following, we aim to show that
				\begin{equation*}
					\overline{W}_{l+1} \overline{h}_{l} = \overline{W}_{l+1} \left[\overline{\widetilde{W}}_{l}\overline{\sigma}_{l}(\overline{a}_{l}) + \overline{\tilde{b}}_{l}\right] = W_{l+1} \left[\widetilde{W}_{l}\sigma_{l}(a_{l}) + \tilde{b}_{l}\right] = W_{l+1} h_{l},
				\end{equation*}
				which implies with the fact, that the neural network architectures $\Phi$ and $\overline{\Phi}$ agree in the bias $b_{l+1}$ and in all weights, biases and activation functions from $\sigma_{l+1}$ on-wards, that $\Phi$ and $\overline{\Phi}$ are equivalent. We calculate
				\begin{align*}
					\overline{W}_{l+1} \overline{h}_{l} &= \overline{W}_{l+1} \left[\overline{\widetilde{W}}_{l}\overline{\sigma}_{l}(\overline{a}_{l}) + \overline{\tilde{b}}_{l}\right]  = \overline{W}_{l+1} \left[\overline{\widetilde{W}}_{l}\sigma_{l}(a_{l}) + \overline{\tilde{b}}_{l}\right] \\ 
					& = \overline{W}_{l+1} \overline{\widetilde{W}}_{l}\sigma_{l}(a_{l}) + \overline{W}_{l+1} \overline{\tilde{b}}_{l}\\
					& = W_{l+1} \widetilde{W}_{l} \sigma_{l}(a_{l}) + W_{l+1}  \tilde{b}_{l} = W_{l+1}  \left[\widetilde{W}_{l}\sigma_{l}(a_{l}) + \tilde{b}_{l}\right] = W_{l+1}  h_{l}
				\end{align*} 
				where we used in the last line that $\overline{W}_{l+1} \overline{\widetilde{W}}_{l} = W_{l+1}  \widetilde{W}_{l}$ and $\overline{W}_{l+1}  \overline{\tilde{b}}_{l} =  W_{l+1}  \tilde{b}_{l} $ since
				\begin{align*}
					\left[\overline{W}_{l+1} \overline{\widetilde{W}}_{l}\right]_{j,k} &= \sum_{p = 1, p \neq i}^{n_{l-1}} [W_{l+1} ]_{j,p} \left[[\widetilde{W}_{l}]_{p,k} + \alpha_p [\widetilde{W}_{l}]_{i,k} \right]= \sum_{p = 1}^{n_{l-1}} [W_{l+1} ]_{j,p}[\widetilde{W}_{l}]_{p,k} = 	\left[W_{l+1}  \widetilde{W}_{l}\right]_{j,k} \\
					\left[\overline{W}_{l+1} \overline{\tilde{b}}_{l}\right]_{j} &= \sum_{p = 1, p \neq i}^{n_{l-1}} [W_{l+1} ]_{j,p} \left[ [\tilde{b}_{l}]_{p} + \alpha_p [\tilde{b}_{l}]_{i} \right] = \sum_{p = 1}^{n_{l}} [W_{l+1} ]_{j,p}[\tilde{b}_{l}]_{p} = 	\left[W_{l+1}  \tilde{b}_{l} \right]_{j}
				\end{align*}
				by the assumption on the linear dependence of the columns of $W_{l+1} $. 
			\end{proof}
			
			The second lemma treats the other case, that an outer weight matrix $\widetilde{W}_{l} \in \R^{m_{l}\times n_{l-1}}$ has for some $l \in \{1,\ldots,L-1\}$ not full rank and shows analogously, how to construct an equivalent MLP with smaller architecture. Combining the two lemmata allows us to construct equivalent MLPs until all weight matrices have full rank and the normal form is reached.
			
			\begin{lemma} \label{lem:nn_notfullrank2_app}
				Let $\Phi \in \Xi^k_{ \mathbb{W}_0}(\mathcal{X},\mathbb{R})$ with a non-zero, non-full rank matrix $\widetilde{W}_l \in \mathbb{R}^{n_l \times m_l}$ for some $l \in \{1,\ldots,L-1\}$, i.e., $0<\textup{rank}(\widetilde{W}_l) < \min\{m_l,n_{l}\}$. Then there exists an equivalent MLP $\overline{\Phi} \in \Xi^k(\mathcal{X},\mathbb{R})$ with smaller architecture. Especially only the number of nodes in layer $h_l$ is reduced by one, i.e., $\overline{n}_{l} = n_{l}-1\geq 1$, and $\widetilde{W}_{l}$, $W_{l+1}$ and $\tilde{b}_{l}$ are replaced by new matrices  $\overline{\widetilde{W}}_{l} \in \mathbb{R}^{\overline{n}_{l} \times m_{l}}$ and $\overline{W}_{l+1} \in \mathbb{R}^{m_{l+1} \times \overline{n}_{l}}$ with $ \textup{rank}(\overline{\widetilde{W}}_l) = \textup{rank}(\widetilde{W}_l)$, $ \textup{rank}(\overline{W}_{l+1}) \leq \textup{rank}(W_{l+1})$, $\overline{W}_{l+1}\overline{\widetilde{W}}_{l} = W_{l+1} \widetilde{W}_l$ and a new bias $\overline{\tilde{b}}_{l} \in \mathbb{R}^{\overline{n}_{l}}$. If $W_{l+1} \in \mathbb{R}^{m_{l+1}\times n_l}$ has full rank, i.e., $\textup{rank}(W_{l+1}) = \min\{m_{l+1},n_l\}$, then $\overline{W}_{l+1}$ has also full rank, i.e., $ \textup{rank}(\overline{W}_{l+1}) = \min\{m_{l+1},\overline{n}_l\}$.
			\end{lemma}
			
			\begin{proof}
				If $0<\textup{rank}(\widetilde{W}_l) < \min\{m_l,n_{l}\}$ for some $l\in \{1,\ldots,L-1\}$, it holds $m_l,n_l \geq 2$. The rows of $\widetilde{W}_l \in \mathbb{R}^{n_l \times m_l}$ are linearly dependent, i.e., for the rows $[\widetilde{W}_l^\top]_j \in \mathbb{R}^{m_l}$, $j \in \{1,\ldots,n_{l}\}$ there exist scalars $\tilde{\alpha}_j$, $j \in \{1,\ldots,n_{l}\}$ not all equal to zero, such that
				\begin{equation*}
					\sum_{j = 1}^{n_{l}} \tilde{\alpha}_j [\widetilde{W}_l^\top]_j = 0 \in \mathbb{R}^{m_l}.
				\end{equation*}
				Since $\tilde{\alpha}_i \neq 0$ for some $i \in \{1,\ldots,n_{l}\}$, it holds 
				\begin{equation*}
					[\widetilde{W}_l^\top]_i = \sum_{j = 1, j \neq i}^{n_{l}} \alpha_j [\widetilde{W}_l^\top]_j
				\end{equation*}
				with $\alpha_j \coloneqq - \tilde{\alpha}_j/\tilde{\alpha}_i$. If $[\widetilde{W}_l^\top]_i = 0$ is a zero row, no information of the activated layer $\sigma_l(a_{l})$ is transferred to the node $[h_l]_i$ of the next layer, which indicates that the neural network $\Phi$ is equivalent to a smaller neural network, where the node $[h_{l}]_i$ and all connected weights are removed. In the following, we prove this statement for the general case that $[\widetilde{W}_l^\top]_i \in \mathbb{R}^{m_l}$ is a linear combination of the other rows. We define a new neural network $\overline{\Phi} \in \Xi^k(\mathcal{X},\mathbb{R})$, which has the same structure, weights, biases and activation functions as~$\Phi$, only the number of nodes in layer $h_l$ is reduced by one, i.e., $\overline{n}_{l} = n_{l}-1 \geq 1$, and $\widetilde{W}_l$, $W_{l+1}$ and $\tilde{b}_{l}$ are replaced by new matrices $\overline{\widetilde{W}}_l \in \mathbb{R}^{ \overline{n}_{l}\times m_l}$ and $\overline{W}_{l+1} \in \mathbb{R}^{m_{l+1}\times\overline{n}_{l}}$ and a new bias $\overline{\tilde{b}}_{l} \in \mathbb{R}^{\overline{n}_{l}}$. We define the matrix
				\begin{equation*}
					\overline{\widetilde{W}}_l = \begin{pmatrix}
						[\widetilde{W}_l^\top]_{1}^\top \\ 
						\vdots \\
						[\widetilde{W}_l^\top]_{i-1}^\top\\
						[\widetilde{W}_l^\top]_{i+1}^\top\\
						\vdots \\
						[\widetilde{W}_l^\top]^\top_{ n_{l}} 
					\end{pmatrix}
				\end{equation*}
				which has the same rank as $\widetilde{W}_l$, as $\widetilde{W}_l$ arises from $\overline{\widetilde{W}}_l$ by adding the new row $[\widetilde{W}_l^\top]^\top_i$, which is a linear combination of the other rows and hence does not increase the maximal number of linearly independent rows. Furthermore we define		
				\begin{align*}
					\overline{W}_{l+1} = &\left(
					[W_{l+1}]_{1} + \alpha_1 [W_{l+1}]_i, \ldots, [W_{l+1}]_{i-1}+ \alpha_{i-1} [W_{l+1}]_i, \right. \\ &\left.[W_{l+1}]_{i+1}+ \alpha_{i+1} [W_{l+1}]_i, \ldots, [W_{l+1}]_{ n_{l}} + \alpha_{n_l} [W_{l+1}]_{i} \right)
				\end{align*}
				which has rank smaller or equal to $\textup{rank}(W_{l+1})$. To see this, we first apply linear column operations to the matrix $W_{l+1}$ by adding to the $j$-th column, $j \in \{1,\ldots,i-1,i+1,\ldots,n_{l}\}$, the $i$-th column multiplied with $\alpha_j$ each. The resulting matrix $Y_{l+1}$ has the same rank as $W_{l+1}$ and  $\overline{W}_{l+1}$ is obtained  from  $Y_{l+1}$ by removing the $i$-th column, hence $ \textup{rank}(\overline{W}_{l+1}) \leq \textup{rank}(Y_{l+1})= \textup{rank}(W_{l+1})$. If $W_{l+1}$ has full rank, i.e., $\textup{rank}(W_{l+1}) = \min\{m_{l+1},n_l\}$, then also $\overline{W}_{l+1}$ has full rank, i.e., $ \textup{rank}(\overline{W}_{l+1}) = \min\{m_{l+1},\overline{n}_l\}$.
				To see this, we distinguish two cases: if $m_{l+1}<n_l$ it holds $\textup{rank}(W_{l+1}) = m_{l+1}$, such that each $m_{l+1}$ columns of $W_{l+1}$ and hence also of $Y_{l+1}$ are linearly independent. Consequently removing one of the $n_l$ columns of $Y_{l+1}$ does not change the rank, such that $ \textup{rank}(\overline{W}_{l+1}) = m_{l+1} = \min\{m_{l+1},\overline{n}_l\}$. If $m_{l+1}\geq n_l$ it holds $\textup{rank}(W_{l+1}) = n_l$, such that all $n_l$ columns of  $W_{l+1}$ and hence also all $\bar{n}_l$ columns  of $Y_{l+1}$ are linearly independent. As $Y_{l+1}$ has  $n_l$ columns, removing one leads to a matrix with $\overline{n}_l = n_l-1$ linearly independent columns. Consequently it holds   $ \textup{rank}(\overline{W}_{l+1}) =  \overline{n}_l =  \min\{m_{l+1},\overline{n}_l\}$. Finally, the new bias is given by
				\begin{equation*}
					\overline{\tilde{b}}_{l} = y_l + c_l \coloneqq \begin{pmatrix}
						[\tilde{b}_{l}]_1\\
						\vdots \\
						[\tilde{b}_{l}]_{i-1}\\
						[\tilde{b}_{l}]_{i+1}\\
						\vdots \\
						[\tilde{b}_{l}]_{n_{l}}
					\end{pmatrix} + 
					\begin{pmatrix}
						[c_{l}]_1\\
						\vdots \\
						[c_{l}]_{i-1}\\
						[c_{l}]_{i+1}\\
						\vdots \\
						[c_{l}]_{n_{l}}
					\end{pmatrix},
				\end{equation*}
				where $c_l \in \mathbb{R}^{\bar{n}_l}$ is an arbitrary solution of the linear system
				\begin{equation*}
					\overline{W}_{l+1} c_l = W_{l+1} \tilde{b}_l - \overline{W}_{l+1}  y_l \eqqcolon d_l \in \mathbb{R}^{m_{l+1}} .
				\end{equation*}
				This system has, by the Rouché-Capelli Theorem (cf.\ \cite{Shafarevich2013}), always at least one solution  as the rank of the matrix $\overline{W}_{l+1}$ is the same as the rank of the extended coefficient matrix 
				\begin{equation*}
					\textup{rank}(\overline{W}_{l+1} ) = \textup{rank}(\overline{W}_{l+1} \big\vert d_l).
				\end{equation*}
				This holds as every column of $W_{l+1}$ is a linear combination of columns of $\overline{W}_{l+1}$, so also the vector $d_l$ is a linear combination of the columns of $\overline{W}_{l+1}$. Consequently the maximal number of linearly independent columns does not increase by attaching the vector $d_l$ to the matrix $\overline{W}_{l+1}$, such that the rank of $\overline{W}_{l+1}$ and  $(\overline{W}_{l+1} \big\vert d_l)$ is the same. As the neural network architectures $\Phi$ and $\overline{\Phi}$ agree in all weights, biases and activation functions until $\sigma_{l}$, it holds
				\begin{equation*}
					\overline{\sigma}_{l}(\overline{a}_{l}) = \overline{\sigma}_{l}(\overline{W}_{l}\overline{h}_{l-1}+\overline{b}_{l}) = \sigma_{l}(W_{l}h_{l-1}+b_{l}) = \sigma_{l}(a_{l}), \qquad \textup{for all } x \in \mathcal{X}.  
				\end{equation*}
				In the following, we aim to show that
				\begin{equation*}
					\overline{W}_{l+1} \overline{h}_{l} = \overline{W}_{l+1} \left[\overline{\widetilde{W}}_{l}\overline{\sigma}_{l}(\overline{a}_{l}) + \overline{\tilde{b}}_{l}\right] = W_{l+1} \left[\widetilde{W}_{l}\sigma_{l}(a_{l}) + \tilde{b}_{l}\right] = W_{l+1} h_{l},
				\end{equation*}
				which implies with the fact, that the neural network architectures $\Phi$ and $\overline{\Phi}$ agree in the bias $b_{l+1}$ and in all weights, biases and activation functions from $\sigma_{l+1}$ on-wards, that $\Phi$ and $\overline{\Phi}$ are equivalent. We calculate
				\begin{align*}
					\overline{W}_{l+1} \overline{h}_{l} &= \overline{W}_{l+1} \left[\overline{\widetilde{W}}_{l}\overline{\sigma}_{l}(\overline{a}_{l}) + \overline{\tilde{b}}_{l}\right]  
					= \overline{W}_{l+1} \left[\overline{\widetilde{W}}_{l}\sigma_{l}(a_{l}) + y_l+c_l \right] \\ 
					& = \overline{W}_{l+1}\overline{\widetilde{W}}_{l}\sigma_{l}(a_{l}) +  \overline{W}_{l+1} y_l + \overline{W}_{l+1}c_l\\
					&= \overline{W}_{l+1}\overline{\widetilde{W}}_{l}\sigma_{l}(a_{l}) + W_{l+1} \tilde{b}_l \\
					&=  W_{l+1} \widetilde{W}_l \sigma_l(a_l) + W_{l+1} \tilde{b}_l = W_{l+1} \left[\widetilde{W}_{l}\sigma_{l}(a_{l}) + \tilde{b}_{l}\right] = W_{l+1} h_{l},
				\end{align*} 
				where we used in the last line that $\overline{W}_{l+1}\overline{\widetilde{W}}_{l} = W_{l+1} \widetilde{W}_l$ since
				\begin{equation*}
					\left[\overline{W}_{l+1}\overline{\widetilde{W}}_{l}\right]_{j,k} = \sum_{p = 1, p \neq i}^{n_l} \left([W_{l+1}]_{j,p} + \alpha_p [W_{l+1}]_{j,i} \right) [\widetilde{W}_l]_{p,k} = \sum_{p = 1}^{n_l} [W_{l+1}]_{j,p}[\widetilde{W}_l]_{p,k} =\left[ W_{l+1} \widetilde{W}_l\right]_{j,k}
				\end{equation*}
				by the assumption on linear dependence of the rows of $\widetilde{W}_l$. 
			\end{proof}
			
			\section{Regular Critical Points of Augmented MLPs} \label{app:regularity}
			
			In this appendix, we prove Theorem~\ref{th:nn_morse}, that augmented MLPs are for all weights, except possibly for a zero set with respect to the Lebesgue measure, Morse functions. The proof includes the calculation of all mixed partial second derivatives with respect to the input $x$ and the weights and biases to apply Lemma~\ref{lem:nn_morsefunction}. The augmented structure of the MLP is used to show the surjectivity of the matrix of mixed partial second derivatives. After the proof of the upcoming theorem, we explain in Remark~\ref{rem:nn_morsetheorem} in which case the proof also applies to MLPs with a bottleneck.
			
			\begin{theorem} \label{th:nn_morse_app}
				Any augmented MLP $\Phi \in \Xi^k_{\textup{A},\mathbb{W}^\ast}(\mathcal{X},\mathbb{R})$, $k \geq 2$, with weight matrices $W\in \mathbb{W}^\ast$ and biases $b \in \mathbb{B}$, is for all weights $(\vem(W);\stack(B)) \in \mathbb{R}^{N_W+N_B} $, except possibly for a zero set in $\mathbb{R}^{N_W+N_B}$ with respect to the Lebesgue measure, of class $(\mathcal{C}1)^k (\mathcal{X},\mathbb{R})$ or $(\mathcal{C}2)^k (\mathcal{X},\mathbb{R})$ and hence a Morse function. The same statement holds if $\mathbb{W}^\ast$ is replaced by $\mathbb{W}$.
			\end{theorem}
			
			\begin{proof}
				Let $\Phi \in \Xi^k_{\textup{A},\mathbb{W}^\ast}(\mathcal{X},\mathbb{R})$, $k \geq 2$ be an augmented MLP with weight matrices $(W_1, \widetilde{W}_1, \ldots, W_L, \widetilde{W}_L)\in \mathbb{W}^\ast$. By Theorem~\ref{th:nn_criticalpoints}, it is possible that the neural network~$\Phi$ has a critical point. In the case that $\Phi$ has no critical point, $\Phi$ is of class  $(\mathcal{C}1)^k (\mathcal{X},\mathbb{R})$ and hence a Morse function. In the following we show that if $\Phi$ has a critical point, then it is for all weights, where $\Phi$ has critical points, except possibly for a zero set with respect to the Lebesgue measure, of class $(\mathcal{C}2)^k(\mathcal{X},\mathbb{R})$. To show this statement, we consider MLPs with full rank weight matrices  $W =  (W_1, \widetilde{W}_1, \ldots, W_L, \widetilde{W}_L)\in \mathbb{W}^\ast$. By Lemma~\ref{lem:nn_gradient}, the gradient of $\Phi$ is given by
				\begin{equation*}
					\nabla \Phi(x) = \left[\widetilde{W}_L\Psi_L(a_{L})W_{L} \ldots \widetilde{W}_1\Psi_1(a_{1})W_1\right]^\top \in \mathbb{R}^{n},
				\end{equation*}
				where $\Psi_l(a_{l}) = \textup{diag}([\sigma_l]_i'([a_l]_i)) \in \mathbb{R}^{n_{l}\times n_{l}}$, $i \in \{1,\ldots,n_l\}$, is a diagonal matrix with $a_l = W_lh_{l-1}+b_l$, $W \in \mathbb{W}^\ast$  and $b\in \mathbb{B}$. Let $v=(\vem(W);\stack(b)) \in \mathbb{R}^p$ with $p = N_W+N_B$ be the vector with all stacked weight matrices and biases and define $\widehat{\Phi}\in C^{k}(\mathcal{X} \times \left(\vem(\mathbb{W}^\ast) \times \stack(\mathbb{B})\right))$, $k \geq 2$ to be the neural network $\Phi$ with an explicit dependence on the weight vector $v \in \vem(\mathbb{W}^\ast) \times \stack(\mathbb{B})$. As a composition of $k$ times continuously differentiable functions $\widehat{\Phi}$ is not only in $x$ but also in $v$ $k$ times continuously differentiable. The second partial derivatives $\frac{\partial^2 \widehat{\Phi}}{\partial v \partial x} \in \mathbb{R}^{n \times p} $ of $\widehat{\Phi}$ with respect to $x$ and $v$ can be explicitly calculated. First, we calculate the derivatives with respect to the components of the biases $[b_l]_i$, $l \in \{1,\ldots,L\}$, $i \in \{1,\ldots,m_l\}$. They are given by
				\begin{align*}
					\frac{\partial^2 \widehat{\Phi}}{\partial [b_l]_i \partial x} 
					= \; & \frac{\partial }{\partial  [b_l]_i}  \left[W_1^\top\Psi_1(a_{1})\widetilde{W}_1^\top \ldots W_L^\top\Psi_L(a_{L})\widetilde{W}_L^\top\right] \\
					= \; & \sum_{k = l}^L W_1^\top \Psi_1(a_1) \ldots   W_k^\top \Psi_k'(a_k) \textup{ diag}\left(  \left[\frac{\partial a_k}{\partial [b_l]_i }\right] \right)\widetilde{W}_k^\top W_{k+1}^\top \ldots\Psi_L(a_L) \widetilde{W}_L^\top\\
					= \; & \sum_{k = l}^L W_1^\top  \ldots \Psi_k'(a_k) \textup{ diag}\left( \left[ W_k \widetilde{W}_{k-1} \Psi_{k-1}(a_{k-1})  \ldots \Psi_l(a_l) \mathrm{e}_i^{(m_l)}\right]\right) \widetilde{W}_k^\top  \ldots  \widetilde{W}_L^\top,
				\end{align*}
				where $\Psi_k'(a_{k}) = \textup{diag}([\sigma_k]_i''([a_k]_i)) \in \mathbb{R}^{m_{k}\times m_{k}}$, $i \in \{1,\ldots,m_k\}$ and $\frac{\partial a_k}{\partial [b_l]_i} \in \mathbb{R}^{m_k}$, $k \in \{l,\ldots,L\}$. The operator diag applied to a vector $y \in \mathbb{R}^{m_k}$ defines a diagonal matrix in $\mathbb{R}^{m_{k}\times m_{k}}$ with the entries $y$ component-wise on its diagonal. The vector $\mathrm{e}_i^{(m_l)}$ denotes the $i$-th unit vector in $\mathbb{R}^{m_l}$. In the calculation, the product rule and the multidimensional chain rule were used. Analogously, it follows for the biases 
				$[\tilde{b}_l]_i$, $l \in \{1,\ldots,L-1\}$, $i \in \{1,\ldots,n_l\}$ that
				\begin{align*}
					\frac{\partial^2 \widehat{\Phi}}{\partial [\tilde{b}_l]_i \partial x} 
					= \; & \frac{\partial }{\partial  [\tilde{b}_l]_i}  \left[W_1^\top\Psi_1(a_{1})\widetilde{W}_1^\top \ldots W_L^\top\Psi_L(a_{L})\widetilde{W}_L^\top\right] \\
					= \; & \sum_{k = l+1}^{L} W_1^\top \Psi_1(a_1) \ldots   W_k^\top \Psi_k'(a_k) \textup{ diag}\left(  \left[\frac{\partial a_k}{\partial [\tilde{b}_l]_i }\right] \right)\widetilde{W}_k^\top W_{k+1}^\top \ldots\Psi_L(a_L) \widetilde{W}_L^\top\\
					= \; & \sum_{k = l+1}^{L} W_1^\top  \ldots \Psi_k'(a_k) \textup{ diag}\left( \left[ W_k \widetilde{W}_{k-1} \Psi_{k-1}(a_{k-1})  \ldots W_{l+1} \mathrm{e}_i^{(n_l)}\right]\right) \widetilde{W}_k^\top  \ldots  \widetilde{W}_L^\top
				\end{align*}
				and
				\begin{equation*}
					\frac{\partial^2 \widehat{\Phi}}{\partial \tilde{b}_L \partial x} = \frac{\partial }{\partial  \tilde{b}_L } \left[W_1^\top\Psi_1(a_{1})\widetilde{W}_1^\top \ldots W_L^\top\Psi_L(a_{L})\widetilde{W}_L^\top\right] = 0
				\end{equation*}
				where $\tilde{b}_L \in \mathbb{R}$.
				
				The derivatives with respect to the matrix entries $[W_l]_{ij}$, $l \in \{1,\ldots,L\}$, $i\in\{1,\ldots,m_l\}$, $j \in \{1,\ldots,n_{l-1}\}$ are given by
				\begin{align*}
					\frac{\partial^2 \widehat{\Phi}}{\partial [W_l]_{ij} \partial x} = \;& \frac{\partial }{\partial  [W_l]_{ij}}  \left[W_1^\top\Psi_1(a_{1})\widetilde{W}_1^\top \ldots W_L^\top\Psi_L(a_{L})\widetilde{W}_L^\top\right] \\
					= \; & \sum_{k = l}^L W_1^\top \Psi_1(a_1) \ldots   W_k^\top \Psi_k'(a_k) \textup{ diag}\left(  \left[\frac{\partial a_k}{\partial  [W_l]_{ij} }\right] \right)\widetilde{W}_k^\top W_{k+1}^\top \ldots\Psi_L(a_L) \widetilde{W}_L^\top\\
					& +W_1^\top \Psi_1(a_1) \ldots \Psi_{l-1}(a_{l-1})\widetilde{W}_{l-1}^\top\delta_{ij}^{n_{l-1} \times m_l} \Psi_{l}(a_{l}) \ldots  \Psi_L(a_L)\widetilde{W}_L^\top \\
					= \; & \sum_{k = l}^L W_1^\top  \ldots \Psi_k'(a_k) \textup{ diag}\left( \left[ W_k \widetilde{W}_{k-1} \Psi_{k-1}(a_{k-1})  \ldots \Psi_l(a_l) \mathrm{e}_i^{(m_l)} [h_l]_j\right]\right) \widetilde{W}_k^\top  \ldots  \widetilde{W}_L^\top \\
					& +W_1^\top \Psi_1(a_1) \ldots 	\Psi_{l-1}(a_{l-1})\widetilde{W}_{l-1}^\top\delta_{ij}^{n_{l-1} \times m_l} \Psi_{l}(a_{l}) \ldots  \Psi_L(a_L)\widetilde{W}_L^\top
				\end{align*}
				where $\delta_{ij}^{n_{l-1} \times m_l}$ denotes the matrix in $\mathbb{R}^{ n_{l-1} \times m_l} $, which has everywhere zeros, only at the $i,j$-th entry the number one. Analogously, it follows for the weight matrices 
				$[\widetilde{W}_l]_{ij}$, $l \in \{1,\ldots,L-1\}$, $i\in\{1,\ldots,n_l\}$, $j \in \{1,\ldots,m_{l}\}$ that
				\begin{align*}
					\frac{\partial^2 \widehat{\Phi}}{\partial [\widetilde{W}_l]_{ij} \partial x} 
					= \; & \frac{\partial }{\partial  [\widetilde{W}_l]_{ij}}  \left[W_1^\top\Psi_1(a_{1})\widetilde{W}_1^\top \ldots W_L^\top\Psi_L(a_{L})\widetilde{W}_L^\top\right] \\
					= \; & \sum_{k = l+1}^{L} W_1^\top \Psi_1(a_1) \ldots   W_k^\top \Psi_k'(a_k) \textup{ diag}\left(  \left[\frac{\partial a_k}{\partial [\widetilde{W}_l]_{ij} }\right] \right)\widetilde{W}_k^\top W_{k+1}^\top \ldots\Psi_L(a_L) \widetilde{W}_L^\top\\
					& +W_1^\top \Psi_1(a_1) \ldots \Psi_{l}(a_{l})\delta_{ij}^{m_l \times n_l}W_{l+1}^\top \Psi_{l+1}(a_{l+1}) \ldots  \Psi_L(a_L)\widetilde{W}_L^\top \\
					= \; & \sum_{k = l+1}^{L} W_1^\top  \ldots \Psi_k'(a_k) \textup{ diag}\left( \left[ W_k \widetilde{W}_{k-1} \Psi_{k-1}(a_{k-1})  \ldots W_{l+1} \mathrm{e}_i^{(n_l)}[\sigma_l(a_l)]_j\right]\right) \widetilde{W}_k^\top  \ldots  \widetilde{W}_L^\top \\
					& +W_1^\top \Psi_1(a_1) \ldots \Psi_{l}(a_{l})\delta_{ij}^{m_l \times n_l}W_{l+1}^\top \Psi_{l+1}(a_{l+1}) \ldots  \Psi_L(a_L)\widetilde{W}_L^\top 
				\end{align*}
				and
				\begin{equation*}
					\frac{\partial^2 \widehat{\Phi}}{\partial [\widetilde{W}_L]_{1j} \partial x} =   W_1^\top\Psi_1(a_{1})\widetilde{W}_1^\top \ldots W_L^\top\Psi_L(a_{L})\delta_{1j}^{ m_l \times 1}
				\end{equation*}
				for $j \in \{1,\ldots,m_L\}$ as $\widetilde{W}_L \in \mathbb{R}^{1 \times m_L}$. All calculated derivatives are columns of the matrix 	$\frac{\partial^2 \widehat{\Phi}}{\partial v \partial x} \in \mathbb{R}^{n \times p} $. To apply Lemma~\ref{lem:nn_morsefunction} we show that the matrix $\frac{\partial^2 \widehat{\Phi}}{\partial v \partial x} \in \mathbb{R}^{n \times p} $ has at any point $(x,v) \in \mathcal{X} \times \left(\vem(\mathbb{W}^\ast) \times \stack(\mathbb{B})\right)$ full rank, i.e., rank $n$ as $n \leq p$.  As elementary column operations do not change the rank of a matrix, we can subtract for fixed $l\in \{1,\ldots,L\}$ from the columns $\frac{\partial^2 \widehat{\Phi}}{\partial [W_l]_{ij} \partial x}$, $i\in\{1,\ldots,m_l\}$, $j \in \{1,\ldots,n_{l-1}\}$, corresponding to the weight matrix $W_l$, the columns $\frac{\partial^2\widehat{\Phi}}{\partial [b_l]_i \partial x} $, $i\in\{1,\ldots,m_l\}$, corresponding to the bias $b_l$, each multiplied with  $[h_l]_j$, $j \in \{1,\ldots,n_{l-1}\}$. We obtain
				\begin{equation} \label{eq:modified1}
					\frac{\partial^2 \widehat{\Phi}}{\partial [W_l]_{ij} \partial x} - \frac{\partial^2 \widehat{\Phi}}{\partial [b_l]_i \partial x} \cdot [h_l]_j =W_1^\top \Psi_1(a_1) \ldots 	\Psi_{l-1}(a_{l-1})\widetilde{W}_{l-1}^\top\delta_{ij}^{n_{l-1} \times m_l} \Psi_{l}(a_{l}) \ldots  \Psi_L(a_L)\widetilde{W}_L^\top.
				\end{equation}
				Analogously we can without changing the rank of the matrix subtract for fixed $l\in \{1,\ldots,L\}$ from the columns $\frac{\partial^2\widehat{\Phi}}{\partial [\widetilde{W}_l]_{ij} \partial x}$, $i\in\{1,\ldots,n_l\}$, $j \in \{1,\ldots,m_l\}$, corresponding to the weight matrix $\widetilde{W}_l$, the columns $\frac{\partial^2 \widehat{\Phi}}{\partial [\tilde{b}_l]_i \partial x} $, $i\in\{1,\ldots,n_l\}$, corresponding to the bias $\tilde{b}_l$, each multiplied with every $[\sigma_l(a_l)]_j$, $j \in \{1,\ldots,m_l\}$. We obtain
				\begin{equation} \label{eq:modified2}
					\frac{\partial^2 \widehat{\Phi}}{\partial [\widetilde{W}_l]_{ij} \partial x} - \frac{\partial^2 \widehat{\Phi}}{\partial [\tilde{b}_l]_i \partial x} \cdot[\sigma_l(a_l)]_j = W_1^\top \Psi_1(a_1) \ldots \Psi_{l}(a_{l})\delta_{ij}^{m_l \times n_l}W_{l+1}^\top \Psi_{l+1}(a_{l+1}) \ldots  \Psi_L(a_L)\widetilde{W}_L^\top .
				\end{equation}
				As the neural network $\Phi$ is augmented, there exists a layer $g_{l^\ast}$ of maximal width, $l^\ast \in \{1,\ldots,2L-1\}$, which has at least $n+1$ nodes, i.e., $d_{l^\ast} \geq n+1$. If $l^\ast$ is even, the layer $g_{l^\ast}$ of maximal width corresponds to the layer $h_{l^\ast/2}$, and we consider the modified columns~\eqref{eq:modified1} for the index $l^\ast/2+1$. We obtain
				\begin{align} 
					&\;\frac{\partial^2 \widehat{\Phi}}{\partial [W_{l^\ast/2+1}]_{ij} \partial x} - \frac{\partial^2 \widehat{\Phi}}{\partial [b_{l^\ast/2+1}]_i \partial x} \cdot [h_{l^\ast/2+1}]_j \nonumber \\
					= \, &\;  W_1^\top  \ldots 	\widetilde{W}_{l^\ast/2}^\top\delta_{ij}^{n_{l^\ast/2} \times m_{l^\ast/2+1}} \Psi_{l^\ast/2+1}(a_{l^\ast/2+1}) \ldots  \widetilde{W}_L^\top  \label{eq:modified1a} \\
					\eqqcolon&\; Y  \delta_{ij}^{ n_{l^\ast/2} \times {m_{l^\ast/2+1}}} Z \nonumber \\
					=&\; Y_i \cdot [Z]_{j} \in \mathbb{R}^n, \label{eq:modified1b}
				\end{align}
				where $Y = [Y_1,\ldots,Y_{n_{l^\ast/2}}] \in \mathbb{R}^{n \times n_{l^\ast/2}}$ and $Z \in \mathbb{R}^{m_{l^\ast/2+1}}$ are implicitly defined by equation~\eqref{eq:modified1a}.  Expression~\eqref{eq:modified1b} follows from the simple structure of the matrix $\delta_{ij}^{n_{l^\ast/2} \times m_{l^\ast/2+1}}$: it holds $$Y\delta_{ij}^{n_{l^\ast/2} \times m_{l^\ast/2+1}} Z = Y\mathrm{e}_i^{(n_{l^\ast/2})}[Z]_{j} = Y_i \cdot [Z]_{j},$$ where $\mathrm{e}_i^{(n_{l^\ast/2})}$ denotes the $i$-th unit vector in $\mathbb{R}^{n_{l^\ast/2}}$ and $[Z]_{j}$ the $j$-th component of the vector $Z$. If $l^\ast$ is odd, the layer $g_{l^\ast}$ of maximal width corresponds to the layer $a_{(l^\ast+1)/2}$, and we consider the modified columns~\eqref{eq:modified1} for the index $(l^\ast+1)/2$. Analogously we obtain 
				\begin{align} 
					&\frac{\partial^2 \widehat{\Phi}}{\partial [\widetilde{W}_{(l^\ast+1)/2}]_{ij} \partial x} - \frac{\partial^2\widehat{\Phi}}{\partial [\tilde{b}_{(l^\ast+1)/2}]_i \partial x} \cdot [\sigma_{(l^\ast+1)/2}(a_{(l^\ast+1)/2})]_j \nonumber \\
					= \, &\; W_1^\top  \ldots \Psi_{(l^\ast+1)/2}(a_{(l^\ast+1)/2})\delta_{ij}^{m_{(l^\ast+1)/2}\times n_{(l^\ast+1)/2}}W_{(l^\ast+3)/2}^\top  \ldots  \widetilde{W}_L^\top \label{eq:modified2a} \\
					\eqqcolon &\;\widetilde{Y} \delta_{ij}^{m_{(l^\ast+1)/2}\times n_{(l^\ast+1)/2}} \widetilde{Z} \nonumber \\
					= &\;\widetilde{Y}_i \cdot [\widetilde{Z}]_{j} \in \mathbb{R}^n, \label{eq:modified2b}
				\end{align}
				where $\widetilde{Y} = [\widetilde{Y}_1,\ldots,\widetilde{Y}_{m_{(l^\ast+1)/2}}] \in \mathbb{R}^{n \times m_{(l^\ast+1)/2}}$ and $\widetilde{Z} \in \mathbb{R}^{n_{(l^\ast+1)/2}}$ are implicitly defined by equation~\eqref{eq:modified2a}. If $(l^\ast+1)/2 = L$, the matrix product $\widetilde{Z}$ does not exist, so we define $\widetilde{Z} = 1\in \mathbb{R}^{n_L} = \mathbb{R}$. In both cases we obtain from equations~\eqref{eq:modified1b} or~\eqref{eq:modified2b} modified columns of the form
				\begin{equation*}
					\overline{Y}_i [\overline{Z}]_j \in \mathbb{R}^n, \qquad i \in \{1,\ldots,d_{l^\ast}\}, j\in \{1,\ldots,d_{l^\ast+1}\}
				\end{equation*} 
				with $\overline{Y} = [\overline{Y}_1,\ldots,\overline{Y}_{d_{l^\ast}}]\in \mathbb{R}^{n \times d_{l^\ast}}$ and $\overline{Z} \in \mathbb{R}^{d_{l^\ast+1}}$., where $d_l^\ast \geq n+1$ is the number of nodes in the layer of maximal width $g_{l^\ast}$.
				
				In the following we aim to show that for every $(x,v) \in \mathcal{X} \times \left(\vem(\mathbb{W}^\ast) \times \stack(\mathbb{B})\right)$, there exists an index $j \in \{1,\ldots,d_{l^\ast+1}\}$, such that the matrix
				\begin{equation}\label{eq:nn_matrixproof}
					Q_j \coloneqq \left[ 	\overline{Y}_1 \cdot [\overline{Z}]_{j}, \ldots, \overline{Y}_{d_{l^\ast}} \cdot [\overline{Z}]_{j} \right]  \in \mathbb{R}^{n \times d_{l^\ast}}
				\end{equation}
				has full rank $n$. As $Q_j$ is a submatrix of $\frac{\partial^2 \widehat{\Phi}}{\partial v \partial x} \in \mathbb{R}^{n \times p} $ of rank $n$, also the matrix $ \frac{\partial^2 \widehat{\Phi}}{\partial v \partial x}$ has full rank $n$  for every $(x,v) \in \mathcal{X} \times \left(\vem(\mathbb{W}^\ast) \times \stack(\mathbb{B})\right)$. Due to the assumption $(W_1, \widetilde{W}_1, \ldots, W_L, \widetilde{W}_L)\in \mathbb{W}^\ast$ all considered weight matrices as well as the diagonal matrices $\Psi_l(a_l)$, $l \in \{1,\ldots,L\}$ have full rank. Since the neural network $\Phi$ is augmented, it holds $d_i \leq d_j$ for $0 \leq i < j \leq l^\ast$ and $d_i \geq d_j$ for $l^\ast \leq i < j \leq L$. By Lemma~\ref{lem:nn_productfullrank}, also the matrix products $\overline{Y}$ and $\overline{Z}$ have full rank, i.e., $\textup{rank}(\overline{Y}) = n$ and $\textup{rank}(\overline{Z}) = 1$ (this also holds in the special case where $\overline{Z} = 1$ for $(l^\ast+1)/2 = L$). Consequently the vector $\overline{Z} \in \mathbb{R}^{d_{l^\ast+1}}$ has at least one non-zero entry, i.e., $[\overline{Z}]_{j^\ast} \neq 0$ for some $j^\ast \in \{1,\ldots,d_{l^\ast+1}\}$. As $\overline{Z}$ depends on $(x,v) \in \mathcal{X} \times \left(\vem(\mathbb{W}^\ast) \times \stack(\mathbb{B})\right)$, also the choice of $j^\ast$ depends on $x$ and $v$. For the index $ j^\ast$, the matrix $Q_{j^\ast}$ arises from the matrix $\overline{Y}$ by elementary column operations, as each column is multiplied by the non-zero scalar $[Z]_{j^\ast}$. As the matrix $\overline{Y}$ has full rank $n$, also the matrix $Q_{j^\ast}$ has full rank $n$, which implies that  $ \frac{\partial^2 \widehat{\Phi}}{\partial v \partial x}$ has full rank $n$  for every $(x,v) \in \mathcal{X} \times \left(\vem(\mathbb{W}^\ast) \times \stack(\mathbb{B})\right)$. As the map  $\widehat{\Phi}\in C^k\left(\mathcal{X} \times \left(\vem(\mathbb{W}^\ast) \times \stack(\mathbb{B})\right)\right)$, $k\geq 2$, fulfills  the assumptions of Lemma~\ref{lem:nn_morsefunction}, it follows that for all weights $v \in \mathcal{V}\coloneqq \vem(\mathbb{W}^\ast) \times \stack(\mathbb{B})$, except possibly for a zero set  $\mathcal{V}_0 \subset \mathcal{V}$ with respect to the Lebesgue measure in $\mathbb{R}^p$, the corresponding augmented MLP $\Phi \in \Xi^k_{\textup{A},\mathbb{W}^\ast}(\mathcal{X},\mathbb{R})$, $k \geq 2$ with weight vector $v$ is a Morse function. 
				
				As $\vem( \mathbb{W}_0) \times \stack(\mathbb{B}) = \mathbb{R}^p \setminus \mathcal{V}$ is by Lemma~\ref{lem:nn_weightspace} a zero set with respect to the Lebesgue measure in $\mathbb{R}^p$, also the finite union $(\mathbb{R}^p \setminus \mathcal{V}) \cup \mathcal{V}_0$ is a zero set with respect to the Lebesgue measure in $\mathbb{R}^p$. It follows that for all weights $v \in \vem(\mathbb{W}) \times \stack(\mathbb{B}) = \mathbb{R}^p$, except possibly for a zero set with respect to the Lebesgue measure in $\mathbb{R}^p$, the augmented MLP $\Phi \in \Xi^k_{\textup{A}}(\mathcal{X},\mathbb{R})$, $k \geq 2$ with weight vector $v$ is a Morse function. The Morse function $\Phi$ can be of class $(\mathcal{C}1)^k (\mathcal{X},\mathbb{R})$ or of class $(\mathcal{C}2)^k (\mathcal{X},\mathbb{R})$, as by Lemma~\ref{th:nn_criticalpoints} it is possible that the augmented network $\Phi$ can or cannot have critical points for $W \in \mathbb{W}^\ast$.
			\end{proof}
			
			\begin{remark} \label{rem:nn_morsetheorem}
				The proof of Theorem~\ref{th:nn_morse_app} relies on the fact that the matrix $Q_j$ defined in~\eqref{eq:nn_matrixproof} has full rank as $\overline{Y}$ and $\overline{Z}$ are full rank matrices. This can be guaranteed, as the considered neural network is augmented and the weight matrices are assumed to have full rank. As this is the only time that the augmented structure of $\Phi$ is used, Theorem~\ref{th:nn_morse_app} is also applicable for neural networks with a bottleneck, where  $\overline{Y}$ and $\overline{Z}$ are full rank matrices. In the analysis of MLPs with a bottleneck in Section~\ref{sec:nn_bottleneck}, we proof in Theorem~\ref{th:nn_bottleneck}\ref{th:nn_bottleneck_c}, that for MLPs, where the first bottleneck is augmented and a certain matrix product is non-zero, the matrix $Q_{j}$ defined in~\eqref{eq:nn_matrixproof} has full rank and the same statement as in Theorem~\ref{th:nn_morse_app} holds.
			\end{remark}
			
			In the following lemma, we prove a technical detail of Example \ref{ex:nn_augmented}, which visualizes the statement of Theorem \ref{th:nn_morse}/\ref{th:nn_morse_app} for a one-dimensional augmented MLP.
			
			\begin{lemma} \label{lem:app_regular}
				Assume the setting of Example \ref{ex:nn_augmented}. Then the set of weights $v \in \mathbb{R}^7$, such that $\Phi$ has one degenerate critical point $x^\ast \in \R$ with $g(v,x^\ast)$, has Lebesgue measure zero. Hence, $\Phi$ is for all sets of weights, except possibly for a set of Lebesgue measure zero, a Morse function.
			\end{lemma}
			
			\begin{proof}
				We continue with the setting and notation of Example \ref{ex:nn_augmented}. For the proof of the statement, we have to exclude the set
				\begin{equation*}
					\mathcal{W}_1 \coloneqq \{v \in \mathbb{R}^7\setminus \mathcal{W}_0: 1+\exp(-w_1h(v)-b_{11})+w_1h(v) = 0\}
				\end{equation*}
				from the weight space. $\mathcal{W}_1$ has Lebesgue measure zero in $\mathbb{R}^7$ as for $\eps>0$, it holds for the weight
				\begin{equation*}
					\hat{v}(v,\eps) = v+(0,0,0,0,\eps,\eps,0) = (w_1,w_2,\tilde{w}_1,\tilde{w}_2,b_{11}+\eps,b_{12}+\eps,\tilde{b}_1)
				\end{equation*}
				that $h(\hat{v}) = h(v)$, which implies that if $v\in\mathcal{W}_1$, then $\hat{v}\notin \mathcal{W}_1$ as $\exp(-\eps)\neq 1$, such that the constraint defining $\mathcal{W}_1$ cannot be fulfilled. As $\eps>0$ was arbitrary, $\mathcal{W}_1$ cannot have non-zero Lebesgue measure. 
				
				In the following, we show that the set of weights
				\begin{equation*}
					\mathcal{W}_2 \coloneqq \{v \in \mathbb{R}^7\setminus (\mathcal{W}_0 \cup \mathcal{W}_1): \; \exists \textup{ critical point } x^\ast \in \mathbb{R} \textup{ with } g(v,x^\ast) = 0\},
				\end{equation*}
				such that $\Phi$ has a degenerate critical point, has Lebesgue measure zero. To that purpose assume there exist weights $v \in \mathbb{R}^7\setminus (\mathcal{W}_0 \cup \mathcal{W}_1)$, such that $\Phi$ has a degenerate critical point $x^\ast$. We notice that by Theorem~\ref{th:nn_criticalpoint_fulllebesgue}, a non-degenerate critical point perturbs to a non-degenerate critical point under a small perturbation of the weights. Hence, if we aim to find degenerate critical points in a neighborhood of the considered weights $v \in \mathbb{R}^7\setminus (\mathcal{W}_0\cup \mathcal{W}_1)$, the only possibility is that the degenerate critical point $x^\ast$ perturbs to a degenerate critical point. For $\eps>0$ sufficiently small, consider the modified weights
				\begin{equation*}
					\bar{v}(v,\eps) = v + \left(\eps,0,	\delta(v,\eps) ,0,0,0,0\right)^\top = \left(w_1+\eps,w_2,\widetilde{w}_1+	\delta(v,\eps) ,\widetilde{w}_2,b_{11},b_{12},\tilde{b}_1\right)^\top \in \mathbb{R}^7,
				\end{equation*}
				where 
				\begin{equation}\label{eq:nn_ex_delta}
					\delta(v,\eps) = \widetilde{w}_1 \left(\frac{w_1}{w_1+\eps}\frac{ 1+\exp(-(w_1+\eps)x^\ast-b_{11})}{1+\exp(-a_1^\ast)}-1\right),
				\end{equation}
				such that 
				\begin{equation*}
					\frac{w_1 \widetilde{w}_1}{1+\exp(-a_1^\ast)} = \frac{(
						w_1+\eps)( \widetilde{w}_1+	\delta(v,\eps) )}{1+\exp(-(w_1+\eps)x^\ast-b_{11})}.
				\end{equation*}
				Hence, if and only if $x^\ast$ is a critical point of the network with weights $v$, then $x^\ast$ is also a critical point of the network with weights $\bar{v}$. It holds
				\begin{align*}
					\frac{\partial g(\bar{v}(v,\eps),x^\ast)}{\partial \eps} &= -\widetilde{w}_2 x^\ast\exp(-(w_1+\eps) x^\ast - b_{11}) + \frac{\partial \delta(v,\eps)}{\partial \eps} \exp(-a_2^\ast) \\
					& = -\widetilde{w}_2 x^\ast\exp(-(w_1+\eps) x^\ast - b_{11}) + \left(-\frac{\widetilde{w}_1 w_1}{(w_1+\eps)^2}\frac{ 1+\exp(-(w_1+\eps)x^\ast-b_{11})}{1+\exp(-a_1^\ast)}\right. \\
					& \hspace{4mm}\left. -\frac{\widetilde{w}_1 w_1x^\ast}{w_1+\eps}\frac{ \exp(-(w_1+\eps)x^\ast-b_{11})}{1+\exp(-a_1^\ast)}\right)\exp(-a_2^\ast),
				\end{align*}
				which yields evaluated at $\eps = 0$:
				\begin{align*}
					\frac{\partial g(\bar{v}(v,0),x^\ast)}{\partial \eps} 
					& = -\widetilde{w}_2 x^\ast\exp(-a_1^\ast) - \frac{\widetilde{w}_1 }{w_1} \exp(-a_2^\ast)- \widetilde{w}_1 x^\ast\frac{ \exp(-a_1^\ast)\exp(-a_2^\ast)}{1+\exp(-a_1^\ast)} \\
					& =-x^\ast( \widetilde{w}_2 \exp(-a_1^\ast)+\widetilde{w}_1 \exp(-a_2^\ast)) - \frac{\widetilde{w}_1 }{w_1} \exp(-a_2^\ast)+\widetilde{w}_1x^\ast\frac{ \exp(-a_2^\ast)}{1+\exp(-a_1^\ast)} \\
					& = \frac{\widetilde{w}_1\exp(-a_2^\ast)}{w_1(1+\exp(-a_1^\ast)) }\left(1+\exp(-w_1h(v)-b_{11})+w_1h(v)\right)
				\end{align*}
				where we used in the last line the property $g(v,x^\ast) = 0$, such that the first summand vanishes, and we replaced $x^\ast$ by $h(v)$ in the second summand. As $v\in\mathbb{R}^7\setminus(\mathcal{W}_0\cup\mathcal{W}_1)$, it holds 
				\begin{equation}\label{eq:nonzeroderiv}
					\frac{\partial g(\bar{v}(v,0),x^\ast)}{\partial \eps} \neq 0.
				\end{equation}
				For $\eps>0$ sufficiently small we can deduce from $g(v,x^\ast) = 0$ and~\eqref{eq:nonzeroderiv} that  $g(\bar{v}(v,\eps),x^\ast)\neq 0$, such that $x^\ast$ is a non-degenerate critical point for the modified weights $\bar{v}$. From~\eqref{eq:nn_ex_delta} it follows that for every $\mu>0$, the parameter $\eps>0$ can be chosen small enough, such that $\abs{\delta(v,\eps)} < \mu$. Hence, we found a point $\bar{v}$ in a $\mu$-neighborhood of $v\in \mathcal{W}_2$, which is not contained in $\mathcal{W}_2$. Consequently $\mathcal{W}_2$ must have zero Lebesgue measure in $\mathbb{R}^7$ and $\Phi$ is a Morse function for all weighs in $\mathbb{R}^7\setminus(\mathcal{W}_0 \cup \mathcal{W}_1\cup \mathcal{W}_2)$, where $\mathcal{W}_0 \cup \mathcal{W}_1 \cup \mathcal{W}_2$ is a set of Lebesgue measure zero in $\mathbb{R}^7$.
			\end{proof}
			
			\vspace{2mm}
			
			\section{Results from Linear Algebra} \label{sec:app_LA}
			
			The proofs of our results about multilayer perceptrons include many statements about matrix products and linear systems. This appendix collects results from basic linear algebra, which are used multiple times throughout Sections~\ref{sec:feedforward} and \ref{sec:neuralODEs}. The first lemma uses Sylvester's rank inequality to prove that a monotone product of full-rank matrices has full rank.
			
			\begin{lemma} \label{lem:nn_productfullrank}
				Let $A_1, \ldots, A_J$ be matrices $A_j \in \mathbb{R}^{k_j\times k_{j+1}}$ with full rank, i.e., for $j \in \{1,\ldots,J\}$ it holds $\textup{rank}(A_j) = \min\{k_j,k_{j+1}\}$. If the dimensions of the matrices are monotone, i.e., $k_1 \geq k_2 \geq \ldots \geq k_{J+1}$ or $k_1 \leq k_2 \leq \ldots \leq k_{J+1}$, then also the matrix product
				\begin{equation*}
					A \coloneqq A_1 A_2 \ldots A_J \in \mathbb{R}^{k_1 \times k_{J+1}}
				\end{equation*}
				has full rank, i.e., $\textup{rank}(A) = \min\{k_1,k_{J+1}\}$.
			\end{lemma}
			
			\begin{proof}
				Sylvester's rank inequality (cf.\ \cite{Thome2016}) implies for the product of two matrices $A_j \in \mathbb{R}^{k_j \times k_{j+1}}$ and $A_{j+1}  \in \mathbb{R}^{k_{j+1} \times k_{j+2}}$ that 
				\begin{equation*}
					\textup{rank}(A_j) + \textup{rank}(A_{j+1}) - k_{j+1} \leq \textup{rank}(A_jA_{j+1}) \leq \min\{k_j,k_{j+2}\}.
				\end{equation*}
				If $k_j \geq k_{j+1} \geq k_{j+2}$, this implies $\textup{rank}(A_jA_{j+1}) = k_{j+2}$ and if $k_j \leq k_{j+1} \leq k_{j+2}$, this implies $\textup{rank}(A_jA_{j+1}) = k_j$, such that in both cases the matrix product $A_jA_{j+1}$ has full rank. Inductively it follows that the matrix product $A$ has full rank, i.e., $\textup{rank}(A) = \min\{k_1,k_{J+1}\}$.
			\end{proof}
			
			The second lemma guarantees that the algorithm to derive the MLP normal form in Theorem~\ref{th:nn_fullrank_equivalent} ends, as all weight matrices become eventually non-singular.
			
			\begin{lemma} \label{lem:nn_iterations}
				Let $A_0 \in \mathbb{R}^{m \times n}$ with $\textup{rank}(A_0)>0$. For $i \geq 0$, if  $A_i \in \mathbb{R}^{m \times (n-i)}$ has not full rank, define iteratively the matrix $A_{i+1} \in  \mathbb{R}^{m \times (n-i-1)}$ by removing a column of $A_i \in \mathbb{R}^{m \times (n-i)}$ in such a way that $\textup{rank}(A_{i+1}) = \textup{rank}(A_i)$. Then there exists an index $r \geq 0$, such that $A_r \in \mathbb{R}^{m \times (n-r)}$ has full rank. The analogous statement holds by iteratively removing the rows of $A_0$.
			\end{lemma}
			
			\begin{proof}
				Let $p = \textup{rank}(A_0)<\min\{m,n\}$. By assumption it holds for $i \in \{0,\ldots,n-p-1\}$ that
				\begin{equation*}
					\textup{rank}(A_i) = \textup{rank}(A_0) = p,
				\end{equation*}
				such that $A_i \in \mathbb{R}^{m \times (n-i)}$ has not full rank as $p<\min\{m,n-i\}$. Consequently, the matrix $A_{n-p} \in \mathbb{R}^{m \times p}$ is well defined and has full rank $p<m$, i.e., the statement holds for $r = n-p$. The statement for iteratively removing the rows of $A_0$ follows with the same argumentation for the transpose $A_0^\top$.
			\end{proof}
			
			The next lemma is necessary to show in Theorem~\ref{th:nn_criticalpoints} via an explicit construction that augmented MLPs and MLPs with a bottleneck can have critical points.
			
			\begin{lemma} \label{lem:nn_linearsystem}
				Given $A \in \mathbb{R}^{1 \times n}$, $\textup{rank}(A) = 1$ and $ C \in \mathbb{R}^{1 \times m}$, $\textup{rank}(C) = 1$  with $n\leq m$, the linear system $AB = C$ always has a full rank solution $B \in \mathbb{R}^{n \times m}$, $\textup{rank}(B) = n$.
			\end{lemma}

			\begin{proof}
				As by assumption $\text{rank}(A) = \text{rank}(A\vert C) = 1$, the Rouché-Capelli Theorem \cite{Shafarevich2013} guarantees the existence of a solution to the linear system. In the following, we show via an explicit construction that the solution $B$ can be chosen as a full-rank matrix. Denote by $\overline{A} = [A]_{i_1,\ldots,i_{\bar{n}}} \in \mathbb{R}^{1\times \bar{n}}$, $\bar{n}\leq n$, $\{i_1,\ldots,i_{\bar{n}}\} \subset \{1,\ldots,n\}$ the submatrix of $A$, which consists only of the non-zero entries of $A$, i.e., $[A]_{i_k} \neq 0$ for all $k \in \{1,\ldots, \bar{n}\}$ and denote the non-zero entries of $C$ by  $ [C]_{j_1,\ldots,j_{\overline{m}}} \in \mathbb{R}^{1\times \overline{m}}$,  $\overline{m} \leq m$, i.e., $[C]_{j_k}\neq 0$ for all $k \in \{1,\ldots,\overline{m}\}$. If $\overline{m} \geq \bar{n}$, we define the following matrix entries of the submatrix $[B]_{(i_1,\ldots,i_{\bar{n}}),(1,\ldots,m)}$:
				\begin{align*}
					[B]_{i_k,j_k} &= \frac{[C]_{j_k}}{[A]_{i_k}} \neq 0, \qquad &&\text{if } k \in \{1,\ldots,\bar{n}\}, \\
					[B]_{i_{\bar{n}},j_k} &= \frac{[C]_{j_k}}{[A]_{i_{\bar{n}}}} \neq 0, \qquad &&\text{if }k \in \{\bar{n}+1,\ldots,\overline{m}\}, 
				\end{align*}
				and set all other matrix entries to zero. By construction, $[B]_{(i_1,\ldots,i_{\bar{n}}),(1,\ldots,m)}$ has rank $\bar{n}$ and it holds
				\begin{equation*}
					\overline{A} \cdot [B]_{(i_1,\ldots,i_{\bar{n}}),(1,\ldots,m)} = C.
				\end{equation*}
				If $\overline{m}<\bar{n}$, define $j_{\overline{m}+1}, \ldots, j_{\bar{n}}$ to be arbitrary, but different indices of $\{1,\ldots,m\}\setminus \{j_1,\ldots,j_{\overline{m}}\}$, which is possible as $\bar{n}\leq n \leq m$. Let the matrix entries of $[B]_{(i_1,\ldots,i_{\bar{n}}),(1,\ldots,m)}$ be
				\begin{align*}
					[B]_{i_k,j_k} &= \frac{[C]_{j_k}}{[A]_{i_k}} \neq 0, \qquad &&\text{if } k \in \{1,\ldots,\overline{m}\}, \\
					[B]_{i_k,j_k} &= A_{i_1}, \quad [B]_{i_1,j_k} = - A_{i_k}, \qquad &&\text{if }k \in \{\overline{m}+1,\ldots,\bar{n}\}, 
				\end{align*}
				and set all other matrix entries to zero. By construction, $[B]_{(i_1,\ldots,i_{\bar{n}}),(1,\ldots,m)}$ has rank $\bar{n}$ and it holds
				\begin{equation*}
					\overline{A} \cdot [B]_{(i_1,\ldots,i_{\bar{n}}),(1,\ldots,m)} = C,
				\end{equation*}
				as for $k \in \{\overline{m}+1,\ldots,\bar{n}\}$ it holds $\overline{A}\cdot [B]_{(i_1,\ldots,i_{\bar{n}}),(j_k)} = A_{i_1}A_{i_k}-A_{i_1}A_{i_k}=0 = [C]_{j_k}$.
				
				For both cases, i.e., $\overline{m}\geq \bar{n}$ and $\overline{m}<\bar{n}$, choose the other rows of $B$, which are not part of $[B]_{(i_1,\ldots,i_{\bar{n}}),(1,\ldots,m)}$ in such a way, that they are linearly independent of the rows of $[B]_{(i_1,\ldots,i_{\bar{n}}),(1,\ldots,m)}$. This is possible, as the number $n$ of rows of $B$ is smaller than its dimension $m$. Hence, the constructed matrix $B$ has rank $n$. As $[A]_i = 0$ for $i \in \{1,\ldots,n\}\setminus \{i_1,\ldots,i_{\bar{n}}\}$ it follows $A B  =C$.
			\end{proof}
			
			\section{Results from ODE Theory} \label{sec:app_ODE}

			To determine if a neural ODE has critical points, we study in Section~\ref{sec:node_criticalpoints} the network gradient of a neural ODE. To characterize if the network gradient can be zero, we need to study the Jacobian matrix $\frac{\partial h_a(T)}{\partial a}$ of the time-$T$ map $ h_a(T)$ with respect to the initial value $a$. To that purpose, we use some results from ODE theory. The considered initial value problem is given by
				\begin{equation} \label{eq:IVP2} 
					\frac{\dd h}{\dd t} = f(t,h(t)), \qquad h(0) = a \in \mathcal{A} \subset \mathbb{R}^m, \tag{IVP}
				\end{equation}
				where $f:\Omega\rightarrow \R^m$, $\Omega = \Omega_t \times \Omega_h \subset \R \times \R^m$ open, with maximal time interval of existence $\mathcal{I}_a \subset \Omega_t$ open, initial value $h(0) = a \in \mathcal{A}$ with $ \varnothing \neq \mathcal{A}\subset \Omega_h$ and solution map $h_a: \mathcal{I}_a \rightarrow \R^m$. The following lemma shows that the Jacobian matrix $\frac{\partial h_a(T)}{\partial a}$ of the time-$T$ map $h_a(T)$ satisfies a linear variational equation.
			
			\begin{lemma}[\cite{Hale1980}] \label{lem:node_variational} 
				Let $f\in C^{0,k}(\Omega_t \times \Omega_h,\mathbb{R}^m)$ with $k \geq 1$. Then the solution $h_a(t)$ of the initial value problem~\eqref{eq:IVP2} is continuously differentiable with respect to the initial condition $a \in \mathcal{A}\subset \Omega_h$ and the Jacobian matrix $\frac{\partial h_a(t)}{\partial a}\in \mathbb{R}^{m \times m}$ satisfies the linear variational equation
				\begin{equation*}
					\frac{\dd Y}{\dd t} = \frac{\partial f(t,h_a(t))}{\partial h_a}Y,
				\end{equation*}
				with $Y(t)\in \mathbb{R}^{m \times m}$ and initial condition $\frac{\partial h_a(0)}{\partial a} = \textup{Id}_m$, where $\textup{Id}_m \in \mathbb{R}^{m \times m}$ is the identity matrix. The matrix $\frac{\partial f(t,h_a(t))}{\partial h_a} \in \mathbb{R}^{m \times m}$ is the Jacobian  of the vector field $f(t,h_a(t))$ with respect to $h_a$, defined by $\left[\frac{\partial f(t,h_a(t))}{\partial h_a}\right]_{ij} = \frac{\partial [f(t,h_a(t))]_i}{\partial [h_a]_j}$.
			\end{lemma}
			
			The variational equation of Lemma~\ref{lem:node_variational} is a linear homogeneous ODE. The following lemma characterizes the determinant of matrix solutions of linear homogeneous ODEs, which is used to show in Proposition~\ref{prop:node_jacobian} that the Jacobian matrix $\frac{\partial h_a(T)}{\partial a}$ has always full rank.
			
			\begin{lemma}[\cite{Chicone2006}] \label{lem:node_linearODE}
				Consider the linear homogeneous ODE 
				\begin{equation} \label{eq:node_linearODE}
					\frac{\dd y}{\dd t} = A(t) y,
				\end{equation}
				where $y \in \mathbb{R}^m$ and $A \in C^0(\mathcal{I} ,\mathbb{R}^{m \times m})$, $\mathcal{I} \subset \mathbb{R}$ open, is continuous. 
				\begin{enumerate}[label=(\alph*), font=\normalfont]
					\item \label{lem:node_linearODE_a} Given an initial condition $y(t_0) = y_0$, $t_0 \in \mathcal{I}$, then~\eqref{eq:node_linearODE} has a unique solution  $y: \mathcal{I} \rightarrow \mathbb{R}^m$. 
					\item  \label{lem:node_linearODE_b} If $Y:\mathcal{I} \rightarrow \mathbb{R}^{m \times m}$ is a matrix solution of~\eqref{eq:node_linearODE} with initial condition $Y(t_0) = Y_0 \in \mathbb{R}^{m \times m}$, i.e., every column $[Y]_i \in\mathbb{R}^m$, $i \in \{1,\ldots,m\}$ is a solution of ~\eqref{eq:node_linearODE} with initial condition $[Y]_i(0) = [Y_0]_i \in \mathbb{R}^m$, then Liouville's formula holds:
					\begin{equation*}
						\det(Y(t)) = \det(Y(t_0))\exp\left(\int_{t_0}^t \textup{tr}(A(r)) \; \dd r\right),
					\end{equation*}
					where $\textup{tr}(A(r))$ is the trace of the matrix $A \in C^0(\mathcal{I} ,\mathbb{R}^{m \times m})$.
				\end{enumerate}
			\end{lemma}
			
		\end{appendices}

	\end{document}